\documentclass[a4paper,11pt]{amsart}%
\usepackage{amsmath}%
\usepackage{amssymb}%
\usepackage{amscd}%
\usepackage{mathrsfs}%
\usepackage[all]{xy}%
\usepackage{enumerate}%

\theoremstyle{plain}%
    \newtheorem{thm}{Theorem}[section]%
    \newtheorem{prop}[thm]{Proposition}%
    \newtheorem{lem}[thm]{Lemma}%
	\newtheorem{cor}[thm]{Corollary}%
    \newtheorem{con}[thm]{Conjecture}%
	\newtheorem{propdef}[thm]{Proposition-Definition}%
	\newtheorem{thmdef}[thm]{Theorem-Definition}%

\theoremstyle{definition}%
    \newtheorem{defn}[thm]{Definition}%
    \newtheorem{rem}[thm]{Remark}%
    \newtheorem{ex}[thm]{Example}%

\def\endpiece{xxx}%
\def\makeop[#1]{\xmakeop#1,xxx,}
\def\mkop#1{\expandafter\def\csname #1\endcsname{{\mathrm{#1}}}} %
\def\xmakeop#1,{\def\temp{#1}\ifx\temp\endpiece\else\mkop{#1}\expandafter\xmakeop\fi}%

\def\N{{\mathbb{N}}}%
\def\Z{{\mathbb{Z}}}%
\def\Q{{\mathbb{Q}}}%
\def\Zp{{\mathbb{Z}_p}}%
\def\Qp{{\mathbb{Q}_p}}%
\def\Fp{{\mathbb{F}_p}}%
\def\C{{\mathbb{C}}}
\def\id{{\mathrm{id}}}%
\def\del{{\partial}}%
\def\gal#1{{\mathrm{Gal}(#1)}}%
\def\ideal#1{{\mathfrak{#1}}}%
\def\integer#1{{\mathcal{O}_{#1}}}%
\def\iw#1{{\Lambda (#1)}}%
\def\frob#1{\varphi(#1)}%
\def\tilder#1{{\widetilde{#1}}}%
\def\comp#1{{\widehat{#1}}}%
\def\line#1{{\overline{#1}}}%
\def\x#1{\mathfrak{X}(#1)}%
\def\iw#1{\Lambda (#1)}%
\def\conj#1{{\mathrm{Conj}(#1)}}%
\def\obj#1{{\mathrm{Ob}\,#1}}%
\def\cal#1{{\mathcal{#1}}}%
\def\mod#1{{\mathrm{mod}\, #1}}%
\def\ind#1#2#3{{\mathrm{Ind}^{#1}_{#2}(#3)}}%
\makeop[Tr,Nr,Ver,Ker,Aut,det,sgn,End,ev]%
\makeop[GL,SL,E,M,K,SK]%

\begin{document}

\setcounter{tocdepth}{1} 
\numberwithin{equation}{section}


\title[Iwasawa theory for non-commutative $p$-extensions]{Iwasawa theory of totally real fields for certain non-commutative $p$-extensions}
\author{Takashi Hara}
\address{Graduate School of Mathematical Sciences, The University of
Tokyo, 8-1 Komaba 3-chome, Meguro-ku, Tokyo, 153-8914, Japan}
\email{thara@ms.u-tokyo.ac.jp}
\date{\today}

\begin{abstract}
In this paper, we prove the Iwasawa main conjecture 
(in the sense of \cite{CFKSV}) of totally real fields for certain 
specific non-commutative $p$-adic Lie extensions, using the integral logarithms 
introduced by Oliver and Taylor. Our result gives certain generalization 
of Kato's proof of the main conjecture for Galois extensions of 
Heisenberg type (\cite{Kato1}).
\end{abstract}

\maketitle
\setcounter{section}{-1}

%
%
\section{Introduction}
%
%

%
\subsection{Introduction}
%

The Iwasawa main conjecture, which describes mysterious 
relation between ``arithmetic'' characteristic elements 
and ``analytic'' $p$-adic zeta functions, has been proven
under many situations for abelian extensions. However, 
for non-abelian extensions, it took many years even to 
formulate the main conjecture. In 2005, Coates et al.\ formulated 
it (\cite{CFKSV}) by using algebraic 
$K$-theory (especially the localization exact sequence), 
and Kazuya Kato has proven it for certain 
specific $p$-adic Lie extensions of totally real fields up to the
present (\cite{Kato1}). Mahesh Kakde generalized Kato's proof and proved the main conjecture for 
other types of extensions (\cite{Kakde}). J\"urgen Ritter and Alfred Weiss also formulated the main conjecture  in a little different way (``Equivariant Iwasawa theory,'' 
see \cite{R-W1}.), also using algebraic $K$-theory.

In this paper, we prove the Iwasawa main conjecture 
(in the sense of \cite{CFKSV}) of totally real fields for certain 
non-commutative $p$-extensions, using the method 
of Kato in \cite{Kato1}. 

Let $F$ be a totally real number field, and let $p$ be a prime number. 
Let $F^{\infty}/F$ be a totally real Lie extension containing
 the cyclotomic $\Zp$-extension $F^{\mathrm{cyc}}$ of $F$. We assume that only finitely many primes of $F$ ramify in $F^{\infty}$.

The aim of this paper is to prove the following theorem.

\begin{thm}[=Theorem \ref{thm:maintheorem}] \label{thm:mt}
Assume that the following conditions are satisfied.
\begin{enumerate}[$(1)$]
\item $G =\gal{F^{\infty}/F} \cong \begin{pmatrix}
1 & \Fp & \Fp & \Fp \\ 0 & 1 & \Fp & \Fp \\
0 & 0 & 1 & \Fp \\ 0 & 0 & 0 & 1 \end{pmatrix} \times \Gamma$, \\
where $\Gamma$ is a commutative $p$-adic Lie group isomorphic to $\Zp$.
\item $p\neq 2,3$.
\item There exists a finite extension $F'\subseteq F^{\infty}$ of $F$ 
such that the $\mu$-invariant of $(F')^{\mathrm{cyc}}/F'$ equals to zero, 
where $(F')^{\mathrm{cyc}}/F'$ is the cyclotomic $\Zp$-extension of $F'$.
\end{enumerate}

Then the $p$-adic zeta function $\xi_{F^{\infty}/F}$ for $F^{\infty}/F$
 exists and the Iwasawa main conjecture for $F^{\infty}/F$ is true.
\end{thm}

Let us summarize how to prove this theorem. We consider the following family of subgroups of $G$.

\begin{align*}
U_0 &= G, & V_0 &= \begin{pmatrix} 1 & 0 & \Fp & \Fp \\ 
0 & 1 & 0 & \Fp \\ 0 & 0 & 1 & 0 \\ 0 & 0 & 0 & 1
\end{pmatrix} \times \{1\}, \\
U_1 &= \begin{pmatrix} 1 & \Fp & \Fp & \Fp \\ 
0 & 1 & 0 & \Fp \\ 0 & 0 & 1 & \Fp \\ 0 & 0 & 0 & 1
\end{pmatrix} \times \Gamma, &
V_1 &= \begin{pmatrix} 1 & 0 & 0 & \Fp \\ 
0 & 1 & 0 & 0 \\ 0 & 0 & 1 & 0 \\ 0 & 0 & 0 & 1
\end{pmatrix} \times \{1\},  \\
\tilder{U_2} &= \begin{pmatrix} 1 & 0 & \Fp & \Fp \\ 
0 & 1 & \Fp & \Fp \\ 0 & 0 & 1 & \Fp \\ 0 & 0 & 0 & 1
\end{pmatrix} \times \Gamma, &
\tilder{V_2} &= \begin{pmatrix} 1 & 0 & 0 & \Fp \\ 
0 & 1 & 0 & \Fp \\ 0 & 0 & 1 & 0 \\ 0 & 0 & 0 & 1
\end{pmatrix} \times \{1\}, \\
U_2 &= \begin{pmatrix} 1 & 0 & 0 & \Fp \\ 
0 & 1 & \Fp & \Fp \\ 0 & 0 & 1 & \Fp \\ 0 & 0 & 0 & 1
\end{pmatrix} \times \Gamma, &
V_2 &= \begin{pmatrix} 1 & 0 & 0 & 0 \\ 
0 & 1 & 0 & \Fp \\ 0 & 0 & 1 & 0 \\ 0 & 0 & 0 & 1
\end{pmatrix} \times \{1\}, \\
U_3 &= \begin{pmatrix} 1 & 0 & 0 & \Fp \\ 
0 & 1 & 0 & \Fp \\ 0 & 0 & 1 & \Fp \\ 0 & 0 & 0 & 1
\end{pmatrix} \times \Gamma, &
V_3 &= \{ I_4 \} \times \{1\},
\end{align*} 
where $I_4$ is the unit matrix of $\GL_4(\Fp)$.

In the following, we use the notation $U_i$ (resp.\ $V_i$) for one of the subgroups $U_0, U_1, \tilder{U_2}, U_2$ and $U_3$ (resp.\ $V_0,V_1,\tilder{V_2},V_2$ and $V_3$). Note that each quotient group $U_i/V_i$ is abelian. 

Now we have a homomorphism
\begin{equation*}
\theta_i \colon K_1(\iw{G}) \xrightarrow{\Nr} K_1(\iw{U_i}) \rightarrow
 K_1(\iw{U_i/V_i})=\iw{U_i/V_i}^{\times}
\end{equation*}
where $\iw{G}$ denotes the Iwasawa algebra of $G$. Here the first map 
is the norm map of $K$-theory and the second one is the natural map 
induced by $\iw{U_i} \rightarrow \iw{U_i/V_i}$. 

On the other hand, John Coates et al.\ introduced the canonical
\O re set $S$ of $\iw{G}$ (See \S 2 and \cite{CFKSV}) and considered the
\O re localization $\iw{G}_S$  to formulate the main conjecture. For this
localized Iwasawa algebra, we also have a homomorphism
\begin{equation*}
\theta_{S,i} \colon K_1(\iw{G}_S) \xrightarrow{\Nr} K_1(\iw{U_i}_S)
 \rightarrow K_1(\iw{U_i/V_i}_S)=\iw{U_i/V_i}_S^{\times}
\end{equation*}
by the same construction as $\theta_i$. We denote
\begin{equation*}
\theta =(\theta_i)_i \colon K_1(\iw{G}) \rightarrow \prod_i \iw{U_i/V_i}^{\times} 
\end{equation*}
and
\begin{equation*}
\theta_S =(\theta_{S,i})_i \colon K_1(\iw{G}_S) \rightarrow \prod_i \iw{U_i/V_i}_S^{\times}.
\end{equation*}

Then we have the following proposition.

\begin{prop}[Proposition \ref{prop:theta},
 Proposition \ref{prop:imageaks} and Proposition \ref{prop:cap}]
There exist subgroups 
\begin{equation*}
\Psi \subseteq \prod_i \iw{U_i/V_i}^{\times}
\end{equation*}
and 
\begin{equation*}
\Psi_S \subseteq \prod_i \iw{U_i/V_i}_S^{\times}
\end{equation*}
which satisfy the following conditions$\colon$
\begin{enumerate}[$(1)$]
\item $\mathrm{Image}(\theta_S) \subseteq \Psi_S$.
\item $\mathrm{Image}(\theta)=\Psi$.
\item $\displaystyle \Psi_S \cap \prod_i \iw{U_i/V_i}^{\times}=\Psi$.
\end{enumerate}
\end{prop}

We can characterize both $\Psi$ and $\Psi_S$ as the subgroups consisting of
 all elements which satisfy certain norm relations and certain
 congruences (for details, see Definition \ref{def:psi} and 
Proposition \ref{prop:psis}.). In the following, we denote the induced homomorphisms by the same symbols
\begin{equation*}
\theta \colon K_1(\iw{G}) \rightarrow \Psi \quad \text{and} \quad \theta_S \colon K_1(\iw{G}_S) \rightarrow \Psi_S.
\end{equation*}

We call the surjective homomorphism $\theta$ {\it the theta map for $G$} and 
the homomorphism $\theta_S$ {\it the localized theta map for $G$}.

Then we obtain the following outstanding theorem which was first
observed by David Burns.

\begin{thm}[=Theorem \ref{thm:ak-pri}, Burns]
Let $\xi_i$ be the $p$-adic zeta function for the abelian $p$-adic Lie extension $F_{V_i}/F_{U_i}$ where $F_{U_i}$
 $($resp.\ $F_{V_i})$ is the maximal intermediate field of $F^{\infty}/F$ fixed by $U_i$ $($resp.\ $V_i)$. If $(\xi_i)_i$ is contained in $\Psi_S$, the $p$-adic zeta function $\xi$ for $F^{\infty}/F$ exists as an element of $K_1(\iw{G}_S)$ and satisfies the main conjecture.
\end{thm}

Note that the $p$-adic zeta function (pseudomeasure) $\xi_i$ for
$F_{V_i}/F_{U_i}$ has been constructed by using the theory of Pierre R.\
Deligne and Kenneth A.\ Ribet (\cite{De-Ri}), and the Iwasawa main
conjecture for $F_{V_i}/F_{U_i}$ has been proven by Andrew Wiles (\cite{Wiles}). 

The condition for $(\xi_i)_i$ to be contained in
$\Psi_S$ is essentially given by the congruences among $\xi_i$'s, 
so we may reduce the non-commutative Iwasawa main conjecture to 
the congruences among the abelian $p$-adic zeta pseudomeasures via the theta map.

In order to study the congruences which abelian $p$-adic zeta pseudomeasures should satisfy,
we use the theory of Hilbert modular forms of Deligne-Ribet (\cite{De-Ri}), and 
derive the congruences of $p$-adic zeta pseudomeasures (that is, the
congruences of constant terms of certain $\Lambda$-adic Hilbert modular forms) from those of the coefficients of non-constant terms of
$\Lambda$-adic Hilbert modular forms (See \S 7). Kato and 
Ritter-Weiss first used this technique in \cite{Kato1} and \cite{R-W3},
and obtained many kinds of congruences among abelian $p$-adic zeta 
pseudomeasures.

Actually it is difficult to prove all the desired congruences by using
only this technique, therefore
we use the existence of the $p$-adic zeta function for a certain
quotient group $\line{G}$ of $G$, proven by Kato in
\cite{Kato1}. We prove our main theorem (Theorem \ref{thm:mt}) by using an inductive technique.

%
\subsection{Overview}
%

The detailed content of this paper is as follows.

In \S 1, we review basic results of (classical) algebraic 
$K$-theory. In particular, we summarize properties of 
integral logarithmic homomorphisms, which were first introduced by 
Robert Oliver and Laurence R.\ Taylor to study the structure of the $K_1$-group 
of a group ring $R[G]$ where $G$ is a finite group and $R$ is the integer ring of a finite extension of $\Qp$. 
The integral logarithmic homomorphisms are maps from  multiplicative $K_1$-groups 
to certain additive groups, which are much easier to treat 
(Proposition-Definition \ref{pd:intlog}).

In \S 2, we review the theory of Coates et al. (\cite{CFKSV}), 
especially how to formulate the main conjecture.

We state our main theorem precisely in \S 3, and introduce Burns'
technique under more general situations than Theorem \ref{thm:mt}. 

We construct the theta map for our case from \S 4 to \S 6. In 
\S 4, we construct the additive version of the theta map
(Proposition-Definition \ref{pd:adak}) by using linear representation theory of finite groups. In \S 5, we translate the results on the additive theta map
proven in \S 4 into the (multiplicative) theta map, using the integral logarithmic
 homomorphisms. In \S 6, we construct the localized version of the theta
map $\theta_S$. For this purpose, we take the $p$-adic completion
$\comp{\iw{\Gamma}_{(p)}}$ of $\iw{\Gamma}_{(p)}$,
and apply the arguments of \S 4 and \S 5 to $\comp{\iw{\Gamma}_{(p)}}[G^f]$.

The condition for abelian $p$-adic zeta pseudomeasures to be contained 
in $\Psi_S$ is essentially given as the congruences among them. 
Hence in \S 7, we study congruences which abelian $p$-adic zeta 
pseudomeasures satisfy (Proposition \ref{prop:congzeta}). We use the theory of Deligne and Ribet on Hilbert modular forms (\cite{De-Ri}). 

Since the congruences obtained in \S 7 are not sufficient to conclude that 
abelian $p$-adic zeta pseudomeasures are contained in $\Psi_S$, 
we introduce Kato's $p$-adic zeta function for a certain 
sub $p$-adic Lie extension $F_N/F$ of $F^{\infty}/F$ (Theorem \ref{thm:kato}), and 
prove Theorem \ref{thm:mt} by an inductive technique.

Our main theorem gives a new example which is not deduced from previous
results (\cite{Kato1}, \cite{R-W2}, \cite{R-W3}, and \cite{R-W4}).

\tableofcontents

%
\subsection{Notation}
%

In this paper, $p$ always denotes a positive prime number. 

We denote by $\N$ the set of natural numbers (the set of {\em
strictly} positive integers), denote by $\Z$ (resp.\ $\Zp$) the ring of integers
(resp.\ $p$-adic integers), and also denote by $\Q$ (resp.\ $\Qp$) 
the rational number field (resp.\ the $p$-adic number field). 

For an arbitrary group $G$, we denote by $\conj{G}$ the set of all
conjugacy classes of $G$.

For every pro-finite group $P$, we always denote by $\iw{P}=\Zp[[P]]$ 
its Iwasawa algebra (that is, its completed group ring over $\Zp$).

We denote by $\Gamma$ a commutative $p$-adic Lie group which is 
isomorphic to $\Zp$. Throughout this paper, we fix a topological 
generator $t$ of $\Gamma$. In other words, we fix an isomorphism
\begin{eqnarray*}
\iw{\Gamma} &\xrightarrow{\simeq} & \Zp[[T]] \\
t & \mapsto & 1+T
\end{eqnarray*}
where $\Zp[[T]]$ is the formal power series ring over $\Zp$.

We always assume that every associative ring has $1$.
We also use the following notation:
\begin{align*}
\M_n(R) &=\{ \text{the ring of $n\times n$-matrices with entries
 in $R$}  \}, \\
\GL_n(R) &=\{ M\in \M_n(R) \mid M \, \text{is an invertible matrix} \},
\end{align*}
where $R$ is an associative ring. If $R$ is a commutative domain, we
denote by $\mathrm{Frac}(R)$ its fractional field.

We always regard $K_0$-groups as additive groups, 
and $K_1$-groups as multiplicative groups.

Finally, we fix an algebraic closure $\line{\Q}$ of $\Q$ and fix embbedings 
\begin{equation*}
\line{\Q}\hookrightarrow \mathbb{C}, \qquad \line{\Q} \hookrightarrow \line{\Q}_p
\end{equation*}
where $\mathbb{C}$ denotes the complex number field and $\line{\Q}_p$ the
algebraic closure of $\Qp$.

%
\subsection{Acknowledgment}
%

The author would like to thank everyone with whom he has been concerned. 
He would like to express his sincere gratitude to Professor Shuji Saito, 
Professor Kazuya Kato and Professor Takeshi Tsuji among them; 
Professor Shuji Saito has invited the author to the mysterious and interesting 
world of number theory through his lectures and seminars; Professor Kazuya Kato has shown the author the mystic aspect of non-commutative Iwasawa theory and motivated the author to study it through his intensive lectures at the University of Tokyo in 2006; and Professor Takeshi Tsuji, whom the author is most grateful to, has given the author a lot of useful advice through his seminars and read the manuscript very carefully. Especially, the main idea of the inductive technique (used in \S8) is due to him. This paper could never have existed without their direct and indirect cooperation.

Mahesh Kakde has recently generalized the Kato's method used in
\cite{Kato1} and proven the main conjecture for other cases in the
different way from this paper (see \cite{Kakde}). The author would like
to express his sincere gratitude to Mahesh Kakde for sending the latest
version of his paper.

%
%
\section{Preliminaries on algebraic $K$-theory}
%
%

In this section, we summarize basic results of algebraic $K$-theory 
which we need to formulate the non-commutative Iwasawa main conjecture
and to construct the theta map. 

%
\subsection{Definitions and first properties}
%

First, we review the definition of $K$-groups and their 
properties. For more details, see \cite{Bass1}.

\begin{defn}[Grothendieck groups, $K_0$-groups]
Let $\mathscr{C}$ be a category with a product $\bot$ $($recall that a category $\mathscr{C}$ with a product $\bot$ is a 
category equipped with a functor $\bot \colon \mathscr{C} \times 
\mathscr{C} \rightarrow \mathscr{C})$. 
Then we define {\it the Grothendieck group of $\mathscr{C}$} 
$K_0(\mathscr{C})$ as an abelian group equipped with a map
\begin{equation*}
[\, \cdot\, ] \colon \obj{\mathscr{C}} \longrightarrow K_0(\mathscr{C})
\end{equation*}
satisfying the following universal properties.
\begin{enumerate}[($K_0$-1)]
\item For every $X,Y\in \obj{\mathscr{C}}$ satisfying $X\cong Y$, 
$[X]=[Y]$.
\item For every $X,Y \in \obj{\mathscr{C}}$, $[X\bot Y]=[X]+[Y]$.
\end{enumerate}

Namely, if a map $f\colon\obj{\mathscr{C}} \longrightarrow 
A $ ($A\colon  \text{an abelian group}$) satisfies the properties ($K_0$-1) and ($K_0$-2), 
there exists a unique group homomorphism $\psi \colon K_0(\mathscr{C}) \rightarrow A$ which makes 
the following diagram commute:
\begin{equation*}
\xymatrix{
\obj{\mathscr{C}} \ar[r]^{[\, \cdot\, ]} \ar[dr]_f & K_0(\mathscr{C}) \ar@{.>}[d]^{{}^{\exists !}\psi} \\
 & A
}
\end{equation*}

For every associative ring $R$, we define $K_0(R)$ 
as the Grothendieck group of the category of finitely generated 
projective left $R$-modules.
\end{defn}

\begin{defn}[Whitehead groups, $K_1$-groups] \label{def:k1}
Let $\mathscr{C}$ be a category with a product $\bot$. 
Let $\Aut(\mathscr{C})$ be the category of automorphisms of 
objects of $\mathscr{C}$. Namely, an object of $\Aut(\mathscr{C})$
 is a pair $(X,\sigma)$ where $X\in \obj{\mathscr{C}}$ and $\sigma \colon X\rightarrow X$ is an automorphism of $X$. A morphism $f \colon (X,\sigma)\longrightarrow (Y,\tau)$ is a morphism $f \colon X\rightarrow Y$ in $\mathscr{C}$ 
which satisfies $f\circ \sigma=\tau \circ f$. 

Then we define {\it the Whitehead group of $\mathscr{C}$} 
$K_1(\mathscr{C})$ as an abelian group equipped with a map
\begin{equation*}
[\, \cdot \,] \colon \obj{\Aut(\mathscr{C})} \longrightarrow K_1(\mathscr{C})
\end{equation*}
satisfying the following universal properties.
\begin{enumerate}[($K_{1}$-1)]
\item For every $(X,\sigma),(Y,\tau)\in \obj{\Aut(\mathscr{C})}$ satisfying $(X,\sigma)\cong (Y,\tau)$, 
$[(X,\sigma)]=[(Y,\tau)]$.
\item For every $(X,\sigma),(Y,\tau) \in \obj{\Aut(\mathscr{C})}$,
      $[(X,\sigma)\bot (Y,\tau)]=[(X,\sigma)]\cdot[(Y,\tau)]$.
\item For every $(X,\sigma),(X,\sigma')\in \obj{\Aut(\mathscr{C})}$,
$[(X,\sigma \circ \sigma')]=[(X,\sigma)]\cdot [(X,\sigma')]$.
\end{enumerate}

For every associative ring $R$, 
we define $K_1(R)$ 
as the Whitehead group of the category of finitely generated 
projective left $R$-modules.
\end{defn}

\begin{rem}
When $R$ is the group ring $\integer{K}[G]$ of a finite group $G$ 
over the integer ring of a number field $K$, the terminology ``Whitehead group'' is usually used for the
 group
\begin{equation*}
\mathrm{Wh}(\integer{K}[G])=K_1(\integer{K}[G])/(\mu(\integer{K})\times G^{\mathrm{ab}}),
\end{equation*}
where $\mu(\integer{K})$ is the multiplicative group consisting of all roots
 of $1$ contained in $\integer{K}$.
\end{rem}

We have another interpretation of the Whitehead group of an associative
ring $R$ (Whitehead's construction): let $\E_n(R)$ be a subgroup of
$GL_n(R)$ generated by all ``elementary
matrices,'' that is,
\begin{equation*}
\E_n(R) = \langle I_n+rE_{ij} \mid 1\leq i\neq j\leq n,\, r\in R \rangle
\end{equation*}
where 
\begin{equation*}
E_{ij}=
\bordermatrix{
  &   &   &  & j & \cr
  & 0 & 0 & \cdots & & 0 \cr
i & 0 & \ddots &  & 1 &  \cr
  &   &  & \ddots &  &  \cr
  & \vdots   &  &   & \ddots  & \vdots \cr
  & 0  & \cdots  &   & \cdots   &  0 }.
\end{equation*}

Here we denote the unit matrix of $\GL_n(R)$ by $I_n$. Note that $\E_n(R)$ is normal in $\GL_n(R)$.

Let 
\begin{align*}
\GL(R)&=\varinjlim_n \GL_n(R) & \text{and} & & \E(R) &=\varinjlim_n \E_n(R),
\end{align*}
then we have 
\begin{equation*}
K_1(R)= \GL(R)/\E(R).
\end{equation*}

For the equivalence of Definition \ref{def:k1} and Whitehead construction,
see \cite{Bass1}, Chapter IX.

\begin{defn}[Relative Whitehead groups]
Let $R$ be an associative ring and $\ideal{a}\subseteq R$ an arbitrary 
(two-sided) ideal. Set
\begin{align*}
\GL_n(R,\ideal{a}) &= \Ker (\GL_n(R) \longrightarrow \GL_n(R/\ideal{a})), \\
\E_n(R,\ideal{a}) &= \text{The minimal normal subgroup of $\GL_n(R,\ideal{a})$} \\
& \qquad \text{containing
 } \{ I_n+rE_{ij} \mid 0\leq i \neq j\leq n,\, r\in \ideal{a} \},
\end{align*}
and
\begin{align*}
\GL(R,\ideal{a})&= \varinjlim_n \GL_n(R,\ideal{a}), & \E(R,\ideal{a}) &= \varinjlim_n 
\E_n(R,\ideal{a}).
\end{align*}

Then we define
\begin{equation*}
K_1(R,\ideal{a})=\GL(R,\ideal{a})/\E(R,\ideal{a}).
\end{equation*}
\end{defn}

\begin{prop}[Whitehead's lemma] \label{prop:whitehead}
Let $R$ and $\ideal{a}$ be as above. Then
\begin{align*}
\E(R) &= [\GL(R), \GL(R)], \\
\E(R,\ideal{a}) &= [\E(R),\E(R,\ideal{a})] = [\GL(R),\GL(R,\ideal{a})].
\end{align*}
\end{prop}

\begin{proof}
See \cite{Milnor}, Lemma 3.1, Lemma 4.3.
\end{proof}

The following two propositions are well known and we often use them later.

\begin{prop} \label{prop:surjonk1}
Let $R$ be an associative ring and $\ideal{a}$ be a two-sided 
ideal contained in its Jacobson radical. Then the canonical homomorphism
\begin{equation*}
K_1(R) \longrightarrow K_1(R/\ideal{a})
\end{equation*}
induced by $R \rightarrow R/\ideal{a}$ is surjective.
\end{prop}

\begin{proof}
See \cite{Bass1}, Chapter IX, Proposition (1.3).
\end{proof}

\begin{prop} \label{prop:sloc}
Let $R$ be a semi-local associative ring $($recall that $R$ is semi-local if $R/J$ is semi-simple where $J$ is the Jacobson radical of $R)$. Then the following properties hold.
\begin{enumerate}[$(1)$]
\item The group homomorphism
\begin{equation*}
R^{\times} \longrightarrow K_1(R) ; u\mapsto [(R,-\cdot u)]
\end{equation*}
is surjective, where $-\cdot u$ is the automorphism of $R$ defined by multiplication by $u$ from the right $($here we regard $R$ as a left $R$-module$)$.

If $R$ is semi-local and also commutative, the homomorphism above is 
an isomorphism.

\item $($stability property$)$

For each $d\geq 2$, we have the canonical isomorphism 
\begin{equation*}
K_1(R)\cong \GL_d(R)/\E_d(R).
\end{equation*}
\end{enumerate}
\end{prop}

\begin{proof}
For (1), see \cite{Bass1} Chapter IX, Proposition (1.4). When $R$ is
 commutative, the determinant map gives the inverse map of $R^{\times}
 \rightarrow K_1(R)$.

For (2), see \cite{Bass1} Chapter V, Theorem (9.1).
\end{proof}

Now let us study the project limit of $K_1$-groups for semi-local rings.

\begin{prop} \label{prop:projlim}
Let $R$ be a semi-local ring and $\{ R^{(n)} \}_{n\in \mathbb{N}}$ a projective system of semi-local rings such that $R^{(n)}\rightarrow R^{(m)}$ is surjective for each $n>m\geq 1$ and $\varprojlim_n R^{(n)} \cong R$.

Then the homomorphism 
\begin{equation*}
K_1(R) \rightarrow \varprojlim_n K_1(R^{(n)})
\end{equation*}
induced by the canonical homomorphisms $K_1(R)\rightarrow K_1(R^{(n)})$ is an isomorphism.
\end{prop}

\begin{proof}
By the stability property (Proposition \ref{prop:sloc} (2)), we have 
\begin{equation*}
K_1(R)\cong \GL_d(R)/\E_d(R) \quad \text{and} \quad K_1(R^{(n)}) \cong \GL_d(R^{(n)})/\E_d(R^{(n)})
\end{equation*}
for every $d\geq 2$. Fix $d\geq 2$.

Consider the exact sequence
\begin{equation} \label{eq:exgle}
\begin{CD}
1 @>>> \E_d(R^{(n)}) @>>> \GL_d(R^{(n)}) @>>> K_1(R^{(n)}) @>>> 1\end{CD}
\end{equation}
for each $n\geq 1$. Then for every $n> m\geq 1$, the natural homomorphism 
\begin{equation*}
\E_d(R^{(n)}) \rightarrow \E_d(R^{(m)})
\end{equation*}
induced by $R^{(n)} \rightarrow R^{(m)}$ is clearly surjective. Therefore $\{ \E_d(R^{(n)}) \}_{n\in \mathbb{N}}$ satisfies the Mittag-Leffler condition, and we have the following exact sequence
\begin{equation*}
\xymatrix{
1 \ar[r] & \varprojlim_n \E_d(R^{(n)}) \ar[r] & \varprojlim_n \GL_d(R^{(n)}) \ar[r] & \varprojlim_n K_1(R^{(n)}) \ar[r] & 1
}
\end{equation*}
by taking the projective limit of (\ref{eq:exgle}). 

Now consider the following diagram:
\begin{equation*}
\xymatrix{
1 \ar[r] & \E_d(R) \ar[d] \ar[r] & \GL_d(R) \ar[d] \ar[r] & K_1(R) \ar[d] \ar[r] & 1 \\1 \ar[r] & \varprojlim_n \E_d(R^{(n)}) \ar[r] & \varprojlim_n \GL_d(R^{(n)}) \ar[r] & \varprojlim_n K_1(R^{(n)}) \ar[r] & 1
}
\end{equation*}
here the vertical homomorphisms are induced by the canonical homomorphism $R \rightarrow R^{(n)}$. We can easily check that each square diagram commutes.

Then the left and middle vertical homomorphisms are canonically isomorphic. Hence the right vertical homomorphism is also an isomorphism by the snake lemma.
\end{proof}

%
\subsection{Norm maps in $K$-theory}
%

The theta map, which we will construct in the following sections, 
is essentially a family of norm maps in algebraic $K$-theory.\footnote{In 
some books and papers, norm maps are also called 
``transfer homomorphisms.''} 
So let us review the construction and properties 
of norm maps of $K$-groups.

Let $R'/R$ be an extension of a ring $R$. Suppose that $R'$ is 
a finitely generated projective module as a left $R$-module. 
Then we may regard $R'$ as left $R$- right $R'$-bimodule ${}_R R'_{R'}$. 
We define

\begin{equation*}
\Nr_{R'/R} := \left[{}_R R'_{R'} \otimes_{R'} - \right] \colon K_i(R') \longrightarrow
 K_i(R) \qquad (i=0,1).
\end{equation*}

We often use norm maps in the following situation: let $G$ be 
a group and $H$ be its subgroup of finite index. Then we have an inclusion of 
group rings $\Zp[H] \rightarrow \Zp[G]$. Hence we obtain 
a norm map 
\begin{equation*}
\Nr_{\Zp[G]/\Zp[H]} \colon K_i(\Zp[G]) \longrightarrow K_i(\Zp[H]) \qquad 
(i=0,1).
\end{equation*}

If $G$ is a pro-finite group and $H$ is its open subgroup, we also obtain $\Nr_{\iw{G}/\iw{H}}$ by the same construction.

Now we calculate
\begin{equation*}
\Nr_{\Zp[G]/\Zp[H]} \colon K_1(\Zp[G]) \longrightarrow K_1(\Zp[H])\cong \Zp[H]^{\times}
\end{equation*}
under the specific condition that $G$ is a group and $H$ is a commutative subgroup of finite index. Note that both $\Zp[G]$ and $\Zp[H]$ are local rings.
By Proposition \ref{prop:sloc} (1), we can identify
$K_1(\Zp[H])$ with $\Zp[H]^{\times}$. Take a system of representatives 
$\{ u_1, \dotsc , u_r \}$ of the coset decomposition $H\backslash
G$. Then $\Zp[G]$ is regarded as a
left free $\Zp[H]$-module with basis $\{ u_1, \dotsc, u_r \}$. 

Let $\phi$ be an element of $K_1(\Zp[G])$. By Proposition 
\ref{prop:sloc} (1) again, we obtain $x\in \Zp[G]^{\times}$ such that 
$[x]=\phi$. Let
\begin{align*}
u_jx &= \sum_{i=1}^r x_{ij}u_i & (x_{ij} \in \Zp[H]) 
\end{align*}
for each $j$. Then we can calculate $\Nr_{\Zp[G]/\Zp[H]} \phi$ as  
\[
 \Nr_{\Zp[G]/\Zp[H]}\phi=\det ((x_{ij})_{1\leq i,j\leq r})  \quad \in
 \Zp[H]^{\times}.
\]

\begin{proof}
By the calculation above, we may identify $\Nr_{\Zp[G]/\Zp[H]} \phi$ with 
the image of $(x_{ij})_{i,j}$ in $K_1(\Zp[H])=\GL(\Zp[H])/\E(\Zp[H])$. 
Since $\Zp[H]$ is commutative, the surjection $\Zp[H]^{\times} 
\rightarrow K_1(\Zp[H])$ is an isomorphism by Proposition 
\ref{prop:sloc} (1). The determinant map gives the inverse map of the isomorphism above, and it maps $\Nr_{\Zp[G]/\Zp[H]} \phi =\left[ (x_{ij})_{i,j} \right]$ to 
$\det (x_{ij})_{i,j}$.
\end{proof}

When $G$ is a pro-finite group and $H$ is its commutative open subgroup,
we may calculate $\Nr_{\iw{G}/\iw{H}}$ explicitly in the same way.

We end this subsection with a certain compatibility property of norm maps.

\begin{prop} \label{prop:commutativity}
Let $G$ be a group and $H$ be its subgroup of finite index.
Let $N$ be a subgroup of $H$ normal in $G$ and $H$.

Then for $i=0,1$, the following diagram commutes$\colon$

\begin{equation*}
\begin{CD}
K_i(\Zp[G]) @>\Nr_{\Zp[G]/\Zp[H]}>> K_i(\Zp[H]) \\
@V{\pi_G}VV   @VV{\pi_H}V \\
K_i(\Zp[G/N]) @>>\Nr_{\Zp[G/N]/\Zp[H/N]}> K_i(\Zp[H/N])
\end{CD}
\end{equation*}
where $\pi_G$ and $\pi_H$ are canonical homomorphisms 
induced by 
\begin{align*}
\pi_G &\colon \Zp[G] \longrightarrow \Zp[G/N] &\text{and} && \pi_H &\colon \Zp[H] \longrightarrow \Zp[H/N].
\end{align*}

When $G$ is a pro-finite group, $H$ is its open subgroup and $N$ is a 
closed subgroup of $H$ normal in $G$ and $H$, the same statement holds 
for $\Nr_{\iw{G}/\iw{H}}$ and $\Nr_{\iw{G/N}/\iw{H/N}}$.
\end{prop}

\begin{proof}
Let $\{ u_1, \dotsc , u_r \}$ be a system of representatives of
 $H\backslash G$. 
Then it is clear that $\{ \line{u}_1 , \dotsc , \line{u}_r \}$ is 
that of $(H/N)\backslash (G/N)$ where
 $\line{u}_i=\pi_G(u_i)$. Then we have 
\begin{align*}
& \quad \pi_H \circ \Nr_{\Zp[G]/\Zp[H]} \\
&= [\Zp[H/N] \otimes_{\Zp[H]}
 \left( {}_{\Zp[H]}\Zp[G]_{\Zp[G]} \right) \otimes_{\Zp[G]} -] \\
&= \left[ \Zp[H/N] \otimes_{\Zp[H]} \left( \bigoplus_{i=1}^r \Zp[H]u_i
 \right) \otimes_{\Zp[G]} - \right]  \\
&= \left[ \left( \bigoplus_{i=1}^r \Zp[H/N] \line{u}_i \right)
 \otimes_{\Zp[G]} - \right] \\
&= \left[ \left( {}_{\Zp[H/N]}\Zp[G/N]_{\Zp[G/N]} \right) \otimes_{\Zp[G/N]} (\Zp[G/N])
 \otimes_{\Zp[G]} - \right]  \\
&= \Nr_{\Zp[G/N]/\Zp[H/N]} \circ \pi_G.
\end{align*}

The pro-finite version can be verified in the same manner.
\end{proof}

%
\subsection{The localization exact sequences}
%

Let $R$ be an associative ring and $S$ be a multiplicatively closed
subset of $R$ containing $1$. In 1968, Hyman Bass constructed 
``the $K_1$-$K_0$ localization sequence'' (\cite{Bass1}):
\begin{equation*}
K_1(R) \rightarrow K_1(S^{-1}R) \xrightarrow{\del} K_0(R,S) \rightarrow
 K_0(R) \rightarrow K_0(S^{-1}R)
\end{equation*}
for central $S$ (recall that $S$ is {\em central} if $S$ contained in the center of $R$). 

In this subsection, we define \O re localizations (certain ``good''
localizations of non-commutative associative rings) and introduce the
$K_1$-$K_0$ localization exact sequence for a non-central \O re set $S$ which is the generalization of Bass' exact sequence for central $S$.

\begin{defn}[\O re sets] \label{def:ore}
Let $R$ be an associative ring, and $S\subseteq R$ be 
a multiplicatively closed subset containing $1$. 

$S$ is called {\it a left 
$($resp.\ right$)$ \O re set} if $S$ satisfies following two conditions.
\begin{enumerate}[(\O re-1)]
\item For every $r\in R$ and $s\in S$, $Sr \cap Rs \neq \emptyset$ 
(resp.\ $rS\cap sR \neq \emptyset$).
\item If $r\in R$ satisfies $rs=0$ (resp.\ $sr=0$) for a certain element $s\in S$, 
there exists $s'\in S$ such that $s'r=0$ (resp.\ $rs'=0$).
\end{enumerate}
\end{defn}

\begin{ex}
Let $R$ be an associative ring and $S$ be a multiplicatively closed subset containing $1$. Suppose that $S$ is contained in the center of $R$. Then $S$ is a 
left and right \O re set: for an arbitrary $r\in R$ and $s\in S$, there exist $r'\in R$ and $s' \in S$ such that $s'r=r's$ and $rs'=sr'$ since we may choose $s'=s$ and $r'=r$.
\end{ex}

\begin{prop} 
Let $R$ be an associative ring and $S$ a left $($resp.\ right$)$ \O
 re set. Let $k\in \N$. Then the following property holds$\colon$
\begin{enumerate}[$(\text{\O re}$-$1')$]
\item For arbitrary $s_1, \dotsc, s_k\in S$, there exist $r_1, \dotsc, r_k \in R$ which satisfy $r_1s_1=\cdots =r_ks_k\in S$ 
$($resp.\ $s_1r_1=\cdots =s_kr_k\in S)$.

\end{enumerate}
\end{prop}

\begin{proof}
Directly by (\O re-1) for $k=2$. Then we can show (\O re-$1'$) for $k\geq 3$ by induction.
\end{proof}

\begin{propdef}[\O re localization] \label{pd:loc}
Let $R$ be an associative ring and $S$ be a left $($resp.\ right$)$ 
\O re set of $R$.
\begin{enumerate}[$(1)$]
\item $($Existence of the \O re localization$)$

There exist a ring $[S^{-1}]R$ $($resp.\ $R[S^{-1}]$$)$ and a canonical ring 
homomorphism $\iota\colon R\longrightarrow [S^{-1}]R$ $($resp.\
      $R[S^{-1}]$$)$ 
which satisfy the following conditions$\colon$
\begin{enumerate}[$(${\upshape loc-}$1)$]
\item For every $s\in S$, $\iota(s)$ is invertible.
\item Every element in $[S^{-1}]R$ $($resp.\ $R[S^{-1}]$$)$ can be described in the form $\iota(s)^{-1} \iota(r)$ $($resp.\ $\iota(r)\iota(s)^{-1}$$)$ for certain 
$r\in R$ and $s\in S$.
\item For $r\in R$, $\iota(r)=0$ if and only if 
there exists $s\in S$ such that $sr=0$ $($resp.\ $rs=0$$)$ in $R$.
\end{enumerate}
\item $($The universality property$)$

Suppose that a ring homomorphism $f\colon R\longrightarrow \tilder{R}$
      maps  all elements of $S$ to invertible elements of
      $\tilder{R}$. Then there exists a unique ring homomorphism $\psi$
      which makes the following diagram commute$\colon$

\begin{equation*}
\xymatrix{
  &  [S^{-1}]R \quad (\text{resp.\ } R[S^{-1}]) \ar@{.>}[d]^{{}^{\exists !} \psi} \\
R \ar[ur]^{\iota} \ar[r]_f & \tilder{R}
}
\end{equation*}
\end{enumerate}

{\upshape We call $[S^{-1}]R$ (resp.\ $R[S^{-1}]$)} the left $($resp.\
 right$)$ \O re localization of $R$ with respect to $S$.

{\upshape For an arbitrary left (resp.\ right) $R$-module $M$, we define }
the left $($resp.\ right$)$ \O re localization of $M$ with respect to $S$ {\upshape to be the module} 
$[S^{-1}]M=[S^{-1}]R \otimes_R M$ $($resp.\ $M[S^{-1}]=M\otimes_R R[S^{-1}])$.
\end{propdef}

\begin{proof}[Sketch of the proof]

We only give the construction of the left \O re localization $[S^{-1}]R$. Set 
\begin{equation*}
[S^{-1}]R = S\times R/\sim
\end{equation*}
where $\sim$ is an equivalence relation defined by $(s,r) \sim (s',r')$ if 
and only if there exist $a,b \in R$ which satisfy $ar=br'$ and $as=bs' \in S$.

Then we may define the additive law and the multiplicative law as follows:
\begin{align*}
(s,r)+(s',r') &= (t, ar+br') & \text{where } t=as=bs'\in S, \\
(s,r)\cdot (s',r') &= (ts, ar') & \text{where } tr=as', \quad t\in S.
\end{align*}

For the well-definedness of $\sim$, $+$, $\, \cdot \,$, and for the universality property, we use the \O re conditions (\O re-1) and (\O re-2). 

For details, see \cite{Sten}, Chapter II.
\end{proof}

\begin{cor} \label{cor:isom}
Let $R$ be an associative ring. If a multiplicatively closed subset $S$ satisfies both left and right \O re conditions, then the canonical 
isomorphism $[S^{-1}]R \xrightarrow{\simeq} R[S^{-1}]$ exists.
\end{cor}

\begin{proof}
Directly from the universality property.
\end{proof}

\begin{prop} \label{prop:flat}
Let $R$ be an associative ring and $S$ a left $($resp.\ right$)$ \O re set. Then the left $($resp.\ right$)$ \O re localization $[S^{-1}]$R $($resp.\ $R[S^{-1}])$ is flat as a right $($resp. left$)$ $R$-module.

In other words, the left $($resp.\ right$)$ \O re localization defines the exact functor from the category of left $($resp.\ right$)$ $R$-modules to that of 
left $[S^{-1}]R$-modules $($resp.\ right $R[S^{-1}]$-modules$)$.
\end{prop}

\begin{proof}
See \cite{Sten}, Chapter II, Proposition 3.5.
\end{proof}

\begin{defn}[$S$-torsion modules]
Let $R$ be an associative ring and $S$ be a left \O re set. We define 
$\ideal{M}_S(R)$ to be the category of finitely generated $S$-torsion left $R$-modules,
 that is, an object $M$ of $\ideal{M}_S(R)$ is a finitely generated left $R$-module satisfying $[S^{-1}]M=0$.
\end{defn}

We may easily show that for an arbitrary object $M$ of $\ideal{M}_S(R)$, there exists an element of $S$ such that $sM=0$ if $S$ is central. 
 
Bass constructed the following localization exact sequence for central $S$.

\begin{defn}
Let $R$ be an associative ring and $S$ its multiplicatively closed subset. Let $\ideal{H}_S(R)$ be the category of 
finitely generated $S$-torsion left $R$-modules with projective resolutions
 of finite length. Then we put $K_0(R, S) = K_0(\ideal{H}_S(R))$.
\end{defn}

\begin{rem}
Note that we may identify $K_0(R,S)$ with the relative Gro\-then\-dieck group
 associated to the canonical ring homomorphism $R \longrightarrow [S^{-1}]R$. 
For the definition of relative Grothendieck groups, see \cite{Bass1}, 
Chapter IX, \S 1. 

We can also identify this group with the Grothendieck group of the category
 of bounded complexes of finitely generated projective left $R$-modules 
whose cohomologies are of $S$-torsion.
\end{rem}

\begin{thm}[Bass, the localization exact sequence for central $S$] 
Let $R$ be an associative ring with $1$ and let $S$ be its multiplicatively closed subset contained in the center of $R$. Suppose that for every element $s$ of $S$, multiplication by $s$ in $R$ induces an injection $R \xrightarrow{s} R$. 

Then there exists an exact sequence of $K$-groups$\colon$
\begin{equation*}
K_1(R) \rightarrow K_1(S^{-1}R) \xrightarrow{\del} K_0(R,S) \rightarrow
 K_0(R) \rightarrow K_0(S^{-1}R).
\end{equation*}
\end{thm}

\begin{proof}
See \cite{Bass1}, Chapter IX, Theorem (6.3).
\end{proof}

\begin{prop}
If $R$ has finite global dimension $($in other words, if every finitely
 generated left $R$-module has finite projective dimension$)$, we have 
\begin{equation*}
K_i(R,S)=K_i(\ideal{M}_S(R)) \qquad (i=0,1).
\end{equation*}
\end{prop}

\begin{proof}
It is clear since $\ideal{H}_S(R)=\ideal{M}_S(R)$ in this case.
\end{proof}

A.\ J.\ Berrick and M.\ E.\ Keating constructed the localization sequence for \O re localizations by generalizing the Bass' construction of the localization sequence for central $S$:

\begin{thm}[Berrick-Keating]
Let $R$ be an associative ring with $1$ and $S$ a left \O re set. Suppose that no elements of $S$ are zero divisors in $R$.

Let $\ideal{H}_{S,1}(R)$ be the subcategory of $\ideal{M}_S(R)$
 consisting of finitely presented $S$-torsion left $R$-modules with projective dimension $1$. Then there exists an exact sequence of $K$-groups$\colon$
\begin{equation} \label{eq:seqore}
K_1(R) \rightarrow K_1([S^{-1}]R) \xrightarrow{\del} K_0(\ideal{H}_{S,1}(R)) \rightarrow
 K_0(R) \rightarrow K_0([S^{-1}]R).
\end{equation}
\end{thm}

\begin{proof}
See \cite{Ber-Keat}.
\end{proof}

We remark that Daniel R.\ Grayson also constructed the exact sequence
(\ref{eq:seqore}) by using Quillen's higher $K$-theory (\cite{Grayson}).

In the following, we use the modified version of Berrick-Keating's localization sequence.

\begin{thm}[The localization exact sequence for \O re localization] \label{thm:locseq}

Let \\ $R$ be an associative ring with $1$ and $S$ a left \O re set. Suppose that no elements of $S$ are zero divisors in $R$. Then there exists an exact sequence of $K$-groups$\colon$
\begin{equation*} 
K_1(R) \rightarrow K_1([S^{-1}]R) \xrightarrow{\del} K_0(R,S) \rightarrow
 K_0(R) \rightarrow K_0([S^{-1}]R).
\end{equation*}
\end{thm}

\begin{proof}
It suffices to show that 
\begin{equation} \label{eq:finite1}
K_0(\ideal{H}_S(R))=K_0(\ideal{H}_{S,1}(R)).
\end{equation} 

First, we show that for every $M\in \obj{\ideal{H}_S(R)}$, $M$ has a finite resolution by objects of $\ideal{H}_{S,1}(R)$. Let $\{ m_1, \dotsc , m_k \}$ be elements of $M$ which generate $M$, and $d$ the projective dimension of $M$. Then we have the canonical surjection
\begin{equation*}
\pi \colon \oplus_{i=1}^k Re_i \rightarrow M ; e_i \mapsto m_i \quad (1\leq i\leq k),
\end{equation*}
where $\{ e_i \}_{1\leq i\leq k}$ is the free basis. We set
 $L=\Ker(\pi)$. Note that $L$ is a left free module. Since $M$ is a $S$-torsion module, there exists $s_i\in S$ which annihilates $m_i$ for each $i$. Then we obtain an element $s\in S$ which annihilates all $m_i$ by (\O re-$1'$). For this $s$, we have $\pi \left( \oplus_{i=1}^k Rse_i \right)=0$, therefore $\oplus_{i=1}^k Rse_i$ is a left free $R$-submodule of $L$. Consider 
\begin{equation*}
0 \rightarrow L/\oplus_{i=1}^k Rse_i \rightarrow \oplus_{i=1}^k Re_i/\oplus_{i=1}^k Rse_i \rightarrow M \rightarrow 0,
\end{equation*}
then $\oplus_{i=1}^k Re_i/\oplus_{i=1}^k Rse_i$ is free and of $S$-torsion 
(Let $x=\sum_{i=1}^k r_ie_i$ be an arbitrary element of $\oplus_{i=1}^k
 Re_i/\oplus_{i=1}^k Rse_i$. For each $i$ there exist $s_i'\in S$ and
 $r_i'\in R$ such that $s_i'r_i=r_i's$. We obtain $r_i''\in R\quad
 (1\leq i\leq k)$ which satisfies $r_1''s_1'=\cdots =r_k''s_k' \in S$ by
 (\O re-$1'$). Then the element $\tilder{s}=r_1''s_1'=\cdots =r_k''s_k'$
 annihilates $x$), therefore $\oplus_{i=1}^k Re_i/\oplus_{i=1}^k Rse_i$
 is an object of $\ideal{H}_{S,1}(R)$. Note that the projective
 dimension of $L/\oplus_{i=1}^k Rse_i$ is $d-1$. Hence we can construct
 a resolution of $M$ by objects of $\ideal{H}_{S,1}(R)$ by induction on
 the projective dimension $d$.

Then (\ref{eq:finite1}) reduces to Grothendieck's resolution theorem
 (See Theorem \ref{thm:Gro}). Note that both $\ideal{H}_S(R)$ and
 $\ideal{H}_{S,1}(R)$ are admissible subcategories of $\ideal{M}(R)$ (that is, $\ideal{H}_S(R)$ and $\ideal{H}_{S,1}(R)$ are full additive
 subcategories of $\ideal{M}_S(R)$ which have at most sets of isomorphism classes of objects, and if $0\rightarrow M' \rightarrow M \rightarrow M'' \rightarrow 0$ is an arbitrary exact sequence in $\ideal{M}(R)$ for which $M$ and $M'$ are objects of $\ideal{H}_S(R)$ (resp.\ $\ideal{H}_{S,1}(R)$), $M''$ is also an object of $\ideal{H}_S(R)$ (resp.\ $\ideal{H}_{S,1}(R)$)). See \cite{Bass2}, Corollary (8.5). 
\end{proof}

\begin{thm}[Grothendieck's resolution theorem] \label{thm:Gro}
Let $\ideal{M}$ be an abelian category and $\ideal{P}\subseteq \ideal{P'}$ admissible subcategories of $\ideal{M}$. Assume an arbitrary object $P'$ of $\ideal{P}'$ has a finite resolution by objects of $\ideal{P}$. Then the inclusion $\ideal{P} \subseteq \ideal{P'}$ induces an isomorphism $K_0(\ideal{P}) \cong K_0(\ideal{P'})$.
\end{thm}

\begin{proof}
See \cite{Bass2}, Theorem (7.1).
\end{proof}

\begin{rem}
In our case, the canonical \O re set $S$ for $\iw{G}$ (See \S 2) is essentially contained in 
the center $\iw{\Gamma}$ (See \S 6, Lemma \ref{lem:locs0}), therefore 
Bass' localization sequence is sufficient for the proof of our main theorem.
\end{rem} 

%
\subsection{Theory of the integral logarithm}
%

Integral logarithmic homomorphisms were used by Robert Oliver and
Laurence R.\ Taylor
to study structure of Whitehead groups of group rings of finite groups
(\cite{Oliver}, \cite{Oli-Tay}). We use these homomorphisms 
to translate ``the additive theta map'' into ``the 
(multiplicative) theta map'' (See \S 5). Ritter and Weiss also used 
them to formulate their ``equivariant Iwasawa theory'' (\cite{R-W1}).

Let $R$ be a complete discrete valuation ring with mixed characteristics
$(0,p)$, $K$ its fractional field, $\pi$ a uniformizer of $R$ and 
$k$ the residue field of $R$. 
Let $\ideal{A}$ be an 
``$R$-order,'' that is, an $R$-algebra which is free as a left $R$-module (though Oliver and Taylor treated only $\Zp$-orders in \cite{Oliver} and \cite{Oli-Tay},  we need integral logarithms for the $\comp{\iw{\Gamma}_{(p)}}$-order 
$\comp{\iw{\Gamma}_{(p)}}[G^f]$ to construct the localized version of
the theta map (see \S\, 6), therefore we introduce here the generalized version of integral
logarithms for general $R$-orders).

\begin{lem}[\cite{Oliver}, Lemma 2.7 for $\Zp$-orders] \label{lem:log}
Let $R$, $K$, $\pi$, $k$ and $\ideal{A}$ be as above, and let $J$ be the Jacobson radical
 of $\ideal{A}$. Then for every two-sided ideal $\ideal{a}\subseteq J$, $1+\ideal{a}$ is a 
multiplicative group. Moreover, 
for an arbitrary element $x\in \ideal{a}$, the power series
\begin{equation} \label{eq:log}
\log (1+x) = \sum_{i=1}^{\infty} (-1)^{i-1}\dfrac{x^i}{i}
\end{equation}
converges $p$-adically in $\ideal{a}_{\Q}=\Q \otimes_{\Z} \ideal{a}$, and satisfies
\begin{equation} \label{eq:log+}
\log ((1+x)(1+y)) \equiv \log(1+x)+\log(1+y) \qquad
 \mod{[\ideal{A}_{\Q},\ideal{a}_{\Q}]}
\end{equation}
for every $x,y \in \ideal{a}$ where $\ideal{A}_{\Q}=\Q \otimes_{\Z} \ideal{A}$.

In particular, $\log$ induces a homomorphism of groups
\begin{equation*}
\log\colon 1+\ideal{a} \longrightarrow \ideal{a}_{\Q}/[\ideal{A}_{\Q}, \ideal{a}_{\Q}].
\end{equation*}
\end{lem}

Here for arbitrary two-sided ideals $\ideal{a}$ and $\ideal{b}$ of $R$, $[\ideal{a},\ideal{b}]$ denotes the two-sided ideal generated by $[a,b]=ab-ba$ where $a\in \ideal{a}$  and $b\in \ideal{b}$.

\begin{proof}
Since $\ideal{A}/\ideal{A}\pi$ is finite over $k=R/R\pi$, it is a left 
Artinian $k$-algebra. Note that $J/\ideal{A}\pi$ is the Jacobson radical of 
$\ideal{A}/\ideal{A}\pi$, and by commutative ring theory, the Jacobson radical 
of an arbitrary (left) Artinian ring is nilpotent. Hence for an arbitrary $x\in \ideal{a} \subseteq J$, there exists a certain natural number $m$ such that 
\begin{equation} \label{eq:xmpi}
x^m \in \ideal{A}\pi.
\end{equation} 

Let $e$ be the absolute ramification index of $K$. Then we have $x^n \in  \ideal{A} p^{[n/em]}$ for each $n\geq em$ where $[r]$ denotes the greatest integer not greater than $r$ for every $r\in \mathbb{R}$. $\ideal{A} p^{[n/em]}$ converges to zero as $n\rightarrow \infty$, therefore we obtain
\begin{equation*} 
\lim_{n\rightarrow \infty} \left| x^n \right|_p=0,
\end{equation*}
which implies the power series
\begin{equation*}
(1+x)^{-1} = \sum_{i=0}^{\infty} (-1)^i x^i
\end{equation*}
converges in $\ideal{A}$. Since $\ideal{a}$ is closed in $p$-adic topology, we may conclude that $(1+x)^{-1}\in 1+\ideal{a}$. It is clear that $1+\ideal{a}$ is closed under multiplication (note that $\ideal{a}$ is a two-sided ideal), therefore $1+\ideal{a}$ is a multiplicative group.

Now we also have $x^n/n \in  \ideal{A} \cdot (p^{[n/em]}/n)$ for each $n\geq em$ by (\ref{eq:xmpi}), therefore we obtain
\begin{equation}
\lim_{n\rightarrow \infty} \left| \dfrac{x^n}{n} \right|_{p}=0,
\end{equation}
which implies that the power series (\ref{eq:log}) converges in $\ideal{A}_{\Q}$. 
Then we can show that (\ref{eq:log}) converges in $\ideal{a}_{\Q}$ by the same argument as above.

We may also show the equation (\ref{eq:log+}) by direct calculation, 
but this calculation is quite complicated because of the
 non-commutativity of $\ideal{A}$. See \cite{Oliver}, Chapter 2, Lemma
 2.7.
\end{proof}

\begin{rem}
Note that though Oliver constructed $\log$ only for $\Zp$-orders in
 \cite{Oliver}, Lemma 2.7 of Chapter 2,
 his proof does not use peculiarities of $\Zp$ (especially the finiteness of the residue field of $R$). So we can apply his proof to the case of general $R$-orders.
\end{rem}

\begin{prop}[\cite{Oliver}, Theorem 2.8 for $\Zp$-orders] \label{prop:logk1}
Let $\ideal{A}$ and $\ideal{a}$ be as above. Then the logarithmic homomorphism of Lemma $\ref{lem:log}$ induces the following homomorphism of abelian groups$\colon$
\begin{equation*}
\log_\ideal{a} \colon K_1(\ideal{A},\ideal{a}) \longrightarrow \ideal{a}_{\Q}/[\ideal{A}_{\Q},\ideal{a}_{\Q}].
\end{equation*}
\end{prop}

\begin{proof}[Sketch of the proof]
For every $n\geq 1$ and for an ideal $\M_n(\ideal{a})\subseteq M_n(\ideal{A})$, we have a logarithmic 
homomorphism (Lemma \ref{lem:log})
\begin{align*}
\log^{(n)} \colon \GL_n(\ideal{A},\ideal{a})=I_n+\M_n(\ideal{a}) \xrightarrow{\log}
 \M_n(\ideal{a}_{\Q})/& [\M_n(\ideal{A}_{\Q}), \M_n(\ideal{a}_{\Q})] \\ 
 &\xrightarrow{\mathrm{trace}} \ideal{a}_{\Q} / [\ideal{A}_{\Q}, \ideal{a}_{\Q}]
\end{align*}
where $I_n$ is the unit matrix of $\GL_n(\ideal{A})$.

For every $X \in \GL_n(\ideal{A})$ and $A\in \GL_n(\ideal{A},\ideal{a})=I_n+\M_n(\ideal{a})$, we have
\begin{align*}
\log^{(n)}([X,A]) &= \log^{(n)}(XAX^{-1})- \log^{(n)}(A) \qquad
 (\text{by (\ref{eq:log+})}) \\
&= \Tr (X\log(A)X^{-1})-\Tr(\log(A)) \\
&=0.
\end{align*}

Hence $\log^{(n)}$ induces 
\begin{equation*}
\log^{(n)} \colon \GL_n(\ideal{A})/[\GL_n(\ideal{A}),\GL_n(\ideal{A},\ideal{a})] 
\longrightarrow \ideal{a}_{\Q}/[\ideal{A}_{\Q},\ideal{a}_{\Q}],
\end{equation*}
and by taking the projective limit, we obtain 
\begin{equation*}
\log_\ideal{a} \colon K_1(\ideal{A},\ideal{a}) \longrightarrow \ideal{a}_{\Q}/[\ideal{A}_{\Q},\ideal{a}_{\Q}],
\end{equation*}
using Whitehead's lemma (Proposition \ref{prop:whitehead}).

(See \cite{Oliver}, Theorem 2.8.)
\end{proof}

\begin{rem} \label{rem:logext}
Let $R$ be the integer ring of a finite extension of $\Qp$, and take
 $\ideal{a}$ to be the Jacobson radical $J$ of $\ideal{A}=R[G]$ where $G$ is a
 finite group. Then we may extend the domain of the logarithmic homomorphism obtained in Proposition \ref{prop:logk1} to $K_1(R[G])$ uniquely. 

The following exact sequence of $K$-groups is well known:
\begin{equation*}
K_2(R[G]/J) \rightarrow K_1(R[G],J) \rightarrow K_1(R[G])
 \rightarrow K_1(R[G]/J) \rightarrow 1
\end{equation*}
(See \cite{Milnor} Lemma 4.1 and Theorem 6.2). Note that the homomorphism $K_1(R[G])
\rightarrow K_1(R[G]/J)$ is surjective by Proposition
\ref{prop:surjonk1}. Since $R[G]/J$ is a
finite semi-simple algebra with $p$-power order, $K_2(R[G]/J)=1$
and $p \nmid a=\sharp K_1(R[G]/J)$ (\cite{Oliver}, Theorem 1.16). 
Therefore, we may regard $K_1(R[G],J)$ as a subgroup of
$K_1(R[G])$, and may extend the domain of $\log_J$ to $K_1(R[G])$ by setting
\begin{equation*}
\log_{R[G]} \phi =\frac{1}{a} \log_J \phi^a
\end{equation*}
for every $\phi \in K_1(R[G])$ (note that $\phi^a \in K_1(R[G],J)$).
The uniqueness of the extension is trivial from this construction.
\end{rem}

Now we define {\em integral logarithmic homomorphisms}. Let $R$ be a complete discrete valuation ring which is absolutely unramified, and let $G$ be a finite $p$-group. 
For simplicity, we assume that $p\neq 2$. We consider a group ring $R[G]$. Let $J$ be the Jacobson radical of $R[G]$. Let 
\begin{align*}
R[\conj{G}]&=R[G]/[R[G],R[G]] & (\text{resp.\ }K[\conj{G}]&=K[G]/[K[G],K[G]])
\end{align*}
be the free $R$-module (resp.\ the $K$-vector space) with basis $\conj{G}$. 
In the following, we fix the Frobenius automorphism $\tilder{\varphi} \colon K \rightarrow K$ when $k=R/pR$ is not perfect. Then we have
\begin{align} \label{eq:frob}
\tilder{\varphi}(r) &\equiv r^p & \mod pR
\end{align}
for every $r\in R$.\footnote{If $R$ is the integer ring of a finite unramified extension of $\Qp$, $\tilder{\varphi}$ is the ordinary Frobenius 
endomorphism (determined uniquely up to the inertia subgroup). If $R$ is the $p$-adic completion of $\iw{\Gamma}_{(p)}$
(we use this ring in \S 6), we may choose $\tilder{\varphi}$ as the endomorphism induced by $t \mapsto t^p$.}

Let $\varphi \colon K[\conj{G}]\longrightarrow K[G]$ be the homomorphism 
defined by
\begin{equation*}
\varphi \left( \sum_g k_g[g] \right) =\sum_g \tilder{\varphi}(k_g) [g^p] 
\qquad (k_g\in K,[g]\in \conj{G}).
\end{equation*}

\begin{propdef}[The integral logarithms] \label{pd:intlog}
Set 
\begin{align*}
\Gamma_{G,J}(u) &=\log_J (u) - \frac{1}{p}\varphi \left( \log_J (u)\right) \in K[\conj{G}],  & u\in K_1(R[G],J).
\end{align*}

Then $\Gamma_{G,J}$ induces a homomorphism
\begin{equation*}
\Gamma_{G,J}\colon K_1(R[G],J) \longrightarrow R[\conj{G}],
\end{equation*}
{\upshape which we call} the integral logarithmic homomorphism for $R[G]$.

When $K$ is a finite unramified extension of $\Qp$ and $R$ is its
 integer ring, then 
\begin{align*}
\Gamma_{G}(u) &=\log_{R[G]} (u) - \frac{1}{p}\varphi \left( \log_{R[G]} (u)\right) \in K[\conj{G}],  & u\in K_1(R[G])
\end{align*}
also induces a homomorphism
\begin{equation*}
\Gamma_{G}\colon K_1(R[G]) \longrightarrow R[\conj{G}].
\end{equation*}
\end{propdef}

\begin{proof}[Sketch of the proof]
Take an arbitrary $x\in J$. Since 
\begin{align*}
\Gamma_{G,J}(1-x) &= -\left( x+\frac{x^2}{2}+\cdots
 \right)+\frac{1}{p}\varphi \left( x+\frac{x^2}{2}\cdots \right) \\
 &\equiv -\sum_{i=1}^{\infty} \frac{1}{pi}\left( x^{pi}-\frob{x^i}
 \right) \qquad \mod{R[\conj{G}]},
\end{align*}
it is sufficient to show that $pi| (x^{pi}-\frob{x^i})$ for every $i\geq 1$, or $p^n | (x^{p^n}-\frob{x^{p^{n-1}}})$ for every $n\geq 1$ 
(note that all primes other than $p$ are invertible in $R$).

To do this, we should check each term of the expansion of $x^{p^n}$ and
 $\frob{x^{p^{n-1}}}$ carefully. See \cite{Oliver}, Theorem 6.2 for details.

The second part follows immediately by the first part and the uniqueness of the extension of $\log$ (Remark \ref{rem:logext}).
\end{proof} 

Now assume that $R$ is the $p$-adic integer ring $\Zp$ and $G$ is a
finite $p$-group. 
In this case we can derive further information about the kernel and image of the integral logarithms.

\begin{thm} \label{thm:intlogstr}
Let $G$ be a
finite $p$-group and $K_1(\Zp[G])_{\mathrm{tors}}$ the torsion part of $K_1(\Zp[G])$. Then there exists an exact sequence 
\begin{equation*}
1 \rightarrow K_1(\Zp[G])/K_1(\Zp[G])_{\mathrm{tors}}
 \xrightarrow{\Gamma_G} \Zp[\conj{G}] \xrightarrow{\omega_G}
 G^{\mathrm{ab}} \rightarrow1.
\end{equation*}

Here $\omega_G$ is defined by 
\begin{equation*}
\omega_G \left( \sum_g a_g[g] \right) = \prod_g \line{g}^{a_g} \qquad (a_g\in \Zp,
 [g]\in \conj{G})
\end{equation*}
and $G^{\mathrm{ab}}$ is the abelization of $G$. We denote by $\line{g}$ the image of $[g]$ in $G^{\mathrm{ab}}$.
\end{thm}

\begin{proof}
See \cite{Oliver}, Theorem 6.6.
\end{proof}

For the torsion part of $K_1(\Zp[G])$, the following properties are known.

\begin{defn}[{$SK_1(\Zp[G])$}]
Let $G$ be a finite group. Then we put
\begin{equation*}
SK_1(\Zp[G]) = \Ker \left( K_1(\Zp[G]) \longrightarrow K_1(\Qp[G]) \right).
\end{equation*}
\end{defn}

\begin{rem}
We may define the $SK_1$-group for an arbitrary associative ring by Whitehead construction, 
but we omit this since we only use $SK_1$-groups for group rings of $\Zp[G]$-type.
\end{rem}

\begin{prop} \label{prop:sk1}
Let $G$ be a finite group.
\begin{enumerate}[$(1)$]
\item $K_1(\Zp[G])_{\mathrm{tors}}\cong \mu_{p-1} \times G^{\mathrm{ab}} \times SK_1(\Zp[G])$.
\item $SK_1(\Zp[G])$ is finite.
\item If $G$ is commutative, $SK_1(\Zp[G])=1$.
\end{enumerate}
\end{prop}

\begin{proof}
\begin{enumerate}
\item See \cite{Oliver}, Theorem 7.4.
\item See \cite{Wall}, Theorem 2.5.
\item This follows from the definition of $SK_1$-groups:
 since $\Zp[G]$ and $\Qp[G]$ are local and commutative, we have $K_1(\Zp[G]) \cong \Zp[G]^{\times}$ and $K_1(\Qp[G]) \cong \Qp[G]^{\times}$ by Proposition \ref{prop:sloc} (1), then $\Zp[G]^{\times} \rightarrow \Qp[G]^{\times}$ is obviously injective.
\end{enumerate}
\end{proof}

%
%
\section{Basic results of non-commutative Iwasawa theory}
%
%

In this section, we review basic results of \cite{CFKSV}.

Let $G$ be a compact $p$-adic Lie group containing a normal closed subgroup $H$ which satisfies $G/H \cong \Zp$, and let $\iw{G}$ be its Iwasawa algebra. In the non-commutative Iwasawa theory of \cite{CFKSV}, characteristic
elements (the ``arithmetic'' $p$-adic zeta functions) are defined
as elements of Whitehead groups of certain \O re localization of $\iw{G}$ by using the localization exact sequence. We first define {\it the canonical \O re set} $S$ for the group $G$ introduced in \cite{CFKSV} \S 2. Then for every element $[\ideal{C}]$ of the relative $K_0$-group $K_0(\iw{G}, \iw{G}_S)$ (in the following, we regard $K_0(\iw{G},\iw{G}_S)$ as the Grothendieck group of the category
 of bounded complexes of finitely generated projective left $\iw{G}$-modules 
whose cohomologies are of $S$-torsion), we define {\it 
a characteristic element of $[\ideal{C}]$} as an element of
$K_1(\iw{G}_S)$. Next, we characterize the $p$-adic zeta function (the
``analytic'' element) as an element of $K_1(\iw{G}_S)$
interpolating special values of Artin $L$-functions, and formulate the 
non-commutative Iwasawa main conjecture.

%
\subsection{The canonical \O re set and characteristic elements}
%

Let $G$ be a compact $p$-adic Lie group and 
\begin{equation*}
\iw{G}=\varprojlim_{U\colon\text{open normal}} \Zp[G/U]
\end{equation*}
be its Iwasawa algebra. 
Suppose that there exists a normal closed subgroup $H$ of $G$ which satisfies 
$G/H \cong \Gamma$, where $\Gamma$ is a commutative $p$-adic Lie group 
isomorphic to $\Zp$ (See \S 0.3). In the following, we fix such a subgroup $H$. 

\begin{propdef}[The canonical \O re set]
Let $G$ and $H$ be as above. Then 
\begin{equation*}
S= \{ f\in \iw{G} \mid \iw{G}/\iw{G}f \text{ is finitely generated as a left $\iw{H}$-module} \}
\end{equation*}
is a left and right \O re set $($See Definition $\ref{def:ore})$. 

{\upshape We call this $S$} the canonical \O re set for the group $G$.\footnote{Since the definition of $S$ depends on the subgroup $H$, we should call $S$ {\it the canonical \O re set for $(G,H)$}. But by abuse of notation, we often call $S$ the canonical \O re set for $G$ when the subgroup $H$ is fixed.}
\end{propdef}

Since $S$ is a left and right \O re set, the left localization and the right 
localization of $\iw{G}$ with respect to $S$ are canonically isomorphic
to each other by Corollary \ref{cor:isom}. Therefore we may identify
these two localizations. Because of this reason, we denote by $\iw{G}_S$ the left (or right) localization of $\iw{G}$ with respect to $S$. 

The canonical \O re sets and their properties are discussed well in
\cite{CFKSV}, \S 2. Here we use the following two properties:

\begin{prop} \label{prop:props}
Let $G$ be a compact $p$-adic Lie group and $S$ the canonical \O re set
 for it. 
\begin{enumerate}[$(1)$]
\item The localized Iwasawa algebra $\iw{G}_S$ is semi-local.
\item The elements of $S$ are non-zero divisors in $\iw{G}$.
\end{enumerate}
\end{prop}

\begin{proof}
\begin{enumerate}[(1)]
\item See \cite{CFKSV}, Proposition 4.2.
\item See \cite{CFKSV}, Theorem 2.4.
\end{enumerate}
\end{proof}

Note that we have the canonical surjection $\iw{G}_S^{\times}
\rightarrow K_1(\iw{G}_S)$ by Proposition \ref{prop:sloc} (1) and
Proposition \ref{prop:props} (1). Also note that by Proposition
\ref{prop:props} (2), we may consider the localization sequence for the
\O re localization $\iw{G} \rightarrow \iw{G}_S$ (Theorem \ref{thm:locseq}).

\begin{rem} \label{rem:S*}
Assume that $G$ has no $p$-torsion elements. In this case we often use
\begin{equation*}
S^* = \bigcup_{n\geq 0} p^n S
\end{equation*}
in place of $S$, as is remarked in \cite{CFKSV}. We may identify 
$K_0(\iw{G}, \iw{G}_{S^*})$ with $K_0(\ideal{M}_{S^*}(\iw{G}))$
where $\ideal{M}_{S^*}(\iw{G})$ is the category of finitely generated
 left $S^*$-torsion $\iw{G}$-modules.

But in our case (See \S \, 3), $G$ has many $p$-torsion elements. Therefore 
we have to treat derived categories of complexes of finitely generated modules.
\end{rem}

Now let us consider the localization exact sequence (Theorem \ref{thm:locseq}) for
$\iw{G}\rightarrow \iw{G}_S$:

\begin{equation} \label{eq:loc-seq}
K_1(\iw{G}) \rightarrow K_1(\iw{G}_S) \xrightarrow{\del} K_0(\iw{G}, \iw{G}_S)
 \rightarrow K_0(\iw{G}) \rightarrow K_0(\iw{G}_S).
\end{equation}

\begin{prop} \label{prop:delsurj}
The connecting homomorphism $\del$ in $(\ref{eq:loc-seq})$ is surjective.
\end{prop}

Therefore, for an arbitrary element $[\mathfrak{C}] \in
K_0(\iw{G},\iw{G}_S)$,
there exists an element $f\in K_1(\iw{G}_S)$ which satisfies 
$\del(f)=-[\mathfrak{C}]$.

\begin{defn}[Characteristic element]
Let $[\mathfrak{C}]\in K_0(\iw{G},\iw{G}_S)$. 
We call an element $f$ of $K_1(\iw{G}_S)$ {\it a characteristic 
element of $[\mathfrak{C}]$} if $f$ satisfies $\del(f)=-[\mathfrak{C}]$.
\end{defn}

\begin{rem} 
By the localization exact sequence, characteristic elements of
 $[\mathfrak{C}]$ are determined up to multiplication by elements of $K_1(\iw{G})$.
\end{rem}

\begin{proof}[Proof of Proposition $\ref{prop:delsurj}$]

We may prove this proposition by the almost same argument as the 
proof of Proposition 3.4 in \cite{CFKSV}.

Though we assume that $G$ has no $p$-torsion elements in the proof of Proposition
 3.4 in \cite{CFKSV}, this assumption is used only to avoid 
treating complexes directly (this point is also remarked in
 \cite{CFKSV}). For each complex 
$\mathfrak{C}=\mathfrak{C}^{\cdot}$, its canonical image in $K_0(\iw{G})$ is
 $\displaystyle \sum_{i\in \Z} (-1)^i[\mathfrak{C}^i]$. By using this fact and translating $\lambda, \tau, \varepsilon$ in
 \cite{CFKSV} appropriately, we may apply their proof of Proposition 3.4
 in \cite{CFKSV} to the case where $G$
 has $p$-torsion elements.  
\end{proof}

\begin{ex} \label{ex:iwasawa}
Let us consider the classical Iwasawa $\Zp$-extension case:
 $G=\Gamma, H=\{ 1\}$ and $\iw{\Gamma} \cong \Zp[[T]]; t \mapsto 1+T$.

In this case, by $p$-adic Weierstrass' preparation theorem, we have 
\begin{equation*}
f(T)=u(T)p^n g(T)
\end{equation*}
for an arbitrary non-zero element $f(T) \in \Zp[[T]]$, where $u(T) \in
 \Zp[[T]]^{\times}$, $n\in \Z_{\geq 0}$, and $g\in \Zp[T]$ is a
 distinguished polynomial. Hence $\iw{\Gamma}/\iw{\Gamma}f$ is finitely
 generated over $\Zp$ if and only if $p\nmid f$ (or equivalently $n=0$). Therefore the canonical \O re set $S$ is $\iw{\Gamma}\setminus p\iw{\Gamma}$.

Since $\Gamma$ has no $p$-torsion elements, we may consider the \O re set $S^*$
as in Remark \ref{rem:S*}. It is easy to see that in this case 
 $S^*$ is the subset of $\iw{\Gamma}$ consisting of all non-zero
 elements. Therefore $\iw{\Gamma}_{S^*}$ coincides with the fractional field $\mathrm{Frac}(\iw{\Gamma})$. 

Furthermore, since $\iw{\Gamma}$ is local and commutative, we have 
\begin{equation*}
K_1(\iw{\Gamma})=\iw{\Gamma}^{\times} \quad \text{and} \quad K_1(\iw{\Gamma}_{S^*})=\mathrm{Frac}(\iw{\Gamma})^{\times}.
\end{equation*}

Hence the localization exact sequence is described as follows:
\begin{equation*}
\iw{\Gamma}^{\times} \rightarrow \mathrm{Frac}(\iw{\Gamma})^{\times} \xrightarrow{\del} 
K_0(\ideal{M}_{\mathrm{tor}}(\iw{\Gamma})) \rightarrow 0,
\end{equation*}
where $\ideal{M}_{\mathrm{tor}}(\iw{\Gamma})$ is the category of all
 finitely generated torsion $\iw{\Gamma}$-modules. The connecting
 homomorphism $\del$ is characterized by
 $\del(f)=[\iw{\Gamma}/\iw{\Gamma}f]$ for $f\in \iw{\Gamma} \setminus \{
 0\}$ in this (abelian) case.

Let $M$ be an arbitrary finitely generated torsion
 $\iw{\Gamma}$-module. Then by the famous structure theorem of finitely
 generated torsion $\iw{\Gamma}$-modules, there exist $f_i\in \iw{\Gamma} \setminus \{ 0\} \,
 (1\leq i \leq N)$ and satisfy
\begin{equation*}
M \sim \bigoplus_{i=1}^N \iw{\Gamma}/\iw{\Gamma}f_i \qquad (\text{pseudo-isomorphic})
\end{equation*}
where the image of each $f_i$ in $\Zp[[T]]$ is a non-invertible
 polynomial. Since the image of every pseudo-null $\iw{\Gamma}$-module in
 $K_0(\ideal{M}_{\mathrm{tor}}(\iw{\Gamma}))$ vanishes (See
 \cite{Sch-Ven}), we have
\begin{equation*}
[M]=\sum_{i=1}^N [\iw{\Gamma}/\iw{\Gamma}f_i].
\end{equation*}

Hence, by the explicit description of $\del$, we have 
\begin{equation*}
\del (f_M)=[M]
\end{equation*}
where $f_M=\prod_{i=1}^N f_i$ ({\it the characteristic polynomial} of
 $M$). $f_M$ is determined up to multiplication by an element of $\iw{\Gamma}^{\times}$.

The calculation above shows that the characteristic elements in \cite{CFKSV} are generalized notion of the classical characteristic polynomials. 
\end{ex}

%
\subsection{The non-commutative Iwasawa main conjecture for totally real
  fields}
%

Now let us review the formulation of the non-commutative Iwasawa main conjecture in the sense of \cite{CFKSV}.

Fix a prime number $p\neq 2$. Let F be a totally real number field and
$F^{\infty}/F$ a Galois extension of infinite degree satisfying 
the following conditions:\footnote{Since the Galois extension we 
will consider has $p$-torsion elements, we need to weaken the conditions in \cite{CFKSV}. See \cite{Kato1} \S 2 and \cite{Fu-Ka} \S 4.3.}

\begin{enumerate}
\item The Galois group $G=\gal{F^{\infty}/F}$ is a compact $p$-adic Lie group.
\item Only finitely many primes of $F$ ramify in $F^{\infty}$.
\item $F^{\infty}$ is totally real and contains the cyclotomic
      $\Zp$-extension $F^{\mathrm{cyc}}$ of $F$.
\end{enumerate}

Fix a finite set $\Sigma$ of primes of $F$ containing all primes which ramify in $F^{\infty}$.

\begin{defn} \label{def:selmer}
Under the conditions above, we define the complex $C$ by
\begin{align*}
C &=C_{F^{\infty}/F} \\
&= R\mathrm{Hom}(R\Gamma_{\mathrm{\acute{e}t}}(\mathrm{Spec}\, \integer{F^{\infty}}[1/\Sigma],
 \Qp/\Zp),\Qp/\Zp).
\end{align*}
\end{defn}

Here $\Gamma_{\mathrm{\acute{e}t}}$ is the global section functor for
\'etale topology.

Note that $H^0(C)=\Zp$, $H^{-1}(C)=\gal{M_{\Sigma}/F^{\infty}}$ where $M_{\Sigma}$ is the maximal abelian pro-$p$ extension of $F^{\infty}$ unramified outside $\Sigma$, and $H^n(C)=0$ for $n\neq 0,-1$. We denote $\gal{M_{\Sigma}/F^{\infty}}$ by $X_{\Sigma}(F^{\infty}/F)$.

Now set $H=\gal{F^{\infty}/F^{\mathrm{cyc}}}$ and $\Gamma
=\gal{F^{\mathrm{cyc}}/F}\cong \Zp$. Then if $[C]$ is an element of 
$K_0(\iw{G},\iw{G}_S)$, we can apply the results of \S 2.1 to $[C]$ and
obtain {\em a characteristic element for $F^{\infty}/F$} as a
characteristic element of $[C]$.

\begin{con} \label{con:selmer}
$[C]$ is always an element of $K_0(\iw{G},\iw{G}_S)$. In other words, 
$X_{\Sigma}(F^{\infty}/F)$ is of $S$-torsion.
\end{con}

\begin{prop} \label{prop:selmer}
Let $G' \subseteq G$ be a pro-$p$ subgroup of $G$ and let $F'$ be the maximal intermediate field of $F^{\infty}/F$ fixed by $G'$, then the followings are equivalent$\colon$
\begin{enumerate}[$(1)$]
\item $X_{\Sigma}(F^{\infty}/F)$ is of $S$-torsion.
\item $\mu({F'}^{\mathrm{cyc}}/F')=0$ where $\mu$ is the $\mu$-invariant.
\end{enumerate}

In particular, if the following condition $(\ast)$ is satisfied, $X_{\Sigma}(F^{\infty}/F)$ is of $S$-torsion$\colon$
\begin{quote}
$(\ast)$ \quad There exists a finite subextension $F'$ of $F^{\infty}$ such that $\gal{F^{\infty}/F'}$ is pro-$p$ and $\mu({F'}^{\mathrm{cyc}}/F')=0$.
\end{quote}
\end{prop}

\begin{proof}
This proposition is a variant of \cite{Ha-Sh}, Lemma 3.4.
\end{proof}

For the $\mu$-invariants of cyclotomic $\Zp$-extensions, Kenkichi Iwasawa conjectured:

\begin{con}[Iwasawa's $\mu=0$ conjecture]
For every number field $K$, $\mu(K^{\mathrm{cyc}}/K)=0$.
\end{con}

\begin{cor}
Assume that Iwasawa's $\mu=0$ conjecture is true, then $X_{\Sigma}(F^{\infty}/F)$ is always of $S$-torsion.
\end{cor}

\begin{cor}
Let $K/\mathbb{Q}$ be a finite abelian extension. Then $X_{\Sigma}(K^{\infty}/K)$ is of $S$-torsion.
\end{cor}

\begin{proof}
$\mu(K^{\mathrm{cyc}}/K)=0$ by Ferrero-Washington's theorem (\cite{FW}).
\end{proof}  

In the following, we always assume the condition $(\ast)$ in Proposition
\ref{prop:selmer}.

\medskip
Now we define the ``$p$-adic zeta function'' as an element of $K_1(\iw{G}_S)$. Let 
\begin{equation*}
\rho \colon G \longrightarrow \GL_d(\line{\Q}) \rightarrow \GL_d(\line{\Q}_p)
\end{equation*}
be an arbitrary Artin representation (that is, $\rho(G)\subseteq
\GL_d(\line{\Q})$ is a finite subgroup). Then their exists a finite
extension $E$ of $\Qp$ such that $GL_d(E)$ contains the image of $\rho$,
and $\rho$ induces a ring homomorphism
\begin{equation*}
\rho \colon \iw{G} \longrightarrow \M_d(E) 
\end{equation*}

This also induces a homomorphism of $K$-groups
\begin{equation*}
\mathrm{ev}_{\rho} \colon K_1(\iw{G}) \longrightarrow K_1(\M_d(E))
 \xrightarrow{\simeq} K_1(E) \cong E^{\times}
\end{equation*}
where the isomorphism $K_1(M_d(E)) \xrightarrow{\simeq} K_1(E)$ is given
by the Morita equivalence between $\M_d(E)$ and $E$. Composing this with
the natural inclusion $E^{\times} \rightarrow \line{\Q}_p^{\times}$, we
obtain the map 
\begin{equation*}
\mathrm{ev}_{\rho} \colon K_1(\iw{G}) \longrightarrow
 \line{\Q}_p^{\times}.
\end{equation*}
As is discussed in \cite{CFKSV} \S 2, this map can be extended to
\begin{equation*}
\mathrm{ev}_{\rho} \colon K_1(\iw{G}_S) \longrightarrow \line{\Q}_p \cup \{ \infty \}.
\end{equation*}

We call $\mathrm{ev}_{\rho}$ {\it the evaluation map at $\rho$}. We denote $\mathrm{ev}_{\rho}(f)$ by $f(\rho)$ for every $f\in K_1(\iw{G}_S)$.

Let 
\begin{equation*}
\kappa\colon \gal{F(\mu_{p^{\infty}})/F} \longrightarrow \Zp^{\times}
\end{equation*}
be the $p$-adic cyclotomic character. Then for every positive integer
$r$ divisible by $p-1$, $\kappa^r$ factors $\Gamma=\gal{F^{\mathrm{cyc}}/F}$.
Therefore we may extend the domain of $\kappa^r$ to $G$ by 
\begin{equation*}
G=\gal{F^{\infty}/F} \rightarrow \gal{F^{\mathrm{cyc}}/F} \xrightarrow{\kappa^r} \Zp^{\times}
\end{equation*}
where the first map is the canonical surjection.

\begin{defn}[$p$-adic zeta function] \label{def:zeta}
If $\xi_{F^{\infty}/F} \in K_1(\iw{G}_S)$ satisfies 
\begin{equation}
\xi_{F^{\infty}/F}(\rho \otimes \kappa^r)=L_{\Sigma}(1-r;F^{\infty}/F,\rho) \label{eq:interp}
\end{equation}
for every positive integer $r$ divisible by $p-1$ and for an arbitrary Artin representation $\rho$ of $G$, we call $\xi_{F^{\infty}/F}$ {\it the $p$-adic zeta function for $F^{\infty}/F$}.
\end{defn}

Here $L_{\Sigma}(s;F^{\infty}/F,\rho)$ is the complex Artin $L$-function
of $\rho$ in which the Euler factors at $\Sigma$ are removed.

\begin{con} \label{conj:iwasawa} 
Let $F^{\infty}/F$ be as above.
\begin{enumerate}[$(1)$]
\item $($The existence and the uniqueness of the $p$-adic zeta function$)$

The $p$-adic zeta function $\xi_{F^{\infty}/F}$ for $F^{\infty}/F$
      exists uniquely.

\item $($The non-commutative Iwasawa main conjecture$)$ 

The $p$-adic zeta function $\xi_{F^{\infty}/F}$ satisfies $\del (\xi_{F^{\infty}/F})=-[C_{F^{\infty}/F}]$.
\end{enumerate}
\end{con} 

\begin{rem}[The abelian case] \label{rem:abel}
Let $G=\gal{F^{\infty}/F}$ be an {\em abelian} $p$-adic Lie group. In this
 case, Coates observed that if certain congruences among the special
 values of the partial zeta functions were proven, we could construct the $p$-adic
 $L$-function for $F^{\infty}/F$ (See \cite{Coates}, Hypotheses $(H_n)$
 and $(C_0)$). These congruences were proven by Deligne and Ribet using
 the deep result about Hilbert-Blumenthal modular varieties
 (\cite{De-Ri}).

Using the Deligne-Ribet's congruences, Serre constructed the element 
$\xi_{F^{\infty}/F}$ of $\mathrm{Frac}(\iw{G})$ which satisfied the
 following two properties (Serre's $p$-adic zeta pseudomeasure
 \cite{Serre2} for $F^{\infty}/F$,
 also see \S \ref{sect:cong}).

\begin{enumerate}[(1)]
\item For an arbitrary element $g$ of $G$, $(1-g)\xi_{F^{\infty}/F}$ is
 contained in $\iw{G}$.
\item $\xi_{F^{\infty}/F}$ satisfies the interpolation property
      (\ref{eq:interp}).
\end{enumerate}
\end{rem}

\begin{rem}
The non-commutative Iwasawa main conjecture which we introduced here was established first by Coates, Fukaya, Kato, Sujatha and Venjakob for elliptic curves without complex multiplication (the $\GL_2$-conjecture) in \cite{CFKSV}. For more precise statements,
 see \cite{CFKSV}. 

In \cite{Fu-Ka}, Fukaya and Kato established the main conjecture for
 general cases and showed the compatibility of the main conjecture with
 the equivariant Tamagawa number conjecture. 
\end{rem}

%
%
\section{The main theorem and Burns' technique}
%
%

%
\subsection{The main theorem}
%

Fix a prime number $p$. We consider the one dimensional $p$-adic Lie 
group $G=G^f \times \Gamma$ where
\begin{equation*}
G^f =\begin{pmatrix} 1 & \Fp & \Fp & \Fp \\ 0 & 1 & \Fp & \Fp \\
0 & 0 & 1 & \Fp \\ 0 & 0 & 0 & 1 \end{pmatrix} \quad (\text{finite part
of $G$})
\end{equation*}
and $\Gamma$ is a commutative $p$-adic Lie group isomorphic to
 $\Zp$ (See \S 0.3). 

In the following, we fix generators of $G^f$ and denote them by
\begin{align*}
\alpha &= \begin{pmatrix}1 & 1 & 0 & 0 \\ 0 & 1 & 0 & 0 \\
0 & 0 & 1 & 0 \\  0 & 0 & 0 & 1 \end{pmatrix}, &
\beta &= \begin{pmatrix}1 & 0 & 0 & 0 \\ 0 & 1 & 1 & 0 \\
0 & 0 & 1 & 0 \\  0 & 0 & 0 & 1 \end{pmatrix}, \\
\gamma &= \begin{pmatrix}1 & 0 & 0 & 0 \\ 0 & 1 & 0 & 0 \\
0 & 0 & 1 & 1 \\  0 & 0 & 0 & 1 \end{pmatrix}, &
\delta &= \begin{pmatrix}1 & 0 & 1 & 0 \\ 0 & 1 & 0 & 0 \\
0 & 0 & 1 & 0 \\  0 & 0 & 0 & 1 \end{pmatrix}, \\
\varepsilon &= \begin{pmatrix}1 & 0 & 0 & 0 \\ 0 & 1 & 0 & 1 \\
0 & 0 & 1 & 0 \\  0 & 0 & 0 & 1 \end{pmatrix}, &
\zeta &= \begin{pmatrix}1 & 0 & 0 & 1 \\ 0 & 1 & 0 & 0 \\
0 & 0 & 1 & 0 \\  0 & 0 & 0 & 1 \end{pmatrix}.
\end{align*}

Then the center of $G^f$ is $\langle \zeta \rangle$ and there are four non-trivial fundamental relations of $G^f$:
\begin{align*}
[\alpha,\beta]&=\delta, &[\beta,\gamma]&= \varepsilon, \\
[\alpha,\varepsilon] &=\zeta, & [\delta,\gamma] &= \zeta,
\end{align*}
where $[x,y]=xyx^{-1}y^{-1}$ is the commutator of $x$ and $y$. Other commutators among $\alpha,\beta,\dotsc ,\zeta$ are $1$. We always denote the indices of $\alpha, \beta, \dotsc , \zeta$ 
by $a,b, \dotsc, f$. Note that $G$ satisfies the conditions in \S 2 for $H=G^f$. We also assume that $p \neq 2,3$. Under this assumption, the exponent of 
the group $G^f$ is $p$.

Let $F$ be a totally real number field and $F^{\infty}$ a Galois extension 
of $F$ satisfying $\gal{F^{\infty}/F}\cong G$ and the
condition $(\ast)$ in Proposition \ref{prop:selmer}, that is, there exists an intermediate field
$F^{\infty}/F'/F$ such that $F'/F$ is finite and the $\mu$-invariant $\mu \left(
(F')^{\mathrm{cyc}}/F' \right)$ equals to zero.

\begin{thm} \label{thm:maintheorem}
Under the notation and the assumptions as above, the $p$-adic zeta function $\xi_{F^{\infty}/F}$ for $F^{\infty}/F$ exists and the main conjecture $($Conjecture $\ref{conj:iwasawa}$ $(2))$ is true.
\end{thm}

\begin{rem}
In this paper we do not discuss the uniqueness of
 $\xi_{F^{\infty}/F}$. Also see Remark \ref{rem:venj}.
\end{rem}

%
\subsection{Burns' technique}
%

In this subsection, let $F^{\infty}/F$ be a general $p$-adic Lie extension of a totally real field $F$ satisfying the conditions in \S 2.2. Put $G=\gal{F^{\infty}/F}$.

Let $\ideal{F}$ be a family of pairs $(U,V)$ where $U$ is an open
subgroup of $G$ and $V$ is an open subgroup of $H$ such that $V$ is
normal in $U$ and $U/V$ is commutative. 

For $\ideal{F}$, we assume the following hypothesis:
\begin{quotation}
$(\flat)$ For an arbitrary Artin representation $\rho$ of $G$, $\rho$ is
 a $\Z$-linear combination of induced representations
 $\ind{G}{U_j}{\chi_j}$, as a virtual representation, where $(U_j, V_j)$
 is an element of $\ideal{F}$ and $\chi_j$ is a character of $U_j/V_j$
 of finite order.
\end{quotation}

In the following, we fix a family $\ideal{F}$ satisfying the hypothesis
$(\flat)$. For every $(U,V)\in \ideal{F}$, we have a homomorphism
\begin{equation*}
\theta_{U,V} \colon K_1(\iw{G}) \longrightarrow \iw{U/V}^{\times}
\end{equation*}
which is the composite of the norm map of $K$-groups (see \S \, 1.2)
\begin{equation*}
\Nr_{\iw{G}/\iw{U}} \colon K_1(\iw{G}) \longrightarrow K_1(\iw{U})
\end{equation*} 
and the canonical homomorphism
\begin{equation*}
K_1( \iw{U}) \longrightarrow K_1(\iw{U/V}) = \iw{U/V}^{\times}.
\end{equation*}

Let $S$ be the canonical \O re set for $G$. Then similarly we have a homomorphism
\begin{equation*}
\theta_{S,U,V} \colon K_1(\iw{G}_S) \longrightarrow \iw{U/V}_S^{\times}.
\end{equation*} 

Here we also denote the canonical \O re set for $U/V$ by the same symbol $S$ by abuse of notation.

Set 
\begin{equation*}
\theta=(\theta_{U,V})_{(U,V)\in \ideal{F}} \colon K_1(\iw{G}) \rightarrow \prod_{(U,V) \in \ideal{F}} \iw{U/V}^{\times}
\end{equation*}
and
\begin{equation*}
\theta_S=(\theta_{S,U,V})_{(U,V)\in \ideal{F}} \colon K_1(\iw{G}_S) \rightarrow \prod_{(U,V) \in \ideal{F}} \iw{U/V}_S^{\times}.
\end{equation*}

Let $\Psi_S$ be a subgroup
of $\prod_{(U,V)\in \ideal{F}} \iw{U/V}_S^{\times}$ and let
\begin{equation*}
\Psi=\Psi_S \cap \prod_{(U,V)\in \ideal{F}}
\iw{U/V}^{\times}.
\end{equation*}

\begin{defn}[The theta map, \cite{Kato1} \S 2.4] \label{def:theta}
Let $G$, $\ideal{F}$, $\theta_S$, $\theta$, $\Psi_S$ and $\Psi$ be as above. 

If $\theta$ and $\theta_S$ satisfy 
\begin{enumerate}[($\theta$-1)]
\item $\mathrm{Image}(\theta_S) \subseteq \Psi_S$,
\item $\mathrm{Image}(\theta)=\Psi$,
\end{enumerate}
we call the induced surjective homomorphism
\begin{equation*}
\theta \colon K_1(\iw{G}) \longrightarrow \Psi
\end{equation*}
{\em the theta map for the group $G$}, and call the induced homomorphism
\begin{equation*}
\theta_S\colon K_1(\iw{G}_S) \longrightarrow \Psi_S
\end{equation*}
{\em the localized theta map for the group $G$}.
\end{defn}

For $(U,V)\in \ideal{F}$, let $F_U$ (resp. $F_V$) be the maximal
subgroup of $F^{\infty}$ fixed by $U$ (resp. $V$). Since
$\gal{F_V/F_U}\cong U/V$ is abelian, the $p$-adic zeta function
(pseudomeasure) $\xi_{U,V}\in \mathrm{Frac}(\iw{U/V})$ for $F_V/F_U$
(see Remark \ref{rem:abel}) exists. 

\begin{thm}[Burns, \cite{Kato1} Proposition 2.5] \label{thm:ak-pri}

Let $G$ be a compact $p$-adic Lie group. Assume that the theta map
 $\theta$ and
 its localized version $\theta_S$ for $G$ exist, and also assume that
 $(\xi_{U,V})_{(U,V)\in \ideal{F}}$  is contained in $\Psi_S$. 

Then the $p$-adic zeta function $\xi_{F^{\infty}/F}$ for $F^{\infty}/F$
 exists and satisfies the main conjecture.

Moreover, if $\theta$ is injective, $\xi_{F^{\infty}/F}$ is determined uniquely.
\end{thm}

\begin{proof}
Let $f \in K_1(\iw{G}_S)$ be an arbitrary characteristic element for
 $F^{\infty}/F$. Significantly, $f$ is an element of $K_1(\iw{G}_S)$ which satisfies $\del(f)=-[C_{F^{\infty}/F}]$. 

Put $f_{U,V}=\theta_{S,U,V}(f)$ and $[C_{U,V}]= \theta_{U,V}([C])$ for each $(U,V)\in \ideal{F}$. 
Set $u_{U,V}=\xi_{U,V}f_{U,V}^{-1}$. Since the Iwasawa main conjecture
 is true for abelian extensions of totally real
 fields (\cite{Wiles}), we have $\del (\xi_{U,V})=-[C_{U,V}]$. Hence we
 obtain $\del (u_{U,V})=0$. By the localization exact sequence (Theorem \ref{thm:locseq}), we have $u_{U,V} \in \iw{U/V}^{\times}$.

On the other hand, since $(f_{U,V})_{(U,V)\in \ideal{F}} \in
 \mathrm{Image}(\theta_S) \subseteq \Psi_S$ by the condition
 ($\theta$-1) and $(\xi_{U,V})_{(U,V)\in \ideal{F}} \in \Psi_S$ by
 assumption, $(u_{U,V})_{(U,V)\in \ideal{F}}$ is an element of $\Psi_S$. Therefore by the definition of $\Psi$, we have 
\begin{equation*}
(u_{U,V})_{(U,V)\in \ideal{F}} \in \Psi=\Psi_S \cap \prod_{(U,V)\in
 \ideal{F}} \iw{U/V}^{\times}.
\end{equation*}

By the condition ($\theta$-2), there exists an element $u$ of $K_1(\iw{G})$ which satisfies $\theta_{U,V}(u)=u_{U,V}$. We denote the image of $u$ in $K_1(\iw{G}_S)$ by the same symbol $u$. Put $\xi_{F^{\infty}/F}=uf$. Since $u$ is an element of $K_1(\iw{G})$, we have $\del(u)=0$. Therefore
\begin{equation*}
\del(\xi_{F^{\infty}/F})=\del(uf)=\del(u)+\del(f)=\del(f)=-[C].
\end{equation*}
This implies that $\xi_{F^{\infty}/F}$ is also a characteristic element of $[C]$.

By the construction of $\xi_{F^{\infty}/F}$, it is clear that
\begin{equation} \label{eq:akzeta}
\theta_{S,U,V}(\xi_{F^{\infty}/F})=\xi_{U,V}.
\end{equation}

Finally, we prove that $\xi_{F^{\infty}/F}$ satisfies the interpolation
 property (\ref{eq:interp}). Since $\xi_{U,V}$ is the $p$-adic zeta function for $F_V/F_U$, it satisfies the following interpolating property:
\begin{equation} \label{eq:interpolate}
\xi_{U,V}(\chi \otimes \kappa^r)=L_{\Sigma} (1-r;F_V/F_U,\chi) \quad
 \text{for $r\in \N, \, (p-1) \mid r$}
\end{equation}
where $\chi$ is an arbitrary character of finite index of $U/V$.

By the hypothesis $(\flat)$, an arbitrary Artin representation $\rho$ is written in the form 
\begin{equation*}
\rho=\sum_{(U,V)\in \ideal{F}} \sum_{i\in I_{U,V}} a_{U,V}^{(i)} \ind{G}{U}{\chi_{U,V}^{(i)}}
\end{equation*}
where $I_{U,V}$ is a index set, $a_{U,V}^{(i)}$ is an integer (assume all but
 finitely many $a_{U,V}^{(i)}$ are zero), and $\chi_{U,V}^{(i)}$ is a character of finite index of $U/V$. Then we have

\begin{align*}
\xi_{F^{\infty}/F} (\rho \otimes \kappa^r) &= \prod_{(U,V)\in
 \ideal{F}} \prod_{i\in I_{U,V}} \xi(\ind{G}{U}{\chi_{U,V}^{(i)}} \otimes \kappa^r)^{a_{U,V}^{(i)}} && \\
&= \prod_{(U,V)\in \ideal{F}} \prod_{i\in I_{U,V}} \theta_{S,U,V}(\xi) (\chi_{U,V}^{(i)} \otimes \kappa^r)^{a_{U,V}^{(i)}}  && \\
&= \prod_{(U,V) \in \ideal{F}} \prod_{i\in I_{U,V}} \xi_{U,V} (\chi_{U,V}^{(i)} \otimes \kappa^r)^{a_{U,V}^{(i)}} && (\text{by (\ref{eq:akzeta})}) \\
&= \prod_{(U,V) \in \ideal{F}} \prod_{i\in I_{U,V}} L_{\Sigma}(1-r; F_V/F_U, \chi_{U,V}^{(i)})^{a_{U,V}^{(i)}} && (\text{by (\ref{eq:interpolate})}) \\
&= L_{\Sigma}\left (1-r; F^{\infty}/F, \sum_{(U,V)\in \ideal{F}} \sum_i a_{U,V}^{(i)}\ind{G}{U}{\chi_{U,V}^{(i)}} \right) &&  \\
&= L_{\Sigma} (1-r; F^{\infty}/F,\rho), &&
\end{align*}
here the first equality follows from the definition of the evaluation
 maps, and the fifth equality follows from the compatibility of Artin
 $L$-functions with direct sum of representations.

The second equality follows from the next lemma (Lemma \ref{lem:akind}).

Note that if $\theta$ is injective, the element $u$ above is determined
 uniquely, so is $\xi_{F^{\infty}/F}$.
\end{proof}

\begin{lem} \label{lem:akind}
Let $(U,V)$ is an element of $\ideal{F}$. 
Then for an arbitrary character $\chi$ of $U/V$, 
the following diagram commutes.
\begin{equation*}
\xymatrix{
K_1(\iw{G}_S) \ar[d]_{\ev_{\ind{G}{U}{\chi}}} \ar[r]^{\hspace{-1cm}\theta_{S,U,V}} & K_1(\iw{U/V}_S)=\iw{U/V}_S^{\times} \ar[dl]^{\ev_{\chi}} \\
\line{\Q}_p\cup \{ \infty \} & 
}
\end{equation*}
\end{lem}

\begin{proof}
We will show the integral version of this lemma. In other words, we will
 prove the commutativity of the following diagram:
 \begin{equation*}
\xymatrix{
K_1(\iw{G}) \ar[d]_{\ev_{\ind{G}{U}{\chi}}} \ar[r]^{\hspace{-1cm}\theta_{U,V}} & K_1(\iw{U/V})=\iw{U/V}^{\times} \ar[dl]^{\ev_{\chi}} \\
\line{\Q}_p^{\times} & 
}
\end{equation*}

Let $W$ be the representation space of $\ind{G}{U}{\chi}$ over $\line{\Q}$ and let $W'$ be that of $\chi$. 
Suppose that $\End_{\line{\Q}} (W)$ (resp.\ $\End_{\line{\Q}} (W')$) acts on $W$ (resp.\ $W'$) from the right. 

Then by the definition of evaluation map, we obtain 
\begin{equation*}
\ev_{\ind{G}{U}{\chi}} = [W \otimes_{\iw{G}} -].
\end{equation*}

On the other hand, we have
\begin{equation*}
W\cong W' \otimes_{\iw{U/V}} \iw{G}
\end{equation*}
by the definition of induced representations, therefore we have 
\begin{equation*}
\ev_{\ind{G}{U}{\chi}} = \left[ W' \otimes_{\iw{U/V}} \left(
						       {}_{\iw{U/V}} \iw{G}_{\iw{G}} \right) \otimes_{\iw{G}} - \right],
\end{equation*}
and the right hand side is nothing but the definition of $\ev_{\chi}
 \circ \theta_{U,V}$.

It is not difficult to generalize this result to the localized version.
\end{proof}

Kazuya Kato has constructed the theta maps for $p$-adic Lie groups of
Heisenberg type (\cite{Kato1}, see also \S8.1) and for certain open subgroups
of $\Zp^{\times} \ltimes \Zp$ (\cite{Kato2}). Kakde also constructed the theta map for the $p$-adic Lie group $H\rtimes \Gamma$ of ``special type'' where $H$ is a compact pro-$p$ abelian $p$-adic Lie group. He used the method of Kato in \cite{Kato1} (See \cite{Kakde}).

In our case $G=G^f \times \Gamma$, it is difficult to show that
$(\xi_{U,V})_{(U,V)\in \ideal{F}} \in \Psi_S$ (See \S 7 and \S 8), so we
should modify the proof of this proposition to show our main theorem (Theorem
\ref{thm:maintheorem}). But this technique gives us the ideas
how to reduce the difficulties come from non-commutativity to the conditions of commutative cases.

%
%
\section{The additive theta map}
%
%

In the following three sections, we construct the theta map (and its
localized version) for the group $G=G^f \times \Gamma$. First, we construct a family $\ideal{F}=\{(U_i, V_i)\}$ which satisfies the hypothesis $(\flat)$ of \S 3.2. Then we define a $\Zp$-module homomorphism
\begin{equation*}
\theta^{+} \colon \Zp [[\conj{G}]] \longrightarrow \prod_i \Zp [[U_i/V_i]]
\end{equation*}
and characterize its image $\Omega$. We show that $\theta^{+}$ 
induces an isomorphism from $\Zp[[\conj{G}]]$ to $\Omega$, which we call {\it the additive
theta map for the group $G$}.

%
\subsection{Construction of the family $\ideal{F}$}
%

First, we define subgroups $H$ and $N$ of $G^f$ as follows:
\begin{align*}
H&= \begin{pmatrix} 1 & \Fp & \Fp & 0 \\ 0 & 1 & \Fp & 0 \\
0 & 0 & 1 & 0 \\ 0 & 0 & 0 &  1
\end{pmatrix}, &
N &= \begin{pmatrix} 1 & 0 & 0 & \Fp \\ 0 & 1 & 0 & \Fp \\
0 & 0 & 1 & \Fp \\ 0 & 0 & 0 & 1
\end{pmatrix}.
\end{align*}
Note that $N$ is abelian and normal in $G$. There are canonical isomorphisms
\begin{align*}
H\xrightarrow{\simeq} \begin{pmatrix} 1 &\Fp & \Fp \\
0 & 1 & \Fp \\ 0 & 0 & 1 \end{pmatrix} & ; 
\begin{pmatrix} 1 & a & d & 0 \\
0 & 1 & b & 0 \\ 0 & 0 & 1 & 0 \\ 0 & 0 & 0 & 1\end{pmatrix} \mapsto
\begin{pmatrix} 1 & a & d \\
0 & 1 & b \\ 0 & 0 & 1 \end{pmatrix}, \\
N\xrightarrow{\simeq} \quad \begin{pmatrix} \Fp \\ \Fp \\ \Fp
			    \end{pmatrix} \quad & ; 
\begin{pmatrix} 1 & 0 & 0 & f \\
0 & 1 & 0 & e \\ 0 & 0 & 1 & c \\ 0 & 0 & 0 & 1\end{pmatrix} \mapsto
\begin{pmatrix} f \\ e  \\ c \end{pmatrix}.
\end{align*}

By the isomorphisms above, we identify $H$ with $\begin{pmatrix} 1 &\Fp & \Fp \\
0 & 1 & \Fp \\ 0 & 0 & 1 \end{pmatrix}$ and $N$ with $\begin{pmatrix} \Fp \\ \Fp \\ \Fp \end{pmatrix}$.

Then $G^f$ is represented as a semi-direct product 
\begin{equation*}
G^f \cong H\ltimes N=\begin{pmatrix} 1 &\Fp & \Fp \\
0 & 1 & \Fp \\ 0 & 0 & 1 \end{pmatrix} \ltimes \begin{pmatrix} \Fp \\ \Fp \\ \Fp    \end{pmatrix},
\end{equation*}
where $H$ acts on $N$ from the left as ordinary product of matrices. 

From representation theory of semi-direct products of finite groups
(See \cite{Serre1}, Chapitre 8.2), all irreducible representations of $G^f$ are obtained from representations of $H$ and those of $N$. Let us review this construction.

First, let $\x{N}$ be the character group of the abelian group $N$:
\begin{equation*}
\x{N} =  \biggl\{   \chi_{ijk} \biggl|  0\leq i,j,k \leq p-1
 \biggr. \biggr \}
\end{equation*}
where $\chi_{ijk}\left( \! \! \begin{pmatrix} f \\ e \\ c \end{pmatrix} 
\! \! \right) = \exp \left( \frac{2\pi \sqrt{-1}}{p}(fi+ej+ck) \right) \quad (c,e,f\in \Fp )$.

The group $H$ acts on $\x{N}$ from the right by 
\begin{equation*}
(\chi*h)(n)=\chi(hn)\qquad \text{for all $n\in N$}
\end{equation*}
where $h\in H$ and $\chi \in \x{N}$. Specifically,
\begin{equation*}
\chi_{ijk} * \begin{pmatrix} 1 & a & d \\ 0 & 1 & b \\ 0 & 0 & 1 \end{pmatrix}
=\chi_{i,ai+j,di+bj+k} \quad \text{for every} \quad \begin{pmatrix} 1 & a & d \\ 0 & 1 & b \\ 0 & 0 & 1 \end{pmatrix} \in H
\end{equation*}
where $\chi_{ijk}\in \x{N}$. Then we have 
\begin{align*}
\chi_{00k} & \quad (0\leq k\leq p-1), & \chi_{0j0} & \quad(1\leq j\leq p-1), & 
\chi_{i00} & \quad (1\leq i\leq p-1)
\end{align*}
as a system of representatives for the orbital decomposition $\x{N}/H$.

Next, let $H_{ijk}$ be the isotropic subgroup of $H$ at each $\chi_{ijk}$. Then for each representative above, we have
\begin{align*}
H_{00k} & =H_0= H  \quad (0\leq k\leq p-1), \\
H_{0j0} & = H_1=\begin{pmatrix} 1 & \Fp & \Fp \\ 0 & 1 & 0 \\ 0 & 0 & 1 
\end{pmatrix} \quad (1\leq j\leq p-1), \\
H_{i00} & =H_2= \begin{pmatrix} 1 & 0 & 0 \\ 0 & 1 & \Fp \\ 0 & 0 & 1 
\end{pmatrix} \quad (1\leq i\leq p-1).
\end{align*}

Let $G^f_{\ell}= H_{\ell}\ltimes N$ for $\ell=0,1,2$. Then we may extend
the domain of the character $\chi_{00k}$ (resp.\ $\chi_{0j0}$, $\chi_{i00}$) to $G^f_0$ (resp.\ $G^f_1$, $G^f_2$) 
by setting
\begin{equation*}
\chi_{00k}(h_0n)=\chi_{00k}(n), \quad \chi_{0j0}(h_1n)=\chi_{0j0}(n), 
\quad \chi_{i00}(h_3n)= \chi_{i00}(n)
\end{equation*}
where $h_{\ell}\in H_{\ell} (\ell=0,1,2)$ and $n\in N$. Note that each $\chi_{00k}$
(resp.\ $\chi_{0j0}$, $\chi_{i00}$) is a character of degree 1 of the
group $G^f_0$ (resp.\ $G^f_1$, $G^f_2$).

Now let $\rho_{\ell}\, (\ell=0,1,2)$ be an arbitrary irreducible representation of
$H_{\ell}$. Then we obtain an irreducible representation of $G^f_{\ell}$ by
composing $\rho_{\ell}$ with the canonical projection $G^f_{\ell} \rightarrow
H_{\ell}$, which we also denote by $\rho_{\ell}$. Consider the tensor
products of representations $\chi_{00k} \otimes \rho_0$ (resp.\
$\chi_{0j0} \otimes \rho_1$, $\chi_{i00} \otimes \rho_2$), and  let
\begin{align*}
\theta_{k,\rho_0} &= \ind{G^f}{G^f_0}{\chi_{00k} \otimes \rho_0} 
= \chi_{00k} \otimes \rho_0, \\
\theta_{j,\rho_1} &= \ind{G^f}{G^f_1}{\chi_{0j0} \otimes \rho_1}, \\
\theta_{i,\rho_2} &= \ind{G^f}{G^f_2}{\chi_{i00} \otimes \rho_2}.
\end{align*}

\begin{prop} \label{prop:irr}
Suppose $0\leq k\leq p-1$, $1\leq j\leq p-1$ and $1\leq i\leq p-1$. Let
 $\rho_{\ell}$ be an irreducible representation of $H_{\ell}$. Then each 
$\theta_{k,\rho_0},\, \theta_{j,\rho_1},\, \theta_{i,\rho_2}$ is an 
irreducible representation of $G^f$. Moreover, an arbitrary irreducible representation of 
$G^f$ is isomorphic to one of the $\theta_{k,\rho_0},\, \theta_{j,\rho_1},\, \theta_{i,\rho_2}$.
\end{prop}

\begin{proof}
See \cite{Serre1}, Chapitre 8.2 (for irreducibility of each $\theta$, we use Mackey's irreducibility criterion).
\end{proof}

Note that $H_1$ and $H_2$ are abelian groups, so if we set 
\begin{align*}
U_1^f &= H_1\ltimes N &
V_1' &= \{ I_3 \} \ltimes \begin{pmatrix} \Fp \\ 0 \\ \Fp \end{pmatrix} \\
& =\begin{pmatrix} 1 & \Fp & \Fp & \Fp \\ 0 & 1 & 0 & \Fp \\
0 & 0 & 1 & \Fp \\ 0 & 0 & 0 & 1 \end{pmatrix}, &
& = 
\begin{pmatrix} 1 & 0 & 0 & \Fp \\ 0 & 1 & 0 & 0 \\ 0 & 0 & 1 & \Fp \\
0 & 0 & 0 & 1 \end{pmatrix} ,\\
U_2^f &= H_2\ltimes N &
V_2' &= \{ I_3 \} \ltimes \begin{pmatrix} 0 \\ \Fp \\ \Fp \end{pmatrix}  \\
& =\begin{pmatrix} 1 & 0 & 0 & \Fp \\ 0 & 1 & \Fp & \Fp \\
0 & 0 & 1 & \Fp \\ 0 & 0 & 0 & 1 \end{pmatrix}, &
& = 
\begin{pmatrix} 1 & 0 & 0 & 0 \\ 0 & 1 & 0 & \Fp \\ 0 & 0 & 1 & \Fp \\
0 & 0 & 0 & 1 \end{pmatrix},
\end{align*}
each irreducible representation $\theta_{j,\rho_1}$
(resp.\ $\theta_{i,\rho_2}$) is obtained as the induced representation
$\ind{G^f}{U_1^f}{\chi_1'}$ (resp.\ $\ind{G^f}{U_2^f}{\chi_2'}$) where
$\chi_1'$ (resp.\ $\chi_2'$) is a certain character of 
$U_1^f/V_1'$ (resp.\ $U_2^f/V_2'$) of finite order.

On the other hand, we may regard $H_0(=H)$ as a semi-direct product
\begin{align*}
H\cong \begin{pmatrix} 1 & \Fp \\ 0 & 1 \end{pmatrix} \ltimes
\begin{pmatrix} \Fp & \Fp \end{pmatrix} &; 
\begin{pmatrix} 1 & a & d \\ 0 & 1 & b \\ 0 & 0 & 1
\end{pmatrix}
\mapsto
\begin{pmatrix} 1 & b \\ 0 & 1 \end{pmatrix}
\ltimes \begin{pmatrix} a & d \end{pmatrix}
\end{align*}
where $\begin{pmatrix} 1 & b \\ 0 & 1 \end{pmatrix} \in 
\begin{pmatrix} 1 &\Fp \\ 0 & 1 \end{pmatrix}$ acts on 
$\begin{pmatrix} a & d \end{pmatrix} \in \begin{pmatrix} \Fp & \Fp \end{pmatrix}$ from the left by
\begin{equation*}
\begin{pmatrix} 1 & b \\ 0 & 1 \end{pmatrix} 
* \begin{pmatrix} a & d \end{pmatrix} = \begin{pmatrix} a & d \end{pmatrix} 
\begin{pmatrix} 1 & b \\ 0 & 1 \end{pmatrix}^{-1} \qquad 
(\mbox{ordinary product of matrices}).
\end{equation*}

Let
\begin{align*}
U_0^H &=H & V_0^H &= \{ I_2 \} \ltimes \begin{pmatrix} 0 & \Fp \end{pmatrix} 
\\
& =\begin{pmatrix} 1 & \Fp & \Fp \\ 0 & 1 & \Fp \\ 0 & 0 & 1
   \end{pmatrix}, &
& = \begin{pmatrix} 1 & 0 & \Fp \\ 0 & 1& 0 \\ 0 & 0 & 1\end{pmatrix}, \\
U_1^H &=\{ I_2 \} \ltimes \begin{pmatrix} \Fp & \Fp \end{pmatrix} & V_1^H &= \{ I_2 \} 
\ltimes \begin{pmatrix} 0 & 0 \end{pmatrix}  \\
& =\begin{pmatrix} 1 & \Fp & \Fp \\ 0 & 1 & 0 \\ 0 & 0 & 1
   \end{pmatrix}, &
& = \{ I_3 \},
\end{align*}
then by an argument similar to above, every irreducible representation of $H$ is obtained as the induced 
representation $\ind{H}{U^H_{\ell}}{\chi_{\ell}}(\ell=0 \, \mbox{or} \, 1)$ where $\chi_{\ell}$ 
is a certain character of $U^H_{\ell}/V^H_{\ell}$. Therefore if we set
\begin{align*}
U_0^f & = U^H_0 \ltimes N  & V_0^f &= V^H_0 \ltimes 
\begin{pmatrix} \Fp \\ \Fp \\ 0 \end{pmatrix} \\
 &=G, & &= 
\begin{pmatrix} 1 & 0 & \Fp & \Fp \\ 0 & 1 & 0 & \Fp \\ 0 & 0 & 1 & 0 \\
0 & 0 & 0 & 1 \end{pmatrix}, \\
U_1^f & = U^H_1 \ltimes N 
   & V_1'' &= V^H_1 \ltimes \begin{pmatrix} \Fp \\ \Fp \\ 0 \end{pmatrix} \\
& = \begin{pmatrix} 1 & \Fp & \Fp & \Fp \\
0 & 1 & 0 & \Fp \\ 0 & 0 & 1 & \Fp \\ 0 & 0 & 0 & 1 \end{pmatrix}, 
& & = 
\begin{pmatrix} 1 & 0 & 0 & \Fp \\ 0 & 1 & 0 & \Fp \\ 0 & 0 & 1 & 0 \\
0 & 0 & 0 & 1 \end{pmatrix}, 
\end{align*}
each irreducible representation $\theta_{k,\rho_0}$ is obtained as the
induced representation $\ind{G^f}{U_0^f}{\chi_0}$
(resp.\ $\ind{G^f}{U_1^f}{\chi_1''}$) where $\chi_0$ (resp.\ $\chi_1''$) is a certain character of 
$U_0^f/V_0^f$ (resp.\ $U_1^f/V_1''$) of finite order.

Let $V_1^f=V_1'\cap V_1''=\langle \zeta\rangle$. Then by the argument above, every irreducible representation of $G^f$ is obtained as the induced 
representation of a certain character $\chi_i \in \x{U_i^f/V_i^f}$ where $\x{U_i^f/V_i^f}$ is the character group of the abelian group $U_i^f/V_i^f$. Hence 
$\ideal{F}^f=\{ (U_0^f, V_0^f), (U_1^f, V_1^f), (U_2^f, V_2')\}$
satisfies the hypothesis $(\flat)$ for the group $G^f$.

For certain technical reasons, we replace $V_2'$ by 
\begin{equation*}
V_2^f=\{ I_3\} \ltimes \begin{pmatrix} 0 \\ \Fp \\ \mathbf{0} \end{pmatrix}
=\begin{pmatrix} 1 & 0 & 0 & 0 \\ 0 & 1 & 0 & \Fp \\ 0 & 0 & 1 & \mathbf{0} \\
0 & 0 & 0 & 1 \end{pmatrix}
\end{equation*}
and add following subgroups to our family $\ideal{F}^f$ 
\begin{align*}
\tilder{U_2}^f &= \begin{pmatrix} 1 & 0 & \Fp \\ 0 & 1 & \Fp \\ 0 & 0 & 1 \end{pmatrix} \ltimes N  &
\tilder{V_2}^f &= \{I_3 \} \ltimes \begin{pmatrix} \Fp \\ \Fp \\ 0 
\end{pmatrix} \\
& = \begin{pmatrix} 1 & 0 & \Fp & \Fp \\
0 & 1 & \Fp & \Fp \\ 0 & 0 & 1 & \Fp \\ 0 & 0 & 0 & 1 \end{pmatrix},&
& = \begin{pmatrix} 1 & 0 & 0 & \Fp \\ 
0 & 1 & 0 & \Fp \\ 0 & 0 & 1 & 0 \\ 0 & 0 & 0 & 1\end{pmatrix}, \\
U_3^f &= \{ I_3\} \ltimes N &
V_3^f &= \{I_3 \} \ltimes \{ 0 \}  \\
& = \begin{pmatrix} 1 & 0 & 0 & \Fp \\
0 & 1 & 0 & \Fp \\ 0 & 0 & 1 & \Fp \\ 0 & 0 & 0 & 1 \end{pmatrix}, &
& = \{ I_4 \}.
\end{align*}

Note that $V_2^f$ (resp.\ $\tilder{V_2}^f$, $V_3^f$) is normal in
$U_2^f$ (resp.\ $\tilder{U_2}^f$, $U_3^f$) and $U_2^f/V_2^f$ (resp.\
$\tilder{U_2}^f/\tilder{V_2}^f$, $U_3^f/V_3^f$) is abelian.

For each $i$ (including $\tilder{U_2}$ and $\tilder{V_2}$), let
\begin{equation*}
U_i = U_i^f\times \Gamma,\qquad V_i =V_i^f \times \{ 1\}
\end{equation*}
and let $\ideal{F}=\{ (U_i, V_i)\}_i$. It is clear that $\ideal{F}$
satisfies the hypothesis $(\flat)$ for the group $G$.

%
\subsection{Construction of the isomorphism $\theta^+$}
%

In the following, we denote
\begin{equation*}
\begin{split}
\prod_i \Zp[[U_i/V_i]]=\Zp[[U_0/V_0]] &\times \Zp[[U_1/V_1]]  \\
&\times \Zp[[\tilder{U_2}/\tilder{V_2}]] \times \Zp[[U_2/V_2]] \times
\Zp[[U_3/V_3]]
\end{split}
\end{equation*}
and we use the notation $U_i$ (resp.\ $V_i$) for one of the subgroups $U_0, U_1, \tilder{U_2}, U_2$ and $U_3$ (resp.\ $V_0,V_1,\tilder{V_2},V_2$ and $V_3$). 

For an arbitrary finite group $F$, we denote by $\Zp[\conj{F}]$ the free
$\Zp$-module of finite rank with free bases $\conj{F}$. For an arbitrary pro-finite group $P$, let $\Zp [[\conj{P}]]$ be the projective limit of the free $\Zp$-modules
$\Zp [\conj{P_{\lambda}}]$ over finite quotient groups $P_{\lambda}$ of $P$.

Let $\{ u_1,u_2,\dotsc , u_{r_i} \} \subseteq G$ be
a system of representatives of the coset decomposition $G/U_i$ for each $i$. Then for an arbitrary conjugacy class
$[g]\in \conj{G}$, set
\begin{equation*}
\Tr_i([g]) = \sum_{j=1}^{r_i} \tau_j([g])
\end{equation*}
where $\tau_j([g])=[u_j^{-1} g u_j]$ if $u_j^{-1} g u_j$ is contained
in $U_i$, and $\tau_j([g])=0$ otherwise. It is easy to see that $\Tr_i([g])$ is independent of the choice of the representatives $\{ u_j \}_{j=1}^{r_i}$,
therefore $\Tr_i$
induces a
well-defined $\Zp$-module homomorphism
\begin{equation*}
\Tr_i \colon \Zp [[\conj{G}]] \longrightarrow \Zp [[\conj{U_i}]], 
\end{equation*}   
which we call {\it the trace homomorphism from
$\Zp[[\conj{G}]]$ to $\Zp [[\conj{U_i}]]$}.

We define the homomorphism $\theta^{+}_i$ as the composition of
the trace map $\Tr_i$ and the natural surjection 
\begin{equation*}
\Zp [[\conj{U_i}]] \longrightarrow \Zp [[U_i/V_i]] .
\end{equation*}

\medskip
Now let us calculate $\theta^+_i([g_0])$ for $[g_0]\in
\conj{G^f}$. Here we regard $[g_0]$ as an element $([g_0],0)$ contained
in $\conj{G}=\conj{G^f} \times \conj{\Gamma}$. 

Here we only compute $\theta^+_2([g_0])$. We may take $\{ \alpha^j\delta^k \mid 0\leq j,k \leq p-1 \}$ as a system of 
representatives of $G/U_2$. If $g_0=\begin{pmatrix} 1 & a & d & f \\
0 & 1 & b & e \\ 0 & 0 & 1 & c \\ 0 & 0 & 0 & 1 \end{pmatrix}$, we have 
\begin{equation*}
(\alpha^j \delta^k)^{-1} g_0(\alpha^j \delta^k)=
\begin{pmatrix}  1 & a & d-bj & f-ej-ck \\ 0 & 1 & b & e \\
0 & 0 & 1 & c \\ 0 & 0 & 0 & 1 \end{pmatrix}.
\end{equation*}

This implies that $(\alpha^j \delta^k)^{-1}g_0(\alpha^j\delta^k)$ is
 contained in $U_2$ if and only if $a=0$ and $d-bj=0$.

First, suppose that $a=0$ and $b\neq 0$. Then there exists a unique $j$ satisfying $d-bj=0$. We denote this $j$ by $\dfrac{d}{b}$. Then we have
\begin{eqnarray*}
\theta^+_2([g_0]) &=& \sum_{k=0}^{p-1} 
\begin{pmatrix} 1 & 0 & 0 & f-e\cdot \frac{d}{b}-ck \\
0 & 1 & b & e \\ 0 & 0 & 1 & c \\ 0 & 0 & 0 & 1 \end{pmatrix} \qquad 
\mod{\langle \varepsilon \rangle} \\
 &=& 
\begin{cases} 
\beta^b \gamma^c (1+\zeta+\cdots +\zeta^{p-1}) & \text{if $c\neq 0$,} \\
p\beta^b \zeta^{f-e\cdot \frac{d}{b}} & \text{if $c=0$.}
\end{cases}
\end{eqnarray*}

If $a=b=0$ and $d\neq 0$, $(\alpha^j \delta^k)^{-1} g_0 (\alpha^j
\delta^k)$ is not contained in $U_2$ for each $0\leq j,k\leq p-1$, hence $\theta^+_2([g_0])=0$. 

If $a=b=d=0$, we have
\begin{eqnarray*}
\theta^+_2([g_0]) &=& \sum_{j,k=0}^{p-1} 
\begin{pmatrix} 1 & 0 & 0 & f-ej-ck \\
0 & 1 & 0 & e \\ 0 & 0 & 1 & c \\ 0 & 0 & 0 & 1 \end{pmatrix} \qquad 
\mod{\langle \varepsilon \rangle} \\
 &=& 
\begin{cases} 
p^2 \zeta^f & \text{if $c=e=0$,} \\
p\gamma^c(1+\zeta+\dotsc +\zeta^{p-1}) & \text{if $c\neq 0$,} \\
p(1+\zeta+\dotsc +\zeta^{p-1}) & \text{if $c=0, e\neq 0$.}
\end{cases}
\end{eqnarray*}

In this way, we have
\begin{equation*}
\theta^+_2([g_0])=
\begin{cases}
0 & \text{if $a\neq 0$ or $a=b=0,d\neq 0$,} \\
\beta^b \gamma^c (1+\zeta+\dotsc +\zeta^{p-1}) & \text{if $a=0,b\neq 0,c\neq 0$,} \\
p\beta^b \zeta^{f-e\cdot \frac{d}{b}} & \text{if $a=c=0,b\neq 0$,} \\
p\gamma^c(1+\zeta+\dotsc +\zeta^{p-1}) & \text{if $a=b=d=0, c\neq 0$,} \\
p(1+\zeta+\dotsc +\zeta^{p-1}) & \text{if $a=b=c=d=0, e\neq 0$,} \\
p^2\zeta^f & \text{if $a=b=c=d=e=0$.}
\end{cases}
\end{equation*}

Similarly, we may take $\{ \beta^j \mid 0\leq j\leq p-1 \}$ for a system of representatives of $G/U_1$, $\{ \alpha^j \mid 0\leq j\leq p-1 \}$ for that of $G/\tilder{U_2}$, and $\{
\alpha^j \beta^k \delta^{\ell}  \mid 0\leq j, k ,\ell \leq p-1 \}$ for that of
$G/U_3$. Using these representatives, we may calculate $\theta^+([g_0])$ as follows:

\begin{align*}
 \theta^+ \colon  \Zp [[\conj{G}]] & \longrightarrow  \prod_i \Zp [[ U_i/V_i ]]
\end{align*}

\begin{align*}
\ideal{c}_{\mathrm{I}}(a,b,c) &=  \{  \alpha^a \beta^b \gamma^c \delta^d \varepsilon^e \zeta^f \mid 0\leq d,e,f\leq p-1 \}  && (a\neq 0,b\neq 0) \\
  & \mapsto  (\alpha^a \beta^b \gamma^c , 0, 0 , 0 ,0), &\\ 
\ideal{c}_{\mathrm{II}}(a,c,d,e)  &= \{  \alpha^a  \gamma^c \delta^d \varepsilon^e \zeta^f  \mid 0\leq f\leq p-1 \}  && (a\neq 0, c\neq 0) \\
  & \mapsto  (\alpha^a \gamma^c,\, \alpha^a \gamma^c \delta^d \varepsilon^e h_{\varepsilon^c \delta^{-a}}, 0, 0 ,0), &\\
\ideal{c}_{\mathrm{III}}(a,e)  &= \{  \alpha^a \varepsilon^e \delta^d \zeta^f \mid 0\leq d,f\leq p-1 \}  && (a\neq 0) \\
  & \mapsto  (\alpha^a , \, \alpha^a \varepsilon^e h_{\delta}, 0, 0 ,0), \\
\ideal{c}_{\mathrm{IV}}(b,c) &= \{ \beta^b  \gamma^c \delta^d \varepsilon^e \zeta^f  \mid 0\leq d,e,f\leq p-1 \}  && (b\neq 0, c\neq 0) \\
 & \mapsto  (\beta^b \gamma^c ,\, 0,\, \beta^b \gamma^c h_{\delta},\, \beta^b \gamma^c h_{\zeta} ,0), &\\
\ideal{c}_{\mathrm{V}}(b,f) &= \{ \beta^b \delta^d \varepsilon^e \zeta^{f'} \mid f'-\frac{d}{b}e \equiv f \quad \mod{p} \}  && (b\neq 0) \\
 & \mapsto  ( \beta^b,\, 0,\, \beta^b h_{\delta},\, p\beta^b \zeta^f ,0), &\\
\ideal{c}_{\mathrm{VI}}(c,d) &= \{ \gamma^c  \delta^d \varepsilon^e \zeta^f \mid 0\leq e,f\leq p-1 \}  && (c\neq 0 , d \neq 0) \\
 & \mapsto  (\gamma^c ,\, \gamma^c \delta^d h_{\varepsilon} ,\, p\gamma^c \delta^d , 0 ,0), &\\
\ideal{c}_{\mathrm{VII}}(c) &= \{ \gamma^c  \varepsilon^e \zeta^f \mid 0\leq e,f\leq p-1 \}  && (c\neq 0) \\
 & \mapsto  (\gamma^c ,\, \gamma^c h_{\varepsilon} ,\, p\gamma^c ,\, p\gamma^c h_{\zeta} ,p\gamma^ch_{\varepsilon} h_{\zeta} ), & \\
\ideal{c}_{\mathrm{VIII}}(d,e) &= \{ \delta^d  \varepsilon^e \zeta^f \mid 0\leq f\leq p-1 \} && (d\neq 0) \\
 & \mapsto  (1,\, p\delta^d \varepsilon^e ,\, p\delta^d , 0 ,0) &\\
\ideal{c}_{\mathrm{IX}}(e) &= \{ \varepsilon^e  \zeta^f \mid 0\leq f\leq p-1 \}  && (e\neq 0)  \\
 & \mapsto  (1,\, p\varepsilon^e ,\, p ,\, ph_{\zeta} , p^2 \varepsilon^e h_{\zeta} ),  &\\ 
\ideal{c}_{\mathrm{X}}(f) &=\{ \zeta^f \} && \\ 
& \mapsto (1, p, p, p^2 \zeta^f ,p^3\zeta^f), & 
\end{align*}
where for an element $g_0\in G^f$, $h_{g_0}$ denotes an element $1+g_0+g_0^2+ \dotsc
+g_0^{p-1}$ of $\Zp [[\conj{G}]]$. 

\medskip
Now let us consider the finite quotients $G^{(n)}= G^f \times
\Gamma/\Gamma^{p^n}$ and $U^{(n)}_i = U_n^f \times
\Gamma/\Gamma^{p^n}$. Then we have
\begin{align*}
\Zp [ G^{(n)} ] &\cong \Zp[G^f] \otimes_{\Zp} \Zp[\Gamma/\Gamma^{p^n}], \\
\Zp [ U^{(n)}_i/V_i ] &\cong \Zp[U_i^f/V_i^f] \otimes_{\Zp}
 \Zp[\Gamma/\Gamma^{p^n}].
\end{align*}

$\theta^+_i$ induces a homomorphism 
\begin{equation*}
\theta^{+,(n)}_i \colon \Zp[G^{(n)}] \rightarrow \Zp[U_i^{(n)}/V_i]
\end{equation*}
which sends $[g]=([g_0], t^z)\in
\conj{G^{(n)}}=\conj{G^f}\times \Gamma/\Gamma^{p^n}$ to
$(\theta^+_i([g_0]), t^z)$. For example, we have  
\begin{equation*}
\theta^{+,(n)}((\ideal{c}_I(a,b,c),t^z))=((\alpha^a \beta^b \gamma^c, t^z),(0,t^z),(0,t^z),(0,t^z),(0,t^z)),
\end{equation*}
and so on.

\medskip
We would like to characterize the image of $\theta^{+,(n)}$. 

Let $I_i^f$ be the submodule of $\Zp [U^f_i/V_i^f]$
defined as follows:
\begin{eqnarray*}
I_1^f &=& [p\delta^d \varepsilon^e, \alpha^a \varepsilon^e h_{\delta} \, (a\neq 0), \gamma^c \delta^d h_{\varepsilon} \, (c\neq 0), \alpha^a \gamma^c \delta^d \varepsilon^e h_{\varepsilon^c \delta^{-a}} \, (a\neq 0, c\neq 0) ]_{\Zp}, \\ 
\tilder{I_2}^f &=& [\beta^b \gamma^c h_{\delta} \, (b\neq 0) , p\gamma^c \delta^d ]_{\Zp}, \\
I_2^f &=& [p^2\zeta^f, p\gamma^c h_{\zeta} , \beta^b \gamma^c h_{\zeta} \, (b\neq 0, c\neq 0) ,\, p\beta^b \zeta^f (b\neq 0)]_{\Zp}, \\
I_3^f &=& [p^3\zeta^f , p^2\varepsilon^e h_{\zeta} \, (e\neq 0), p\gamma^c h_{\varepsilon}h_{\zeta} (c\neq 0)]_{\Zp}.
\end{eqnarray*}

Here we denote by $[S]_R$ the $R$-submodule of $M$ generated by $S$ over
$R$ where $R$ is a commutative ring, $M$ is an $R$-module and $S$ is a
finite subset of $M$. Each $I_i^f$ is a $\Zp$-submodule of $\Zp[U^f_i/V_i^f]$ generated by
$\{ \theta^+_i([g_0]) \mid g_0 \in G^f \}$.

Let 
\begin{equation*}
I_i^{(n)}= I_i^f \otimes_{\Zp} \Zp [\Gamma/\Gamma^{p^n} ] \quad \left(
								\subseteq
								\Zp
								[
								 U_i^{(n)}/V_i ] \right)
\end{equation*}
and
\begin{equation*}
I_i = \varprojlim_n I^{(n)}_i \quad \left( \subseteq \Zp[[U_i/V_i]]\right).
\end{equation*}

For each $i$, $I_i$ is the image of the homomorphism $\theta^{+}_i$.

Note that generators of $I^{(n)}_1$ (resp.\ $\tilder{I_2}^{(n)}$,
$I^{(n)}_3$) above are $\Zp[\Gamma/\Gamma^{p^n}]$-linearly independent, but those of $I^{(n)}_2$ are not $\Zp[\Gamma/\Gamma^{p^n}]$-linearly
independent, which have a relation 
\begin{equation}
\sum_{f=0}^{p-1} p^2\zeta^f = p\cdot ph_{\zeta}. \label{eq:dep}
\end{equation}

\begin{defn} \label{def:omega}
 We define $\Omega$ to be the $\Zp$-submodule of $\displaystyle \prod_i \Zp [[U_i/V_i]]$ consisting of all elements $(y_0, y_1, \tilder{y_2}, y_2, y_3 )$ satisfying the following two conditions:

\begin{enumerate}
  \item (trace relations)
\begin{enumerate}[(rel-1)]
  \item $\Tr_{\Zp[[U_0/V_0]]/\Zp[[U_1/V_0]]}y_0 \equiv y_1$,
  \item $\Tr_{\Zp[[U_0/V_0]]/\Zp[[\tilder{U_2}/V_0]]}y_0 \equiv \tilder{y_2}$,
  \item $\Tr_{\Zp[[\tilder{U_2}/\tilder{V_2}]]/\Zp[[U_2/\tilder{V_2}]]}\tilder{y_2} \equiv y_2$,

  \item $\Tr_{\Zp[[U_1/\tilder{V_2}]]/\Zp[[U_1\cap \tilder{U_2}/\tilder{V_2}]]}y_1 \equiv \Tr_{\Zp[[\tilder{U_2}/\tilder{V_2}]]/\Zp[[U_1\cap \tilder{U_2}/\tilder{V_2}]]}\tilder{y_2}$,
  \item $\Tr_{\Zp[[U_1/V_1]]/\Zp[[U_3/V_1]]}y_1 \equiv y_3$,
  \item $\Tr_{\Zp[[U_2/V_2]]/\Zp[[U_3/V_2]]}y_2 \equiv y_3$.
\end{enumerate}
   (These trace relations are described in Figure \ref{fg:relations}.)
 \item $y_i \in I_i$ for each $i$.
\end{enumerate}
\end{defn}

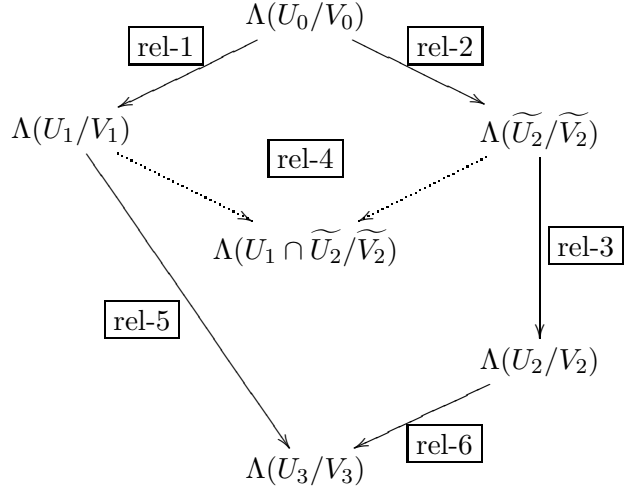
\begin{figure*}
\begin{equation*}
\xymatrix{
   & \iw{U_0/V_0} \ar[dl]_{\fbox{rel-1}} \ar[dr]^{\fbox{rel-2}} &  \\
\iw{U_1/V_1} \ar[dddr]_{\fbox{rel-5}} \ar@{.>}[dr]
 \ar@{}[rr]_{\fbox{rel-4}} &  &
 \iw{\tilder{U_2}/\tilder{V_2}} \ar[dd]^{\fbox{rel-3}} \ar@{.>}[dl]
 \\
 & \iw{U_1\cap \tilder{U_2}/\tilder{V_2}} & \\
& & \iw{U_2/V_2} \ar[dl]^{\fbox{rel-6}} \\
& \iw{U_3/V_3} &}
\end{equation*}
\caption{Trace and norm relations.}
\label{fg:relations}
\end{figure*}

\begin{propdef} \label{pd:adak}
 The homomorphism $\theta^+=(\theta^+_i)$ induces an isomorphism
\begin{equation*}
 \theta^+\colon \Zp[[\conj{G}]] \xrightarrow{\simeq} \Omega.
\end{equation*}

{\upshape We call this induced isomorphism $\theta^+$} the additive theta map for $G$.
\end{propdef}

\begin{proof}
It is clear from the calculation above that $\Omega$ contains the image of $\theta^+$. Hence it is sufficient to prove its injectivity and surjectivity. 

We now show that $\theta^+$ induces an isomorphism on the finite quotients
\begin{equation*}
\theta^{+,(n)} \colon \Zp[\conj{G^{(n)}}] \xrightarrow{\simeq} \Omega^{(n)},
\end{equation*}
for every $n\geq 1$, where $\Omega^{(n)}$ is defined to be the subgroup of $\displaystyle \prod_i \Zp [U^{(n)}_i/V_i]$ by the same conditions in Definition \ref{def:omega}. Then by taking the projective limit, we obtain the desired isomorphism.

To simplify the notation, we denote the induced homomorphism $\theta^{+,(n)}$ by $\theta^+$ in the following.

\begin{itemize}
 \item Injectivity.

Let $y \in \Zp[\conj{G^{(n)}}]$ be an element satisfying
       $\theta^+(y)=0$. Then from the hypothesis $(\flat)$ for the
       family $\{(U_i^{(n)}, V_i)\}_i$,\footnote{Recall that an arbitrary
       irreducible representation $\tau$ of $G^{(n)}=G^f \times
       \Gamma/\Gamma^{p^n}$ is isomorphic to $\rho \otimes \omega_n$
       where $\rho$ is an irreducible representation of $G^f$ and
       $\omega_n$ is a character of $\Gamma/\Gamma^{p^n}$. By the
       construction of $U_i$ and $V_i$, an arbitrary irreducible representation
       $\rho$ of $G^f$ is described as $\sum_{i,j} a_i^{(j)}
       \ind{G^f}{U^f_i}{\chi_i^{(j)}}$ where $a_i^{(j)}$ is an integer
       and $\chi_i^{(j)}$ is a character of $U^f_i/V_i^f$ of finite
       index. Then we have $\tau=\rho \otimes \omega_n = \left(
       \sum_{i,j} a_i^{(j)} \ind{G^f}{U^f_i}{\chi_i^{(j)}} \right)
       \otimes \omega_n=\sum_{i,j} a_i^{(j)}
       \ind{G^{(n)}}{U^{(n)}_i}{\chi_i^{(j)} \otimes
       \omega_n}$. $\chi_{i,j} \otimes \omega_n$ is a character of
       $U^{(n)}_i/V_i$ of finite index, hence $\{ U^{(n)}_i, V_i\}_i$
       also satisfies the hypothesis $(\flat)$ for the group $G^{(n)}$.} an arbitrary irreducible representation of
       $G^{(n)}$ is isomorphic to a $\Z$-linear combination (as a
       virtual representation) of $\ind{G^{(n)}}{U^{(n)}_i}{\chi_i}$'s
       where  $\chi_i$ is a character of $U^{(n)}_i/V_i$.  We denote by
       $\tilder{\chi}_i$ the character of the induced representation
       $\ind{G^{(n)}}{U^{(n)}_i}{\chi_i}$. Then we have
\begin{equation*}
\tilder{\chi}_i(y)=\sum_{[g]\in \conj{G^{(n)}}} m_{[g]}\sum_{j, u_j^{-1}g u_j\in U^{(n)}_i} \chi_i(u_j^{-1} g u_j)=\chi_i \circ \theta_i^+(y)
\end{equation*}
by the property of induced characters, where $\{ u_1,\dotsc ,u_{r_i} \}$
       is a system of representatives of $G^{(n)}/U^{(n)}_i$ and 
\begin{align*}
y&=\displaystyle \sum_{[g] \in \conj{G^{(n)}}} m_{[g]} [g] &
 (m_g\in \Zp).
\end{align*}

On the other hand, $\theta^+_i(y)$ vanishes by assumption, hence
       $\tilder{\chi}_i(y)=0$ This implies that $\chi(y)=0$ for an arbitrary irreducible character $\chi$ of $G^{(n)}$. 
Since all irreducible characters of $G^{(n)}$ form a dual basis of the $\line{\Q}_p$-vector space $\line{\Q}_p \otimes_{\Zp} \Zp[\conj{G^{(n)}}]$, we may conclude that $y=0$.

 \item Surjectivity.

Let $(y_0,y_1,\tilder{y_2},y_2,y_3)$ be an arbitrary element of $\Omega^{(n)}$, then each $y_i$ is a $\Zp[\Gamma/\Gamma^{p^n}]$-linear combination of generators of $I_i^{(n)}$:

\begin{align*}
y_0 =& \sum_{a,b,c} \kappa_{abc} \alpha^a \beta^b \gamma^c, \\
y_1 =& \sum_{d,e} \lambda^{(1)}_{de} \delta^d \varepsilon^e+\sum_{a\neq 0,e} \lambda^{(2)}_{ae} \alpha^a \varepsilon^e h_{\delta} \\
 &+ \sum_{c\neq 0,d} \lambda^{(3)}_{cd} \gamma^c \delta^d h_{\varepsilon}+
\sum_{a\neq 0, c\neq 0, d,e} \lambda^{(4)}_{acde} \alpha^a \gamma^c \delta^d \varepsilon^e h_{\varepsilon^c \delta^{-a}},   \\
\tilder{y_2} =&\sum_{b\neq 0, c} \mu^{(1)}_{bc} \beta^b \gamma^c h_{\delta} +\sum_{c,d} p \mu^{(2)}_{cd} \gamma^c \delta^d, \\
y_2 =&  \sum_f p^2 \nu^{(1)}_f \zeta^f + \sum_c p\nu^{(2)}_c \gamma^c h_{\zeta} \\
      &+\sum_{b\neq 0, c\neq 0} \nu^{(3)}_{bc} \beta^b \gamma^c h_{\zeta} +\sum_{b\neq 0,f} p\nu^{(4)}_{bf} \beta^b \zeta^f, \\
y_3 =& \sum_f p^3 \sigma_f^{(1)} \zeta^f +\sum_{e\neq 0} p^2 \sigma_e^{(2)} \varepsilon^e h_{\zeta}+\sum_{c\neq 0} p\sigma_c^{(3)} \gamma^c h_{\varepsilon} h_{\zeta}.
\end{align*}

Note that by the relation (\ref{eq:dep}), $\nu_f^{(1)}(0\leq f\leq p-1)$ and $\nu_0^{(2)}$ are not determined uniquely.

The trace relations in Definition \ref{def:omega} (1) give linear relations among these coefficients. For example, let us describe the (rel-6) explicitly.
Take $\{ \beta^j \mid 0\leq j \leq p-1\}$ as a system of representatives
       of $U_2^{(n)}/U_3^{(n)}$. Then for $u_2 \in U_2^f/V_2^f$,
       $\beta^{-j} u_2 \beta^j$ is contained in $U_3^{(n)}/V_2$ if and
       only if $b(u_2)=0$ where 
\[
 u_2 \equiv \beta^{b(u_2)} \gamma^{c(u_2)} \zeta^{f(u_2)} \qquad \mod
       \langle \varepsilon \rangle.
\]

It is obvious that $\beta^{-j}u_2\beta^j=u_2$ when $b(u_2)=0$. Therefore,
\begin{align*}
\Tr_{\Zp[U_2^{(n)}/V_2]/\Zp[U_3^{(n)}/V_2]} y_2 =& \sum_f p^2( p\nu_f^{(1)}+\nu_0^{(2)}) \zeta^f \\
&+ \sum_{c\neq 0} p^2 \nu_c^{(2)} \gamma^c h_{\zeta}.
\end{align*}
On the other hand,
\begin{equation*}
y_3 \equiv p^2 \sum_{e\neq 0, f} (p\sigma_f^{(1)}+\sigma_e^{(2)}) \zeta^f + \sum_{c\neq 0} p^2 \sigma_c^{(3)} \gamma^c h_{\zeta} \qquad \mod{V_2}.
\end{equation*}
By comparing these coefficients, we have
\begin{eqnarray}
p\nu_f^{(1)} +\nu_0^{(2)} &=& p\sigma_f^{(1)}+\sum_{e\neq 0} \sigma_e^{(2)} \qquad (0\leq f\leq p-1)  \label{eq:rel6}, \\ 
\nu_c^{(2)} &=& \sigma_c^{(3)} \qquad (1\leq c\leq p-1).
\end{eqnarray}

However, because of the relation (\ref{eq:dep}), we may change $\nu_f^{(1)} \rightarrow \nu_f^{(1)}+z$ and $\nu_0^{(2)} \rightarrow \nu_0^{(2)}-pz$ for every $z\in \Zp[\Gamma/\Gamma^{p^n}]$. Therefore in the equation (\ref{eq:rel6}), we may choose $\nu_f^{(1)}$ and $\nu_0^{(2)}$ as they satisfy
\begin{eqnarray*}
\nu_f^{(1)} &=& \sigma_f^{(1)} \qquad (0\leq f\leq p-1), \\
\nu_0^{(2)} &=& \sum_{e\neq 0} \sigma_e^{(2)}.
\end{eqnarray*}

Similarly, we can describe the other trace relations explicitly as follows:
\begin{enumerate}[(rel-1)]
\item $\displaystyle \kappa_{000}=\sum_{d,e} \lambda^{(1)}_{de}$, $\displaystyle \kappa_{a00}=\sum_e \lambda^{(2)}_{ae} \quad (a\neq 0)$, \\
$\displaystyle \kappa_{00c}=\sum_d \lambda^{(3)}_{cd} \quad (c\neq 0)$, \\ 
$\displaystyle \kappa_{a0c}=\sum_{d,e} \lambda^{(4)}_{acde} \quad (a\neq 0,c\neq 0)$,

\item $\displaystyle \kappa_{00c}=\sum_d \mu_{cd}^{(2)}$, $\displaystyle \kappa_{0bc}=\mu^{(1)}_{bc} \quad (b\neq 0)$,

\item $\mu_{bc}^{(1)}=\nu_{bc}^{(3)} \quad (b\neq 0,c\neq 0)$, \\
$\displaystyle \mu_{b0}^{(1)}=\sum_f \nu_{bf}^{(4)} \quad (b\neq 0)$, \\
$\mu_{c0}^{(2)}=\nu_c^{(2)} \quad (c\neq 0)$, $\displaystyle \mu_{00}^{(2)}=\sum_f \nu_f^{(1)}+\nu_0^{(2)}$,

\item $\displaystyle \sum_e \lambda_{de}^{(1)}=\mu_{0d}^{(2)}$, $\lambda^{(3)}_{cd}=\mu_{cd}^{(2)} \quad (c\neq 0)$

\item $\displaystyle \lambda^{(1)}_{00}=\sum_f \sigma^{(1)}_f$, $\lambda^{(1)}_{0e}=\sigma^{(2)}_e \quad (e\neq 0)$, \\
$\lambda^{(3)}_{c0}=\sigma^{(3)}_c \quad (c\neq 0)$,

\item $\nu^{(1)}_f=\sigma^{(1)}_f$, $\displaystyle \nu_0^{(2)}=\sum_{e\neq 0} \sigma_e^{(2)}$, $\nu_c^{(2)}=\sigma^{(3)}_c \quad (c\neq 0)$.
\end{enumerate}

Then we may show by direct calculation, using the explicit trace relations above, that an element of $\Zp[\conj{G^{(n)}}]$
\begin{equation*}
\begin{split}
y =& \sum_{a\neq 0,b\neq 0,c} \kappa_{abc} \ideal{c}_{\mathrm{I}}(a,b,c) + \sum_{a\neq 0,c\neq 0,d,e} \lambda_{acde}^{(4)} \ideal{c}_{\mathrm{II}}(a,c,d,e) \\
   & +\sum_{a\neq 0,e} \lambda^{(2)}_{ae} \ideal{c}_{\mathrm{III}}(a,e) +\sum_{b\neq 0,c\neq 0} \mu^{(1)}_{bc} \ideal{c}_{\mathrm{IV}} (b,c) \\
   & +\sum_{b\neq 0,f} \nu^{(4)}_{bf} \ideal{c}_{\mathrm{V}}(b,f) +\sum_{c\neq 0, d\neq 0} \mu^{(2)}_{cd} \ideal{c}_{\mathrm{VI}}(c,d) + \sum_{c\neq 0} \nu_c^{(2)} \ideal{c}_{\mathrm{VII}}(c) \\
   & +\sum_{d\neq 0,e} \lambda^{(1)}_{de} \ideal{c}_{\mathrm{VIII}}(d,e) + \sum_{e\neq 0} \sigma^{(2)}_e \ideal{c}_{\mathrm{IX}} (e) +\sum_f \nu^{(1)}_f \ideal{c}_{X} (f)
\end{split}
\end{equation*}
satisfies $\theta^+(y)=(y_0,y_1,\tilder{y_2},y_2,y_3)$. 
 \end{itemize} 

Therefore $\theta^+$ induces an isomorphism between $\Zp[\conj{G^{(n)}}]$ and $\Omega^{(n)}$.
\end{proof}

%
%
\section{Translation into the multiplicative theta map}
%
%

In the previous section, we construct the additive theta map $\theta^+$. Now we 
shall translate it into {\it the multiplicative theta map} $\theta$. The main tool for this translation is the integral logarithmic homomorphism introduced in \S 1.3.

Let $\theta_i \colon K_1(\iw{G}) \longrightarrow \iw{U_i/V_i}^{\times}$ be the composition of the norm homomorphism of $K$-groups
\begin{equation*}
\Nr_i\colon K_1(\iw{G}) \longrightarrow K_1(\iw{U_i})
\end{equation*}
and the natural homomorphism $K_1(\iw{U_i}) \longrightarrow K_1(\iw{U_i/V_i})=\iw{U_i/V_i}^{\times}$ induced by $\iw{U_i} \rightarrow \iw{U_i/V_i}$. Set
\begin{equation*}
\theta=(\theta_i)_i \colon K_1(\iw{G}) \longrightarrow \prod_i \iw{U_i/V_i}^{\times}.
\end{equation*}

\begin{propdef}[The Frobenius homomorphism]
For an arbitrary element $g\in G$, set 
\begin{equation} \label{eq:p-power}
\frob{g}=g^p.
\end{equation}

Then $\varphi$ induces a group homomorphism $\varphi \colon G \rightarrow \Gamma$. {\upshape We call this homomorphism} the Frobenius homomorphism.

{\upshape We denote the induced ring homomorphism $\iw{G} \rightarrow \iw{\Gamma}$ by the same symbol $\varphi$, and also call it} the Frobenius homomorphism. 
\end{propdef}

\begin{proof}
 Since the exponent of $G^f$ is exactly equal to $p$, we may decompose $\varphi$ as follows:

\begin{equation*}
\xymatrix{ G \ar[rr]^{\varphi} \ar[dr]_{\pi_G} & & \Gamma \\
  & \Gamma \ar[ur]_{\tilder{\varphi}} & }
\end{equation*}
where $\tilder{\varphi} \colon \Gamma \rightarrow \Gamma$ is the Frobenius endomorphism of $\Gamma$ defined by $t \mapsto t^p$, and $\pi_G$ is the canonical surjection. Hence $\varphi$ is clearly a group homomorphism.
\end{proof}

\begin{rem}
In general, the correspondence (\ref{eq:p-power}) does not induce a group homomorphism $\varphi \colon G \longrightarrow G$.
\end{rem}

In the following, we denote by the same symbol $\varphi$ the induced ring homomorphism $\iw{U_i/V_i} \rightarrow \iw{\Gamma}$.

\begin{defn} \label{def:psi}
 We define $\Psi$ to be the subgroup of $\displaystyle \prod_i \iw{U_i/V_i}^{\times}$ consisting of all elements $(\eta_0, \eta_1, \tilder{\eta_2}, \eta_2, \eta_3 )$ satisfying the following two conditions:

\begin{enumerate}
 \item (norm relations)
\begin{enumerate}[(rel-1)]
  \item $\Nr_{\iw{U_0/V_0}/\iw{U_1/V_0}}\eta_0 \equiv \eta_1$,
  \item $\Nr_{\iw{U_0/V_0}/\iw{\tilder{U_2}/V_0}}\eta_0 \equiv \tilder{\eta_2}$,
  \item $\Nr_{\iw{\tilder{U_2}/\tilder{V_2}}/\iw{U_2/\tilder{V_2}}}\tilder{\eta_2} \equiv \eta_2$,

  \item $\Nr_{\iw{U_1/\tilder{V_2}}/\iw{U_1\cap \tilder{U_2}/\tilder{V_2}}}\eta_1 \equiv \Nr_{\iw{\tilder{U_2}/\tilder{V_2}}/\iw{U_1\cap \tilder{U_2}/\tilder{V_2}}}\tilder{\eta_2}$,
  \item $\Nr_{\iw{U_1/V_1}/\iw{U_3/V_1}}\eta_1 \equiv \eta_3$,
  \item $\Nr_{\iw{U_2/V_2}/\iw{U_3/V_2}}\eta_2 \equiv \eta_3$,
\end{enumerate}
(See Figure \ref{fg:relations}).

 \item (congruences)
  \begin{align*}
	\eta_1 &\equiv \frob{\eta_0} & & \mod{I_1}, &
	\tilder{\eta_2} &\equiv \frob{\eta_0} & & \mod{\tilder{I_2}}, \\
	\eta_2 & \equiv \frob{\eta_0}^p && \mod{I_2}, &
	\eta_3 &\equiv \frob{\eta_0}^{p^2} && \mod{I_3}.
  \end{align*}  
\end{enumerate}
\end{defn}

Now let us recall the definition of norm maps of commutative rings: let
$R$ be a commutative ring and $R'$ a commutative $R$-algebra. Assume
that $R'$ is free and finitely generated as an $R$-module. Then we
define $\Nr_{R'/R} (y)$ to be the determinant of the multiplication-$y$ homomorphism.

\begin{prop} \label{prop:theta}
 The homomorphism $\theta$ induces a surjection
\begin{equation} \label{eq:aksurj}
 \theta\colon K_1(\iw{G}) \rightarrow \Psi.
\end{equation}

{\upshape In other words, $\theta$ induces} the $($multiplicative$)$
 theta map for $G$ $($in the sense of Definition $\ref{def:theta})$.
\end{prop}

In the rest of this section, we prove Proposition \ref{prop:theta} by 
using the additive theta map $\theta^+$ and the integral logarithmic
homomorphisms. Since the integral logarithmic homomorphisms are defined
only for group rings of finite groups,  we fix $n\geq 1$ and construct an isomorphism for finite quotients
\begin{equation*}
\theta^{(n)} \colon K_1(\Zp[G^{(n)}]) / SK_1 ( \Zp[ G^{(n)}]) \xrightarrow{\simeq} \Psi^{(n)} \left( \subseteq \prod_i \Zp[U^{(n)}_i/V_i] \right)
\end{equation*}
where $G^{(n)}=G_f \times \Gamma/\Gamma^{p^n}, U^{(n)}_i=U_i^f \times
\Gamma/\Gamma^{p^n}$, and $\Psi^{(n)}$ is the subgroup of $\prod_i \Zp[U_i^{(n)}/V_i]$ defined by the same conditions of
Definition \ref{def:psi}. Then we obtain the surjection (\ref{eq:aksurj}) by taking the projective limit.

To simplify the notation, we denote the theta map for $G^{(n)}$ by
$\theta$ in \S 5.1, \S 5.2, and \S 5.3. In \S 5.4, where we take the
projective limit of theta maps, we denote the theta map for of
$G^{(n)}$ by $\theta^{(n)}$ again.

%
\subsection{Logarithmic isomorphisms}
%

In the following three subsections, we fix $n\geq 1$.

For each $i\geq 1$, we define the ideal $J^{(n)}_i$ of $\Zp[U^{(n)}_i/V_i]$ as follows:
\begin{align*}
J^{(n)}_1 &= (h_{\varepsilon^j \delta^k}\quad (0\leq j,k \leq p-1))  , & 
\tilder{J_2}^{(n)} &= (p, h_{\delta}), \\
J^{(n)}_2 &=J^{(n)}_3= (p, h_{\zeta}). &
\end{align*}

Note that $J^{(n)}_i$ contains $I^{(n)}_i$. Since $h_{\varepsilon^0 \delta^0}=p$, $J^{(n)}_1$ also contains $p$.

\begin{lem} \label{lem:logj}
For each $i\geq 1$, $1+J^{(n)}_i$ is a multiplicative group and the $p$-adic logarithmic homomorphism induces an isomorphism $1+J^{(n)}_i \xrightarrow{\simeq} J^{(n)}_i$.
\end{lem} 

\begin{proof}
Since $J^{(n)}_i$ is an ideal, $(J^{(n)}_i)^2 \subseteq J^{(n)}_i$. Therefore $1+J^{(n)}_i$ is closed under multiplication.

It is easy to calculate that $h_{\delta}^2=ph_{\delta}$ and $h_{\zeta}^2=ph_{\zeta}$. Then for $\tilder{J_2}^{(n)}, J_2^{(n)}$ and $J_3^{(n)}$, we have $(J_i^{(n)})^2 = pJ^{(n)}_i$. 
For $J_1^{(n)}$, we have 
\begin{equation*}
h_{\varepsilon^j \delta^k} h_{\varepsilon^{j'} \delta^{k'}} =
\begin{cases}
h_{\varepsilon, \delta} & \text{if $(j,k)$ and $(j',k')$ are $\mathbb{Z}$-linearly independent}, \\
ph_{\varepsilon^j\delta^k} & \text{otherwise},
\end{cases}
\end{equation*}
where 
\begin{equation*} 
h_{\varepsilon, \delta}= \sum_{0\leq \ell,\ell'\leq p-1} \varepsilon^{\ell} \delta^{\ell'}.
\end{equation*}

For each $0\leq j,k \leq p-1$, we have $h_{\varepsilon,
 \delta}h_{\varepsilon^j \delta^k}=ph_{\varepsilon, \delta}$ and
 $h_{\varepsilon, \delta}^2=p^2h_{\varepsilon, \delta}$ (note that the
 exponent of $U_1^f/V_1^f$ is $p$). Thus we have 
\[
 (J_1^{(n)})^3 =p(J_1^{(n)})^2 \subseteq pJ^{(n)}_1. 
\]

Set $N=3$ for $i=1$ and $N=2$ for other $i$.

\begin{itemize}
\item The existence of inverse elements.

By the argument above, we have $y^m \in p^{m-N+1}J_i^{(n)}$ for an arbitrary element $y\in J^{(n)}_i$ and every $m\geq N$. Therefore,
\begin{equation} 
(1+y)^{-1}=\sum_{m\geq0} (-1)^m y^m
\end{equation}
converges in $J_i^{(n)}$ $p$-adically.

\item Convergence of logarithms.

It is obvious that $y^m/m ^in J_i^{(n)}$ for $1\leq m<N$, and we have 
\begin{equation*}
\dfrac{y^m}{m} \in \dfrac{p^{m-N+11}}{m}\cdot J^{(n)}_i\subseteq J^{(n)}_i 
\end{equation*}
for an arbitrary $y\in J_i^{(n)}$ and every $m\geq N$ (here we use the
      fact that $p\neq 0$). Thus the logarithm $\log (1+y)=\sum_{m\geq 1} (-1)^{m-1} (y^m/m)$ converges $p$-adically in $J^{(n)}_i$. 

Since $\left\{ ( J_i^{(n)} )^m \right\}^N=p^m(J_i^{(n)})^m(N-1) \subseteq p^m
      ( J_i^{(n)})^m$, we may also show that
      $1+\left(J_i^{(n)}\right)^m$ is a subgroup of $1+J_i^{(n)}$ and
\[ \log \left( 1+\left( J_i^{(n)} \right)^m \right) \subseteq \left(
      J_i^{(n)} \right)^m \qquad \text{for every }m\geq 1,
\]
by the same argument as above.

\item Logarithmic isomorphisms.

First note that $p$-adic topology and $J_i^{(n)}$-adic topology are the same on $J_i^{(n)}$: this is because $p^{m+1}J_i^{(n)} \subseteq (J^{(n)}_i)^{m+N-1} \subseteq p^m J^{(n)}_i$ for every $m\geq 1$. Since $J^{(n)}_i$ is $p$-adically complete, it is also $J^{(n)}_i$-adically complete. Therefore to show that the logarithm induces an isomorphism $1+J^{(n)}_i \xrightarrow{\simeq} J^{(n)}_i$, it is sufficient to show that it induces an isomorphism 
\begin{equation*}
\begin{split}
\log\colon (1+(J^{(n)}_i)^m)/(1+(J^{(n)}_i)^{m+1}) \xrightarrow{\simeq} (J^{(n)}_i)^m/&(J^{(n)}_i)^{m+1} \\ & ; 1+y \mapsto y
\end{split}
\end{equation*}
for each $m\geq 1$, and to prove this, it suffices to show that
\begin{align} \label{eq:yp}
y^{p^k}/{p^k} \in \left( J^{(n)}_i \right)^{m+1} & \text{for every $1+y\in 1+\left(J^{(n)}_i \right)^m$ and $k\geq 1$}.
\end{align}

However, since $y^{N+1} \in p\left( J_i^{(n)} \right)^{m+1}$, we obtain
      $y^p \in p\left( J_i^{(n)} \right)^{m+1}$ (recall $p\neq 2,3$ by assumption). This implies (\ref{eq:yp}).
\end{itemize}
\end{proof}

\begin{lem} \label{lem:logi}
For each $i\geq 1$, $1+I^{(n)}_i$ is a multiplicative group and the $p$-adic logarithmic homomorphism induces an isomorphism $1+I^{(n)}_i \xrightarrow{\simeq} I^{(n)}_i$.
\end{lem} 

\begin{rem}
Since each $I^{(n)}_i$ is \textbf{not} an ideal but only a $\mathbb{Z}_p$-submodule of $\Zp[U^{(n)}_i/V_i]$, it is not trivial even that $1+I^{(n)}_i$ is closed under multiplication.
\end{rem}

\begin{proof}
By direct calculation, we have
\begin{align*}
(I^{(n)}_1)^2 &= [p\alpha^a \gamma^c \delta^d \varepsilon^e h_{\varepsilon^c \delta^{-a}}, p \delta^d \varepsilon^e h_{\varepsilon^c \delta^{-a}}, \alpha^a \gamma^c h_{\varepsilon, \delta} ((a,c)\neq (0,0))]_{\Zp[\Gamma/\Gamma^{p^n}]}, \\
(\tilder{I_2}^{(n)})^2 &=  [p\beta^b \gamma^c h_{\delta}, p^2\gamma^c \delta^d]_{\Zp[\Gamma/\Gamma^{p^n}]}, \\
(I^{(n)}_2)^2 &= [p \beta^b \gamma^c h_{\zeta}, p^2 \beta^b \zeta^f]_{\Zp[\Gamma/\Gamma^{p^n}]}, \\
(I^{(n)}_3)^2 &= [p^6\zeta^f, p^5 \varepsilon^e h_{\zeta}, p^4 \gamma^c h_{\varepsilon}
 h_{\zeta}]_{\Zp[\Gamma/\Gamma^{p^n}]},
\end{align*}

These results imply that $(I^{(n)}_i)^2 \subseteq I^{(n)}_i$.\footnote{For
 $I^{(n)}_1$, recall that $I^{(n)}_1$ is also stable under the
 multiplication of $\delta$ and $\varepsilon$. Therefore we have
 $\displaystyle p\delta^d \varepsilon^e h_{\varepsilon^c
 \delta^{-a}}=\left( \sum_{\ell=0}^{p-1} \delta^{d-\ell a}
 \varepsilon^{e+\ell c} \right) p \in I^{(n)}_1$ and $\displaystyle \alpha^a \gamma^c h_{\varepsilon, \delta}=\left( \sum_{\ell=0}^{p-1} (\varepsilon^{c'} \delta^{-a'})^{\ell} \right) \alpha^a \gamma^c h_{\varepsilon^c \delta^{-a}} \in I^{(n)}_1$ where $(a',c')\in \Z^2$ is an arbitrary element $\Z$-linearly independent of $(a,c)$. Therefore we have $(I_1^{(n)})^2 \subseteq I_1^{(n)}$.} Therefore, each $1+I^{(n)}_i$ is stable under multiplication. 

Now let 
\begin{align*}
I_1' & = I^{(n)}_1, \\
\tilder{I_2}' & = [\beta^b \gamma^c h_{\delta}, p\gamma^c \delta^d]_{\Zp[\Gamma/\Gamma^{p^n}]}, \\
I_2' & = [p\beta^b \zeta^f, \beta^b \gamma^c h_{\zeta}]_{\Zp[\Gamma/\Gamma^{p^n}]}, \\
I_3' &= [p^3 \zeta^f, p^2 \varepsilon^e h_{\zeta}, p\gamma^c h_{\varepsilon} h_{\zeta}]_{\Zp[\Gamma/\Gamma^{p^n}]}.
\end{align*}

Note that $I_i' \supseteq I^{(n)}_i$, and by simple calculation we have
 $(I^{(n)}_1)^3 =(I_1')^3 \subseteq p(I'_1)^2$ for $i=1$ and
 $(I^{(n)}_i)^2=(I_i')^2\subseteq pI_i'$ for other $i$. 

Then we may show the existence of inverse elements of $1+I'_i$ and the
 logarithmic isomorphism $\log \colon 1+I'_i \xrightarrow{\simeq} I_i'$
 for each $i$ by the same argument as the proof of Lemma \ref{lem:logj}
 by replacing $J_i$ by $I_i'$. Since $(I_i')^2 \subseteq I_i^{(n)}
 \subseteq I_i'$ and $I^{(n)}_i$ is a closed subset of $I_i'$, the power
 series $(1+x)^{-1}$ and the $p$-adic logarithm $\log (1+x)$ on $1+I^{(n)}_i$ converges into $I^{(n)}_i$. Moreover since $\log\colon 1+(I_i')^k \rightarrow (I_i')^k$ is an isomorphism for $k=1,2$ and 
\begin{equation*}
 \log \colon 1+I^{(n)}_i/1+(I_i')^2 \longrightarrow I^{(n)}_i/(I_i')^2 ; 1+y_i \mapsto y_i \quad \mod{(I_i')^2}
\end{equation*}
is also an isomorphism, we may conclude that $\log\colon 1+I^{(n)}_i \rightarrow I^{(n)}_i$ is an isomorphism.
\end{proof}   

%
\subsection{$\Psi^{(n)}$ contains the image of $\theta$}
%

\begin{lem} \label{lem:lognorm}
The following diagram commutes for each $i$ and  $n\geq 1 \colon$
\begin{equation*}
\begin{CD}
K_1(\Zp[G^{(n)}]) @>\log>> \Zp[\conj{G^{(n)}}] \\
@V\Nr_{\Zp[G^{(n)}]/\Zp[U_i^{(n)}]}VV  @VV\Tr_{\Zp[G^{(n)}]/\Zp[U_i^{(n)}]}V \\
K_1(\Zp[U_i^{(n)}]) @>>\log> \Zp[\conj{U_i^{(n)}}]
\end{CD}
\end{equation*}
\end{lem}

\begin{proof}
By the proof of \cite{Oli-Tay}, Theorem 1.4., the following diagram 
commutes.
\begin{equation*}
\begin{CD}
K_1(\Zp[G^{(n)}]) @>\log>> \Zp[\conj{G^{(n)}}] \\
@V\Nr_{\Zp[G^{(n)}]/\Zp[U_i^{(n)}]}VV  @VVR'V \\
K_1(\Zp[U_i^{(n)}]) @>>\log> \Zp[\conj{U_i^{(n)}}]
\end{CD}
\end{equation*}
where $R'\colon \Zp[\conj{G^{(n)}}]\rightarrow \Zp[\conj{U_i^{(n)}}]$ is 
defined as follows: for an arbitrary $x\in G^{(n)}$, let $\{ u_1', \dotsc , u_{s_i}' \}$ be a set of representatives of the double coset decomposition 
$\langle x \rangle \backslash G^{(n)}/U_i^{(n)} $, and let 
\begin{equation*}
J=\{ j \mid 1\leq j\leq s_i, \, {u_j'}^{-1} x u_j' \in U_i^{(n)} \}.
\end{equation*}

Then we define $\displaystyle R'(x)=\sum_{j\in J} {u_j'}^{-1}xu_j'$.

It suffices to show that 
\begin{equation} \label{eq:r'tr}
R'=\Tr_{\Zp[G^{(n)}]/\Zp[U_i^{(n)}]}, 
\end{equation}
but this is not difficult at all.

Let $\{ u_1', \dotsc , u_{s_i}' \}$ be as above. Then we can write 
\begin{align*}
G^{(n)} &= \bigsqcup_j \langle x \rangle u'_j U_i^{(n)} = \bigcup_{k=0}^{p-1} \bigsqcup_j x^k u'_j U_i^{(n)},
\end{align*}
therefore it is clear that $\{x^k u'_j\}_{0\leq k\leq p-1, j}$ contains
 a set of representatives of the right coset decomposition
 $G^{(n)}/U_i^{(n)}$. Moreover, since each $u'_jU_i^{(n)}$ is disjoint,
 $\{u'_j\}_j$ is a subset of representatives of
 $G^{(n)}/U_i^{(n)}$.

\begin{enumerate}[({Case}-1)]
\item $\displaystyle G^{(n)}=\bigsqcup_j u'_j U_i^{(n)}$.

In this case, the desired equation $(\ref{eq:r'tr})$ is obvious by definition of $R'$ and
      $\Tr$.

\item $\displaystyle G^{(n)}=\bigsqcup_{k=0}^{p-1} \bigsqcup_j x^k u'_j U_i^{(n)}$.

Then by simple calculation
\begin{align*}
R'(x) &= \sum_{j\in J} {u'_j}^{-1} x u'_j = \dfrac{1}{p}\sum_{k=0}^{p-1} \sum_{j\in J} (x^ku'_j)^{-1} x
 (x^ku'_j) \\
 &= \dfrac{1}{p} \Tr_{\Zp[G^{(n)}]/\Zp[U_i^{(n)}]} x.
\end{align*}

On the other hand, recall that $j \in J$ implies that $xu'_j$ is an
      element of $u'_jU_i^{(n)}$. However, this is impossible since $\{x^k
      u'_j\}_{0\leq k\leq p-1,j}$ is a set of representatives of
      $G^{(n)}/U_i^{(n)}$ in this case. Therefore $J=\emptyset$ and $R'(x)=\Tr_{\Zp[G^{(n)}]/\Zp[U_i^{(n)}]}(x)=0$.
\end{enumerate}
\end{proof}

\begin{prop} \label{prop:image}
For an arbitrary element $\eta \in K_1(\Zp[G^{(n)}])$, the following
 equations hold$\colon$
\begin{align*}
\theta^+_1\circ \Gamma_{G^{(n)}}(g) &= \log \dfrac{\theta_1(\eta)}{\frob{\theta_0(\eta)}}, \\
\tilder{\theta}^+_2 \circ \Gamma_{G^{(n)}}(g) &= \log \dfrac{\tilder{\theta}_2(\eta)}{\frob{\theta_0(\eta)}}, \\
\theta^+_2\circ \Gamma_{G^{(n)}}(g) &= \log \dfrac{\theta_2(\eta)}{\frob{\theta_0(\eta)}^p}, \\
\theta^+_3\circ \Gamma_{G^{(n)}}(g) &= \log \dfrac{\theta_3(\eta)}{\frob{\theta_0(\eta)}^{p^2}},
\end{align*}
where $\Gamma_{G^{(n)}} \colon K_1(\Zp[G^{(n)}])\rightarrow
 \Zp[\conj{G^{(n)}}]$ is the integral logarithm for $G^{(n)}$
 $($Proposition-Definition $\ref{pd:intlog})$.
\end{prop}

\begin{proof}
This proposition follows from Lemma \ref{lem:lognorm} and by following commutative diagrams:

\[ 
\begin{CD}
\Zp [\conj{G^{(n)}}] @>{\theta^+_0}>> \Zp[U^{(n)}_0/V_0] @. \qquad @. \Zp[\conj{G^{(n)}}] @>{\theta^+_0}>> \Zp[U^{(n)}_0/V_0] \\
@V{\frac{1}{p}\varphi}VV  @VV{\varphi}V @.  @V{\frac{1}{p}\varphi}VV @VV{\varphi}V \\
\Zp [\conj{G^{(n)}}] @>>\theta^+_1> \Zp[U_1^{(n)}/V_1] @. \qquad @. \Zp[\conj{G^{(n)}}] @>>\tilder{\theta}^+_2> \Zp[\tilder{U_2}^{(n)}/\tilder{V_2}]
\end{CD}
\]

\[ 
\begin{CD}
\Zp [\conj{G^{(n)}}] @>{\theta^+_0}>> \Zp[U_0^{(n)}/V_0] @. \qquad @. \Zp[\conj{G^{(n)}}] @>{\theta^+_0}>> \Zp[U_0^{(n)}/V_0] \\
@V{\frac{1}{p}\varphi}VV  @VV{p\varphi}V @.  @V{\frac{1}{p}\varphi}VV @VV{p^2\varphi}V \\
\Zp [\conj{G^{(n)}}] @>>\theta^+_2> \Zp[U_2^{(n)}/V_2] @. \qquad @. \Zp[\conj{G^{(n)}}] @>>\theta^+_3> \Zp[U_3^{(n)}/V_3]
\end{CD}
\]

These diagrams are easily checked for an arbitrary element of $G^{(n)}$. Then we may calculate as
\begin{align*}
\theta^+_2 \circ \Gamma_{G^{(n)}} (\eta) &= \theta^+_2 \left(\log \eta -\dfrac{1}{p} \frob{\log (\eta)} \right) \\
 &= \theta^+_2 (\log (\eta))-p\frob{\theta^+_0 (\log (\eta)}  \\
 &=\log(\theta_2(\eta))-p\frob{\log (\theta_0(\eta))} \qquad  \\
 &= \log \dfrac{\theta_2(\eta)}{\frob{\theta_0(\eta)}^p},
\end{align*}
here we use the diagram above for the second equality, and use Lemma
 \ref{lem:lognorm} for the third equality. The other equations may be
 derived similarly.

Note that for an arbitrary $y \in \Zp[U_i^{(n)}/V_i]$, we have
\begin{equation*}
\frob{\log (y)} = \log \frob{y} 
\end{equation*}
by the definition of the $p$-adic logarithms and the fact that the 
Frobenius homomorphism $\varphi$ is a ring homomorphism on each $\Zp[U_i^{(n)}/V_i]$.
\end{proof}

\begin{lem} \label{lem:congj}
For an arbitrary element $\eta \in K_1(\Zp[G^{(n)}])$, the following
 congruences hold$\colon$
\begin{align*}
\theta_1(\eta) & \equiv \frob{\theta_0(\eta)} & \mod{J^{(n)}_1}, \\
\tilder{\theta}_2 (\eta) & \equiv \frob{\theta_0(\eta)} & \mod{\tilder{J_2}^{(n)}}, \\
\theta_2 (\eta) & \equiv \frob{\theta_0(\eta)}^p & \mod{J^{(n)}_2}, \\
\theta_3 (\eta) & \equiv \frob{\theta_0(\eta)}^{p^2} & \mod{J^{(n)}_3}.
\end{align*}
\end{lem}

\begin{proof}
We calculate the image of the theta map as in \S 1.2.

In the following proof, we use the same notation $\eta$ for a lift of
 $\eta$ to $\Zp[G^{(n)}]^{\times}$ (Proposition \ref{prop:sloc}).

\begin{itemize}
\item The congruences for $\theta_1, \tilder{\theta}_2$.

We only show the congruence for $\tilder{\theta}_2$. We may prove the
      congruence for $\theta_1$ in the same manner.

Since $\{ \alpha^i \mid 0\leq i\leq p-1 \}$ is a $\Zp[\tilder{U_2}^{(n)}]$-basis of $\Zp[G^{(n)}]$, $\eta$ is described as a $\Zp[\tilder{U_2}^{(n)}]$-linear combination as follows:
\begin{equation*}
\eta = \sum_{i=0}^{p-1} \eta_i \alpha^i, \qquad \eta_i \in \Zp[\tilder{U_2}^{(n)}].
\end{equation*}

Then we have
\begin{align*}
\alpha^j \eta &= \sum_{i=0}^{p-1} (\alpha^j \eta_i \alpha^{-j}) \cdot \alpha^{i+j} \\
&= \sum_{i=0}^{p-1} \nu_j (\eta_{i-j}) \alpha^i
\end{align*}
where $\nu_j \colon \Zp[\tilder{U_2}^{(n)}] \rightarrow \Zp[\tilder{U_2}^{(n)}]; x \mapsto \alpha^j x \alpha^{-j}$. Here we consider the sub-index of $\eta_i$ as an element of $\Z/p\Z$. By
      abuse of notation, we denote the image of $\eta_i$ in
      $\Zp[\tilder{U_2}^{(n)}/\tilder{V_2}]$ by the same symbol
      $\eta_i$. Then we may compute $\tilder{\theta}_2(\eta)$ as follows: 
\begin{align*}
\tilder{\theta}_2(\eta) &= \det (\nu_j (\eta_{i-j}))_{i,j} \\
&=\sum_{\sigma \in \ideal{S}_{p}} \sgn (\sigma) \prod_{j=0}^{p-1} \nu_j (\eta_{\sigma(j)-j}) \\
&= \sum_{\sigma \in \ideal{S}_{p}} P_{\sigma}
\end{align*}
where $\ideal{S}_p$ is the permutation group of $\{ 0,1,\dotsc ,p-1\}$
      and $P_{\sigma}= \sgn(\sigma) \prod_{j=0}^{p-1} \nu_j(\eta_{\sigma(j)-j})$.

Now we have
\begin{align*}
\nu_k (P_{\sigma}) &= \sgn(\sigma) \prod_j \nu_{j+k}( \eta_{\sigma(j)-j} ) \\ 
&= \sgn(\sigma) \prod_j \nu_j (\eta_{(\sigma(j-k)+k)-j}).
\end{align*}

First Suppose that $\sigma$ does not satisfy 
\begin{equation} \label{eq:sigma}
\sigma(j-k)+k=\sigma(j) \qquad \text{for each } k\in \Z/p\Z,
\end{equation}
then
\begin{equation*}
\tau_k(j)=\sigma(j-k)+k, \qquad k\in \Z/p\Z
\end{equation*}
are distinct elements of $\ideal{S}_p$ and we have
\begin{equation} \label{eq:tauksigma}
\nu_k(P_{\sigma})=P_{\tau_k}.
\end{equation}

Here we use the fact $\sgn(\sigma)=\sgn(\tau_k)$. This follows from the equation $\tau_k=c_k \circ \sigma \circ c_k^{-1}$ where $c_k(i)=i+k$. 

The equation (\ref{eq:tauksigma}) implies 
\begin{equation*}
\sum_{k=0}^{p-1} P_{\tau_k} \in \sum_{k=0}^{p-1} \nu_k \left( \Zp[\tilder{U_2}^{(n)}/\tilder{V_2}] \right).
\end{equation*}

Note that each $\nu_k$ is a $\Zp[\tilder{U_2'}^{(n)}/\tilder{V_2}]$-linear map where 
\begin{equation*}
\tilder{U_2'}^{(n)}=\langle \gamma, \delta, \varepsilon, \zeta \rangle \times 
\Gamma/\Gamma^{p^n}
\end{equation*}
and $\Zp[\tilder{U_2}^{(n)}/\tilder{V_2}]$ is generated by $\{ \beta^b \mid b\in \Z/p\Z \}$ over $\Zp[\tilder{U_2'}^{(n)}/\tilder{V_2}]$. Since
\begin{equation*}
\sum_{k=0}^{p-1} \nu_k(\beta^b)=
\begin{cases}
p & \text{if } b=0, \\
\beta^b (1+\delta+\cdots \delta^{p-1}) & \text{otherwise},
\end{cases}
\end{equation*}
we may conclude that
\begin{equation} \label{eq:sumJ_2tilder}
\sum_{k=0}^{p-1} \nu_k \left( \Zp[\tilder{U_2}^{(n)}/\tilder{V_2}] \right)\subseteq \tilder{J_2}^{(n)}.
\end{equation}

This implies $\displaystyle \sum_{k=0}^{p-1} P_{\tau_k} \in \tilder{J_2}^{(n)}$.

Next assume $\sigma$ satisfies (\ref{eq:sigma}). It is easy to see that all permutations satisfying (\ref{eq:sigma}) are 
\begin{equation*}
c_h(j)=j+h \qquad (0\leq h\leq p-1).
\end{equation*}

For $\sigma =c_h$, we have
\begin{align*}
P_{c_h} = \sgn(c_h) \prod_{i=0}^{p-1} \nu_j(\eta_{c_h(j)-j})= \prod_{i=0}^{p-1} \nu_j(\eta_h).
\end{align*}

Note that since $c_h\, (h\neq 0)$  is a cyclic permutation of degree $p$, $c_h$ is an even permutation, therefore $\sgn(c_h)=1$ for $0\leq h\leq p-1$.

\begin{quotation}
{\bfseries Claim.} For an arbitrary element $x \in \Zp[\tilder{U_2}^{(n)}/\tilder{V_2}]$, 
\begin{equation*}
\prod_{j=0}^{p-1} \nu_j(x) \equiv \frob{x} \qquad \mod{\tilder{J_2}^{(n)}}.
\end{equation*}
\end{quotation}

If this claim is true, we have
\begin{align*}
\tilder{\theta_2}(\eta) &\equiv \sum_{h=0}^{p-1} P_{c_h} &
 \mod{\tilder{J_2}^{(n)}} \\
&\equiv \sum_{h=0}^{p-1} \prod_{j=0}^{p-1} \nu_j (\eta_h) & \mod{\tilder{J_2}^{(n)}} \\
&\equiv \sum_{h=0}^{p-1} \frob{\eta_h} & \mod{\tilder{J_2}^{(n)}} \\
&\equiv \sum_{h=0}^{p-1} \frob{\eta_h \beta^h} & \mod{\tilder{J_2}^{(n)}} \\
&\equiv \frob{\eta} & \mod{\tilder{J_2}^{(n)}},
\end{align*}
and the congruence for $\tilder{\theta_2}$ holds.

Now let
\begin{equation*}
x=\sum_{\ell=1}^N x_{\ell}
\end{equation*}
where each $x_{\ell}$ is a monomial of
      $\Zp[\tilder{U_2}^{(n)}/\tilder{V_2}]$. Then 
\begin{equation} \label{eq:expansion}
\begin{aligned}
\prod_{j=0}^{p-1} \nu_j (x) &= \prod_{j=0}^{p-1} \sum_{\ell=1}^N \nu_j(x_{\ell}) \\
&= \sum_{1\leq {\ell}_0,\dotsc ,{\ell}_{p-1} \leq N} \nu_0(x_{\ell_0}) \nu_1(x_{\ell_1}) \cdot \cdots \cdot \nu_{p-1} (x_{\ell_{p-1}}) \\
&= \sum_{1\leq \ell_0, \dotsc , \ell_{p-1}\leq N} Q_{\ell_0, \dotsc ,\ell_{p-1}}.
\end{aligned}
\end{equation}

Here we set $Q_{\ell_0, \dotsc , \ell_{p-1}}=\nu_0(x_{\ell_0}) \nu_1(x_{\ell_1}) \cdot \cdots \cdot \nu_{p-1} (x_{\ell_{p-1}})$.

If $\ell_0=\ell_1=\dotsc =\ell_{p-1}=\lambda \, (\text{constant})$, we obtain $Q_{\lambda,\dotsc,\lambda}=\frob{x_{\lambda}}$ by easy calculation (here we use that the exponent of $G^f$ is $p$). Otherwise, each
\begin{equation*}
Q_{\ell_k,\ell_{k+1},\dotsc, \ell_{p-1}, \ell_0, \ell_1, \dotsc ,\ell_{k-1}}=\nu_{p-k} (Q_{\ell_0,\dotsc ,\ell_{p-1}})
\end{equation*}
is a distinct term in the expansion (\ref{eq:expansion}), so we have
\begin{equation*}
\sum_{k=0}^{p-1} Q_{\ell_k,\ell_{k+1},\dotsc, \ell_{p-1}, \ell_0, \ell_1, \dotsc ,\ell_{k-1}}  \in \sum_{k=0}^{p-1} \nu_k \left( \Zp[\tilder{U_2}^{(n)}/\tilder{V_2}] \right),
\end{equation*}
which is contained in $\tilder{J_2}^{(n)}$ by (\ref{eq:sumJ_2tilder}). Hence,
\begin{align*}
\prod_{j=0}^{p-1} \nu_j(x) &\equiv \sum_{\lambda=1}^N \frob{x_{\lambda}} & \mod{\tilder{J_2}^{(n)}} \\
& \equiv \frob{x} & \mod{\tilder{J_2}^{(n)}},
\end{align*}
and Claim is proven.

\item The congruence for $\theta_2$.

The basic strategy is essentially the same as the congruence for $\tilder{\theta_2}$, but since $U_2$ is
      not a normal subgroup of $G$, the calculation is much more
      complicated.

Since $\{ \alpha^i \delta^j \mid 0\leq i,j \leq p-1 \}$ is a system of
      representatives of $U^{(n)}_2\backslash G^{(n)}$ and $\{ \beta^k
      \mid 0\leq k\leq p-1 \}$ is that of $U^{(n)}_3\backslash U^{(n)}_2$, we can describe $\eta$ in the following form:
\begin{equation*}
\eta = \sum_{0\leq i,j,k \leq p-1} (\eta_{i,j}^{(k)} \beta^k) \alpha^i
 \delta^j 
\end{equation*}
where $\eta_{i,j}^{(k)}$ is an element of $\Zp[U_3^{(n)}]$. Then we have
\begin{align*}
\alpha^{\ell} \delta^m \eta &= \sum_{0\leq i,j,k\leq p-1} \nu_{\ell,m}
 (\eta_{i,j}^{(k)} \beta^k) \alpha^{i+\ell} \delta^{j+m} \\
&= \sum_{0\leq i,j,k \leq p-1} \{ \nu_{\ell,m} (\eta_{i,j}^{(k)})
 \delta^{\ell k}\beta^k \} \alpha^{i+\ell} \delta^{j+m} \\
&= \sum_{0\leq i,j \leq p-1} \left( \sum_{k=0}^{p-1} \nu_{\ell,m}
 (\eta_{i-\ell,j-m-\ell k}^{(k)})
 \beta^k \right) \alpha^i \delta^j
\end{align*}
where $\nu_{\ell,m}(x)= (\alpha^{\ell} \delta^m)x(\alpha^{\ell} \delta^m)^{-1}$. We use the fundamental relations $[\alpha, \delta]=1$ and $[\alpha, \beta]=\delta$ for the second equality. 

Here we regard the sub-index $i,j$ of $\eta_{i,j}^{(k)}$ as an element
      of $(\Z/p\Z)^{\oplus2}$, and the upper-index $(k)$ of it as an element of $\Z/p\Z$. We denote the image of $\eta_{i,j}^{(k)}$ in $\Zp[U_3^{(n)}/V_2^{(n)}]$ by the same symbol $\eta_{i,j}^{(k)}$. Then we obtain
\begin{align*}
& \qquad \theta_2(\eta) = \det \left( \sum_{k=0}^{p-1} \nu_{\ell, m}
 (\eta^{(k)}_{(i,j)-(\ell,m)-(0,\ell k)}) \beta^k \right)_{(i,j), (\ell,m)}  \\
&= \sum_{\sigma \in \ideal{S}_{\square}} \sgn(\sigma)
 \prod_{(\ell,m)\in (\Z/p\Z)^{\oplus2}} \left( \sum_{k=0}^{p-1}
 \nu_{\ell,m}(\eta_{\sigma(\ell,m)-(\ell,m)-(0,\ell k)}^{(k)}) \beta^k \right) \\
&= \sum_{\sigma \in \ideal{S}_{\square}} P_{\sigma}
\end{align*}
where $\ideal{S}_{\square}$ is the permutation group of $\{ (\ell,m)
      \mid 0\leq \ell, m \leq p-1 \}$.

Now by using $[\beta, \delta]=1$, we have 
\begin{align*}
& \qquad \nu_{0,\mu}(P_{\sigma}) \\
&= \sgn(\sigma) \prod_{0\leq \ell,m\leq p-1}
 \left( \sum_{k=0}^{p-1} \nu_{\ell,m+\mu}
 (\eta^{(k)}_{\sigma(\ell,m)-(\ell,m)-(0,\ell k)})
 \beta^k \right) \\
&= \sgn(\sigma) \prod_{0\leq \ell,m \leq p-1} \left( \sum_{k=0}^{p-1}
 \nu_{\ell,m}(\eta^{(k)}_{(\sigma(\ell,m-\mu)+(0,\mu))-(\ell,m)-(0,\ell k)}) \beta^k
 \right).
\end{align*}

Therefore if $\sigma$ does not satisfy
\begin{equation} \label{eq:sigma2}
\sigma(\ell,m) = \sigma(\ell,m-\mu)+(0,\mu) \qquad \text{for each } \mu \in \Z/p\Z,
\end{equation}
then
\begin{equation*}
\tau_{\mu} (\ell,m) = \sigma(\ell,m-\mu)+(0,\mu), \qquad \mu \in \Z/p\Z
\end{equation*}
are distinct elements of $\ideal{S}_{\square}$, and we obtain
\begin{equation*}
\sum_{\mu=0}^{p-1} P_{\tau_{\mu}} \in \sum_{\mu}^{p-1} \nu_{0,\mu} \left(
						   \Zp[U_2^{(n)}/V_2]
								   \right),
\end{equation*}
here we use the fact $\sgn(\sigma)=\sgn(\tau_{\mu})$ which can be proven by the same argument as the case of $\tilder{\theta_2}$.

Note that each $\nu_{0,\mu}$ is a $\Zp[U_2'^{(n)}/V_2]$-linear map where 
\begin{equation*}
{U_2'}^{(n)}=\langle \beta,  \varepsilon, \zeta \rangle \times 
\Gamma/\Gamma^{p^n}
\end{equation*}
and $\Zp[U_2^{(n)}/V_2]$ is generated by $\{ \gamma^c \mid c\in \Z/p\Z \}$ over $\Zp[{U_2'}^{(n)}/V_2]$. Since
\begin{equation*}
\sum_{\mu=0}^{p-1} \nu_{0,\mu}(\gamma^c)=
\begin{cases}
p & \text{if } c=0, \\
\gamma^c (1+\zeta+\cdots +\zeta^{p-1}) & \text{otherwise},
\end{cases}
\end{equation*}
we may conclude that
\begin{equation} \label{eq:sumJ_2}
\sum_{\mu=0}^{p-1} \nu_{0,\mu} \left( \Zp[U_2^{(n)}/V_2] \right)\subseteq J_2^{(n)}
\end{equation}
and this implies $\displaystyle \sum_{k=0}^{p-1} P_{\tau_k} \in
      \tilder{J_2}^{(n)}$ (note that the argument up to here is almost the same as the case of $\tilder{\theta_2}$).

Now suppose that $\sigma \in \ideal{S}_{\square}$ satisfies
      (\ref{eq:sigma2}). If we set $\sigma(\ell,0)=(a_{\ell},b_{\ell})$, then by 
(\ref{eq:sigma2}) we have
\begin{equation*}
\sigma(\ell,-\mu)=(a_{\ell},b_{\ell}-\mu) \qquad \text{for each }\mu \in \Z/p\Z.
\end{equation*}

This calculation implies that all permutations satisfying
      (\ref{eq:sigma2}) are described in the following forms:
\begin{equation}
c_{s,h}(\ell,m)=(s(\ell),h_{\ell} +m)
\end{equation}
where $s$ is a permutation of $\{ 0,1,\dotsc , p-1 \}$ and $h=(h_{\ell})_{\ell}$
      is an element of $(\Z/p\Z)^{\oplus p}$. Then we have
\begin{align*}
P_{c_{s,h}} &= \sgn(c_{s,h}) \prod_{0\leq \ell,m\leq p-1} \left(
 \sum_{k=0}^{p-1} \nu_{\ell,m} (\eta_{c_{s,h}(\ell,m)-(\ell,m)-(0,\ell k)}^{(k)})
 \beta^k \right) \\
&= \sgn(s) \prod_{0\leq \ell,m\leq p-1} \left( \sum_{k=0}^{p-1} \nu_{\ell,m}
 (\eta_{(s(\ell)-\ell, h_{\ell}-\ell k)}^{(k)}) \beta^k \right) \\
&= \sgn(s) \prod_{\ell=0}^{p-1} \nu_{\ell,0} \left( \prod_{m=0}^{p-1} \nu_{0,m} (Q_{s,h;\ell})\right),
\end{align*}
where $\displaystyle Q_{s,h;\ell}=\sum_{k=0}^{p-1}
      \eta_{(s(\ell)-\ell,h_{\ell}-\ell k)}^{(k)}\beta^k$. Note that $\sgn(c_{s,h})=\sgn(s)^p=\sgn(s)$ since $p$ is odd.

Since $Q_{s,h;\ell}$ is independent of $m$, we may prove
\begin{equation*}
\prod_{m=0}^{p-1}\nu_{0,m}(Q_{s,h;\ell}) \equiv \frob{Q_{s,h;\ell}} \qquad \mod{J_2^{(n)}}
\end{equation*}
in the same way as Claim above.

Therefore,
\begin{align}
& \qquad \sum_{s,h} P_{c_{s,h}} \nonumber \\
 &\equiv \sum_{s,h} \sgn(s) \left( \prod_{\ell=0}^{p-1}
 \nu_{\ell,0} (\frob{Q_{s,h;\ell}}) \right) &
 \mod{J_2^{(n)}}  \nonumber \\
&\equiv \sum_{s,h} \sgn(s) \left( \sum_{0\leq k_0,\dotsc ,k_{p-1}\leq p-1} \prod_{\ell=0}^{p-1} 
\frob{\eta_{(s(\ell)-\ell,h_{\ell}-\ell k_{\ell})}^{(k_{\ell})}} \right)
 &\mod{J_2^{(n)}} \nonumber \\
& = \sum_{s,h,k_{\ell}} R_{s,h;k_0,k_1,\dotsc, k_{p-1}} & \label{eq:expansion2}
\end{align}
where $\displaystyle R_{s,h;k_0,\dotsc,k_{p-1}}= \sgn(s) \prod_{\ell=0}^{p-1} 
\frob{\eta_{(s(\ell)-\ell,h_{\ell}-\ell k_{\ell})}^{(k_{\ell})}}$. We use the fact that $\nu_{\ell,0} \circ \varphi=\varphi$ (recall $\frob{\Zp[U_2^{(n)}/V_2]} \subseteq \Zp[\Gamma/\Gamma^{(n)}]$) for the second congruence.

If $s(\ell)-\ell=\lambda \, (\text{constant})$, $k_0=k_1=\dotsc =k_{p-1}$
      and $h_{\ell}-\ell k_{\ell}=\mu \, (\text{constant})$, that is, if 
$s(\ell)=s_{\lambda}(\ell)=\ell+\lambda$, $k_{\ell}=\kappa \quad \text{for each }\ell$, and 
$h=h_{\mu,\kappa}=(\ell \kappa+\mu)_{\ell}$ for $0\leq \kappa, \lambda,
      \mu\leq p-1$, we have
      $R_{s_{\lambda},h_{\mu,\kappa};\kappa,\dotsc,
      \kappa}=\frob{\eta_{\lambda,\mu}^{(\kappa)}}^p$. Note that
      $\sgn(s_{\lambda})=1$ because $s_{\lambda}$ is a cyclic
      permutation of degree $p$.  

Otherwise, set 
\begin{align*}
k^{(w)}_{\ell} &= k_{\ell+w}, \\
s^{(w)}(\ell) &= s(\ell+w)-w, \\
h^{(w)}_{\ell} &= h_{\ell+w}-wk_{\ell+w},
\end{align*}
for each $0\leq w \leq p-1$. Then we have
\begin{equation*}
R_{s^{(0)},h^{(0)};k^{(0)}_0,\dotsc, k^{(0)}_{p-1}}= \dotsc 
=R_{s^{(p-1)},h^{(p-1)};k^{(p-1)}_0,\dotsc, k^{(p-1)}_{p-1}}
\end{equation*}
and $R_{s^{(w)},h^{(w)};k^{(w)}_0,\dotsc ,k^{(w)}_{p-1}} \quad (0\leq w\leq p-1)$ are distinct terms in the expansion (\ref{eq:expansion2}). Hence we have
\begin{equation*}
\sum_{w=0}^{p-1} R_{s^{(w)},h^{(w)};k^{(w)}_0,\dotsc ,k^{(w)}_{p-1}}
 \in p\Zp[\Gamma^{(n)}] \subseteq J_2^{(n)}.
\end{equation*}

Therefore we have 
\begin{align*}
\theta_2(\eta) &\equiv \sum_{s,h,k_w} R_{s,h;k_0,\dotsc ,k_{p-1}} &
 \mod{J_2^{(n)}} \\
&\equiv \sum_{0\leq \lambda,\mu,\kappa \leq p-1}
 R_{s_{\lambda},h_{\mu,\kappa}; \kappa,\kappa,\dotsc, \kappa} & 
\mod{J_2^{(n)}} \\
&\equiv \sum_{0\leq \lambda,\mu,\kappa \leq p-1}
 \frob{\eta_{\lambda,\mu}^{(\kappa)}}^p &\mod{J_2^{(n)}} \\
&\equiv \left( \sum_{0\leq \lambda,\mu,\kappa \leq p-1}
 \frob{\eta_{\lambda,\mu}^{(\kappa)} \beta^{\kappa} \alpha^{\lambda}
 \delta^{\mu}} \right)^p &\mod{J_2^{(n)}} \\
&\equiv \varphi \left( \sum_{0\leq \lambda,\mu,\kappa \leq p-1}
 \eta_{\lambda,\mu}^{(\kappa)} \beta^{\kappa} \alpha^{\lambda}
 \delta^{\mu} \right)^p &\mod{J_2^{(n)}} \\
&\equiv \frob{\eta}^p &\mod{J_2^{(n)}}.
\end{align*}

Note that since $J_2^{(n)} \supseteq p\Zp[U_2^{(n)}/V_2]$, we have
\begin{equation*}
(x+y)^p \equiv x^p+y^p\quad \mod{J_2^{(n)}}.
\end{equation*} 

\item The congruence for $\theta_3$.

We may show the congruence for $\theta_3$ by an argument similar to that for
      $\theta_2$. Moreover, the calculation for $\theta_3$ is much
      simpler than the case for $\theta_2$ since $U_3$ is normal in $G$. Therefore we omit the proof.
\end{itemize}
\end{proof}

By Lemma \ref{lem:congj} and the fact that $\frob{\theta_0(\eta)}$ is invertible, we have
\begin{align*}
\dfrac{\theta_1(\eta)}{\frob{\theta_0(\eta)}} &\in 1+J^{(n)}_1, 
 & \dfrac{\tilder{\theta_2}(\eta)}{\frob{\theta_0(\eta)}} & \in 1+\tilder{J_2}^{(n)}, \\
\dfrac{\theta_2(\eta)}{\frob{\theta_0(\eta)}^p} & \in 1+J^{(n)}_2, 
 & \dfrac{\theta_3(\eta)}{\frob{\theta_0(\eta)}^{p^2}} & \in 1+J^{(n)}_3.
\end{align*}

Hence by Lemma \ref{lem:logj}, we have 
\begin{align*}
\log \dfrac{\theta_1(\eta)}{\frob{\theta_0(\eta)}} &\in J^{(n)}_1, 
 & \log \dfrac{\tilder{\theta_2}(\eta)}{\frob{\theta_0(\eta)}} & \in \tilder{J_2}^{(n)}, \\
\log \dfrac{\theta_2(\eta)}{\frob{\theta_0(\eta)}^p} & \in J^{(n)}_2, 
 & \log \dfrac{\theta_3(\eta)}{\frob{\theta_0(\eta)}^{p^2}} & \in J^{(n)}_3.
\end{align*}

But by Proposition \ref{prop:image}, these elements are contained in the
image of $\theta_i^+$, and therefore contained in $I^{(n)}_i$ (recall
the definition of $I_i^{(n)}$). Moreover, since the $p$-adic logarithmic
homomorphisms induce isomorphisms $1+I^{(n)}_i \xrightarrow{\simeq}
I^{(n)}_i$ (by Lemma \ref{lem:logi}), we may conclude that 
\begin{equation} \label{eq:i}
\begin{aligned}
\dfrac{\theta_1(\eta)}{\frob{\theta_0(\eta)}} &\in 1+I^{(n)}_1, 
 & \dfrac{\tilder{\theta_2}(\eta)}{\frob{\theta_0(\eta)}} & \in 1+\tilder{I_2}^{(n)} \\
\dfrac{\theta_2(\eta)}{\frob{\theta_0(\eta)}^p} & \in 1+I^{(n)}_2, 
 & \dfrac{\theta_3(\eta)}{\frob{\theta_0(\eta)}^{p^2}} & \in 1+I^{(n)}_3.
\end{aligned}
\end{equation}

Since each $\theta_i$ is a norm map in $K$-theory, it is clear that $(\theta_i(\eta))_i$ satisfies norm relations. This and (\ref{eq:i}) imply that $\theta(\eta)\in \Psi^{(n)}$.

%
\subsection{Surjectivity for finite quotients}
%

Let $(\eta_i)_i$ be an arbitrary element of $\Psi^{(n)}$. By the congruences in Definition \ref{def:psi} and by the fact that $\frob{\eta_0}$ is invertible, we have
\begin{align*}
\dfrac{\eta_1}{\frob{\eta_0}} &\in 1+I^{(n)}_1, 
 & \dfrac{\tilder{\eta_2}}{\frob{\eta_0}} & \in 1+\tilder{I_2}^{(n)}, \\
\dfrac{\eta_2}{\frob{\eta_0}^p} & \in 1+I^{(n)}_2, 
 & \dfrac{\eta_3}{\frob{\eta_0}^{p^2}} & \in 1+I^{(n)}_3.
\end{align*}

Then we obtain
\begin{align*}
\log \dfrac{\eta_1}{\frob{\eta_0}} &\in I^{(n)}_1, 
 & \log \dfrac{\tilder{\eta_2}}{\frob{\eta_0}} & \in \tilder{I_2}^{(n)}, \\
\log \dfrac{\eta_2}{\frob{\eta_0}^p} & \in I^{(n)}_2, 
 & \log \dfrac{\eta_3}{\frob{\eta_0}^{p^2}} & \in I^{(n)}_3
\end{align*}
via the $p$-adic logarithms (Lemma \ref{lem:logi}).

Note that since $\frob{\eta_0}$ is contained in the center of
$\Zp[U_i^{(n)}/V_i]^{\times}$, it is easy to see that 
\begin{align*}
\Nr_{\Zp[\tilder{U^{(n)}_2}/\tilder{V_2}]/\Zp[U^{(n)}_2/\tilder{V_2}]} \frob{\eta_0} &=
\frob{\eta_0}^p, \\
\Nr_{\Zp[U^{(n)}_1/\tilder{V_2}]/\Zp[U^{(n)}_1\cap
 \tilder{U_2}^{(n)}/\tilder{V_2}]} \frob{\eta_0}
 &=\Nr_{\Zp[\tilder{U_2}^{(n)}/\tilder{V_2}]/\Zp[U^{(n)}_1\cap
 \tilder{U_2}^{(n)}/\tilder{V_2}]} \frob{\eta_0} \\
 &= \frob{\eta_0}^p, \\
\Nr_{\Zp[U^{(n)}_1/V_1]/\Zp[U^{(n)}_3/V_1]} \frob{\eta_0} &= \frob{\eta_0}^{p^2}, \\
\Nr_{\Zp[U^{(n)}_2/V_2]/\Zp[U^{(n)}_3/V_2]} \frob{\eta_0}^p &=
 \frob{\eta_0}^{p^2}.
\end{align*}

By using Lemma \ref{lem:lognorm}, 
\begin{align*}
(y_1,\tilder{y_2}, y_2, y_3) = &\left( \log \dfrac{\eta_1}{\frob{\eta_0}}, \log \dfrac{\tilder{\eta_2}}{\frob{\eta_0}}, \log \dfrac{\eta_2}{\frob{\eta_0}^p}, \log \dfrac{\eta_3}{\frob{\eta_0}^{p^2}} \right) \\
 &\in I^{(n)}_1 \times \tilder{I_2}^{(n)} \times I^{(n)}_2 \times I^{(n)}_3
\end{align*}
satisfies the trace relations in Definition \ref{def:omega} (except for $\Zp[U^{(n)}_0/V_0]$-part). 

Now set $y_0=\Gamma_{U^{(n)}_0/V_0} (\eta_0)=(1/p)\log (\eta_0^p/\frob{\eta_0}) \in \Zp[U^{(n)}_0/V_0]$. Then we may easily check that $(y_0, y_1, \tilder{y_2}, y_2, y_3)$ satisfies the trace relations in Definition \ref{def:omega}. Hence we have $(y_0,y_1,\tilder{y_2},y_2,y_3)\in \Omega^{(n)}$.

By the additive theta isomorphism (Proposition-Definition \ref{pd:adak}), there exists a unique element $y\in \Zp[\conj{G^{(n)}}]$ satisfying 
\begin{equation} \label{eq:y}
\begin{aligned}
\theta^+(y) &=(y_0,y_1,\tilder{y_2},y_2,y_3) \\
&= \left( \Gamma_{U_0^{(n)}/V_0}(\eta_0), \log \dfrac{\eta_1}{\frob{\eta_0}}, \log \dfrac{\tilder{\eta_2}}{\frob{\eta_0}}, 
\log \dfrac{\eta_2}{\frob{\eta_0}^p}, \log \dfrac{\eta_3}{\frob{\eta_0}^{p^2}} \right).\end{aligned}
\end{equation}

\begin{prop}
$\omega_{G^{(n)}}(y)=1$, that is, $y$ is in the image of the integral
 logarithmic homomorphism $\Gamma_{G^{(n)}}$ where $\omega_{G^{(n)}}$ is
 defined as in Theorem $\ref{thm:intlogstr}$.
\end{prop}

\begin{proof}
By the definition, $\omega_{G^{(n)}}$ is decomposed as follows:
\begin{equation*}
\xymatrix{
 \Zp[\conj{G^{(n)}}] \ar[r]^{\omega_{G^{(n)}}} \ar[d]_{\pi^{(n)}} & G^{\mathrm{ab}}=U_0^{(n)}/V_0 \\
 \Zp[U_0^{(n)}/V_0] \ar[ur]_{\omega_{U_0^{(n)}/V_0}}}
\end{equation*}
where $\pi^{(n)}$ is the canonical surjection. Note that $\pi^{(n)}$ is
 the same map as
 $\theta^+_0$ by the definition of $\theta^+$. Therefore, we have 
\begin{align*}
\omega_{G^{(n)}} (y) &= \omega_{U^{(n)}_0/V_0} (\theta^+_0 (y)) \\
&= \omega_{U^{(n)}_0/V_0} (\Gamma_{U^{(n)}_0/V_0} (\eta_0)) \qquad 
\text{by} \, (\ref{eq:y}) \\
&= 1 \, \qquad \text{by Theorem \ref{thm:intlogstr} }.
\end{align*}
\end{proof}

Let $\eta$ be an element of $K_1(\Zp[G^{(n)}])$ satisfying $\Gamma_{G^{(n)}}(\eta)=y$. Then $\eta$ is determined up to multiplication by a torsion
element of
$K_1(\Zp[G^{(n)}])$ (See Theorem \ref{thm:intlogstr}). By the definition of $\eta$ and the equation
(\ref{eq:y}), 
\begin{equation*}
\theta^+ \circ \Gamma_{G^{(n)}} (\eta) =(y_0,y_1,\tilder{y_2},y_2,y_3).
\end{equation*}

Combining with Proposition \ref{prop:image}, we have 
\begin{equation*} 
y_0 =  \dfrac{1}{p} \log \dfrac{\eta_0^p}{\frob{\eta_0}} = \dfrac{1}{p}
 \log \dfrac{\theta_0(\eta)^p}{\frob{\theta_0(\eta)}},
\end{equation*}
that is,
\begin{equation} \label{eq:eta0}
\Gamma_{U^{(n)}_0/V_0} (\eta_0) =\Gamma_{U^{(n)}_0/V_0} (\theta_0(\eta)).
\end{equation}

We also have
\begin{equation} \label{eq:logeqs}
\begin{aligned}
\log \dfrac{\eta_1}{\frob{\eta_0}} &= \log \dfrac{\theta_1(\eta)}{\frob{\theta_0(\eta)}}  , & \log \dfrac{\tilder{\eta_2}}{\frob{\eta_0}} &= \log \dfrac{\tilder{\theta_2}(\eta)}{\frob{\theta_0(\eta)}}, \\
\log \dfrac{\eta_2}{\frob{\eta_0}^p} &= \log \dfrac{\theta_2(\eta)}{\frob{\theta_0(\eta)}^p}  , & \log \dfrac{\eta_3}{\frob{\eta_0}^{p^2}} &= \log \dfrac{\theta_3(\eta)}{\frob{\theta_0(\eta)}^{p^2}}.
\end{aligned}
\end{equation}

\begin{prop} \label{prop:surj}
There exists an element $\tau \in \mu_{p-1}(\Zp) \times
 (G^{(n)})^{\mathrm{ab}}$ such that
\begin{equation*}
\theta(\eta \tau)=(\eta_0, \eta_1, \tilder{\eta_2},  \eta_2, \eta_3)
\end{equation*}
where $\mu_{p-1}(\Zp)$ is the multiplicative group generated by all $(p-1)$-th roots of $1$ in $\Zp$.

In particular, $\theta$ surjects on $\Psi^{(n)}$.
\end{prop}

\begin{proof}
From the equation (\ref{eq:eta0}) and the fact that the integral logarithm 
homomorphism is injective modulo torsion elements, there exists an
 element $\tau\in K_1(\Zp[U^{(n)}_0/V_0])_{\mathrm{tors}}$ satisfying $\eta_0 = \tau \theta_0(\eta)$.

Then by Lemma \ref{prop:sk1} (3), $\tau$ is contained in 
\begin{equation*}
K_1(\Zp[U^{(n)}_0/V_0])_{\mathrm{tors}}=\mu_{p-1}(\Zp) \times (U^{(n)}_0/V_0) = \mu_{p-1}(\Zp) \times (G^{(n)})^{\mathrm{ab}}.
\end{equation*}

Hence we have
\begin{equation*}
\theta_0(\eta \tau)=\theta_0(\eta)\theta_0(\tau)=\theta_0(\eta)\tau=\eta_0.
\end{equation*}

Since $\tau \in \Ker \left(\Gamma_{G^{(n)}}\right)$, we have $\theta^+ \circ \Gamma_{G^{(n)}} (\eta \tau)=\theta^+ \circ \Gamma_{G^{(n)}}(\eta)$. Therefore we may replace $\eta$ by $\eta \tau$ in the right hand sides of the equations (\ref{eq:logeqs}).

From this fact and Lemma \ref{lem:logi}, we have $\theta_i(\eta \tau)=\eta_i$ for each $i$, using $\theta_0(\eta \tau)=\eta_0$.
\end{proof}

%
\subsection{Taking the projective limit}
%

In this subsection, we take the projective limit of the surjections 
\begin{equation*}
\theta^{(n)} \colon K_1(\Zp[G^{(n)}]) \longrightarrow \Psi^{(n)}
\end{equation*}
obtained by the argument in the previous subsections, and construct the surjection
(\ref{eq:aksurj}).

First, we study the kernel of the surjection $\theta^{(n)}$
precisely.

Suppose that $\tau \in \Ker(\theta^{(n)})$. Then we have $\theta^{+,(n)} \circ \Gamma_{G^{(n)}} (\tau) =0$ by Proposition 
\ref{prop:image}. Since 
$\theta^{+,(n)}$ is an isomorphism by Proposition-Definition \ref{pd:adak}, we obtain 
$\Gamma_{G^{(n)}} (\tau)=0$. In other words,
\begin{align*}
\Ker(\theta^{(n)}) \subseteq \Ker (\Gamma_{G^{(n)}})
 &=K_1(\Zp[G^{(n)}])_{\mathrm{tors}} \\
 &=\mu_{p-1}(\Zp) \times (G^{(n)})^{\mathrm{ab}} \times SK_1(\Zp[G^{(n)}]).
\end{align*}

By the definition of $SK_1$ groups, norm maps of $K_1$-groups induce homomorphisms on $SK_1$, so we have 
\begin{equation*}
\theta^{(n)}(SK_1(\Zp[ G^{(n)}])) \subseteq \prod_i SK_1 (\Zp[U_i^{(n)}/V_i]).
\end{equation*}

However, $\displaystyle SK_1 (\Zp[U_i^{(n)}/V_i]) = 1$ since each
$U_i^{(n)}/V_i$ is abelian (Proposition \ref{prop:sk1} (3)). Hence we
obtain  $SK_1(\Zp[G^{(n)}]) \subseteq \Ker(\theta^{(n)})$.

On the other hand, for every element $\tau \in \mu_{p-1}(\Zp) \times
(G^{(n)})^{\mathrm{ab}}$ not equal to $1$, we have $\theta^{(n)}_0(\tau)=\tau \neq 1$. This implies
$\left( \mu_{p-1}(\Zp)\times (G^{(n)})^{\mathrm{ab}}\right) \cap \Ker (\theta^{(n)}) =\{1 \}$. 

Therefore we may conclude that $\Ker(\theta^{(n)})=SK_1(\Zp[G^{(n)}])$, and we obtain an exact sequence of projective systems of abelian groups
\begin{equation} \label{cd:akn}
\begin{CD}
 1 @>>> SK_1(\Zp[G^{(n)}]) @>>> K_1(\Zp[G^{(n)}]) @>\theta^{(n)}>> 
\Psi^{(n)} @>>> 1.
\end{CD}
\end{equation}

By Proposition \ref{prop:sk1} (2), $SK_1(\Zp[G^{(n)}])$ is a finite abelian group for every $n\geq 1$. Hence the projective system $\{SK_1(\Zp[G^{(n)}]) \}_{n\in \mathbb{N}}$
satisfies the Mittag-Leffler condition, which implies $\varprojlim_n^1 SK_1(\Zp[G^{n}])=0$. By taking the projective limit of (\ref{cd:akn}),
we obtain the following exact sequence
\begin{equation}
1 \rightarrow \displaystyle \varprojlim_n SK_1(\Zp[G^{(n)}]) \rightarrow \displaystyle \varprojlim_n K_1(\Zp[G^{(n)}]) \xrightarrow{(\theta^{(n)})_n} \Psi \rightarrow 1.
\end{equation}

Therefore, the surjectivity of (\ref{eq:aksurj}) is reduced to the following 
proposition.

\begin{prop}
The homomorphism 
\begin{equation*}
K_1(\iw{G}) \rightarrow \varprojlim_n K_1(\Zp[G^{(n)}])
\end{equation*}
induced by canonical homomorphisms $K_1(\iw{G}) \rightarrow
 K_1(\Zp[G^{(n)}])$ is an isomorphism.
\end{prop}

\begin{proof}
Apply Proposition \ref{prop:projlim} for $\iw{G}$ and $\{ \Zp[G^{(n)}] \}_{n\in \mathbb{N}}$.
\end{proof}

\begin{rem} \label{rem:kato}
In the case of Kato's Heisenberg type (\cite{Kato1}), that is, in the case
 that $\gal{F^{\infty}/F}\cong \overline{G}$ where
\begin{equation*}
\line{G}=\begin{pmatrix} 1 & \Fp & \Fp \\ 0 & 1 & \Fp \\ 0 & 0 & 1 
\end{pmatrix} \times \Gamma = \line{G}^f \times \Gamma ,
\end{equation*}
it is known (\cite{Oliver}) that
\begin{equation*}
SK_1(\Zp[\line{G}^{(n)}])=SK_1(\Zp[\line{G}^f]) \oplus
 SK_1(\Zp[\Gamma/\Gamma^{p^n}])=1,
\end{equation*}
therefore we may show that the theta map
\begin{equation*}
K_1(\iw{\overline{G}}) \xrightarrow{\simeq} \overline{\Psi}
 \left(  \subseteq\prod_i \iw{U_i/V_i} \right)
\end{equation*}
is an isomorphism, and $\xi_{F^{\infty}/F}$ is determined uniquely.
\end{rem}

\begin{rem} \label{rem:venj}
By the calculation above, we know that the kernel of the theta map tends
 to be $SK_1(\iw{G})$. It is conjectured that $SK_1(\iw{G})$ vanishes
 for general compact $p$-adic Lie groups.

For our special case $G=G^f \times \Gamma$, Otmar Venjakob announced to
 the author that he and Peter Schneider
 have recently proven the vanishing of $SK_1(\iw{G})$. If we admit their result, we
 may prove the uniqueness of the $p$-adic zeta function
 $\xi_{F^{\infty}/F}$ which we would construct in \S 8.
\end{rem}

%
%
\section{Localized theta map}
%
%

In this section, we construct {\it the localized theta map}
\begin{equation*}
\theta_S : K_1(\iw{G}_S) \longrightarrow \Psi_S.
\end{equation*}

Let $\iw{\Gamma}=\Zp[[\Gamma]]$ be the Iwasawa algebra for $\Gamma$.
Then since $G=G^f \times \Gamma$, we have
\begin{align*}
\iw{G} &= \Zp[[G^f \times \Gamma]] \\
& \cong \Zp[G^f] \otimes_{\Zp} \Zp[[\Gamma]] \\
& \cong \iw{\Gamma} [G^f],
\end{align*}
and we also have $\iw{U_i/V_i} \cong \iw{\Gamma} [U_i^f/V_i^f]$

Let $S$ be the canonical \O re set (see \S 2.1) for the group $G$ 
and let $S_0$ be that for $\Gamma$. Note that $S_0 = \iw{\Gamma} \setminus 
p\iw{\Gamma}$ (see Example \ref{ex:iwasawa}). 

\begin{lem} \label{lem:locs0}
$\iw{G}_S \cong \iw{\Gamma}_{S_0} [G^f]$ and $\iw{U_i/V_i} \cong \iw{\Gamma}_{S_0} 
[U_i^f/V_i^f]$.
\end{lem}

\begin{proof}
We show this lemma only for the group $G$. 

First, we show that $\iw{G}_S=\iw{G}_{S_0}$ (by abuse of notation, 
we also denote the image of $S_0$ under the canonical map
 $\iw{\Gamma} \longrightarrow \iw{G}$ by the same symbol $S_0$). For an arbitrary $f_0\in S_0$, $\iw{G}/\iw{G}f_0 \cong \left( \iw{\Gamma}/\iw{\Gamma}f_0\right) [G^f]$ is $\Zp$-finitely generated, so $S_0\subseteq S$. Then by the universality of \O re localizations (Proposition-Definition \ref{pd:loc}), we obtain the canonical homomorphism 
\begin{equation} \label{eq:loc1}
\iw{G}_{S_0} \longrightarrow \iw{G}_S.
\end{equation}

Let $f$ be an element of $S$. By the definition of the canonical \O re
 set, $\iw{G}/\iw{G}f$ is a finitely generated left
 $\Zp$-module.\footnote{Since $H$ is finite, a left $\Zp[H]$-module
 is finitely generated if and only if it is finitely generated as a left
 $\Zp$-module.} Let $\ideal{z}$ be a kernel of the following composition of homomorphisms:
\begin{equation} \label{eq:modf}
\begin{CD}
\iw{\Gamma} @>>> \iw{G} @>>> \iw{G}/\iw{G}f,
\end{CD}
\end{equation}
where the first map is the canonical map 
\begin{equation*}
\iw{\Gamma} \longrightarrow \iw{G}\cong \iw{\Gamma} \otimes_{\Zp} \Zp[G^f] \, ; \, 
f_0 \mapsto f_0 \otimes 1.
\end{equation*}

Take an arbitrary non-zero element $g_0\in \ideal{z}$. Since $\iw{\Gamma}$ is a unique factorization 
domain, $g_0$ has a decomposition
\begin{equation*}
g_0=up^{n} \pi_1^{n_1} \cdots \pi_k^{n_k},
\end{equation*}
where $u\in \iw{\Gamma}^{\times}$ and $\pi_i$ is a prime element of $\iw{\Gamma}$.

Suppose that for every non-zero $g_0\in \ideal{z}$, the index of the
 prime $p$ ($n$ in the decomposition above) is not zero. Then we have
 $\ideal{z} \subseteq p\iw{\Gamma}$, and this induces a surjection
 $\iw{\Gamma}/\ideal{z} \rightarrow \iw{\Gamma}/p\iw{\Gamma}$. This is
 contradiction since $\iw{\Gamma}/\ideal{z} \subseteq \iw{G}/\iw{G}f$ is
 finitely generated over $\Zp$ by the definition of $S$, but
 $\iw{\Gamma}/p\iw{\Gamma} \cong \Fp[[\Gamma]]$ is not so. Hence 
there exists $g_0\in \ideal{z}$ such that $p\nmid g_0$, which implies $g_0\in S_0$.

Via the homomorphism (\ref{eq:modf}), $g_0$ is contained in $\iw{G}f$,
 hence there exists $\phi \in \iw{G}$ satisfying $g_0=\phi f$. This
 means $(g_0^{-1} \phi) \cdot f =1$ in $\iw{G}_{S_0}$, 
in other words, $f$ has its (left) inverse element $g_0^{-1}\phi$ in
 $\iw{G}_{S_0}$. Hence, by the universality of \O re localizations again, we obtain
\begin{equation} \label{eq:loc2}
\iw{G}_{S} \longrightarrow \iw{G}_{S_0}.
\end{equation}

It is easy to show that the homomorphisms (\ref{eq:loc1}) and
 (\ref{eq:loc2}) are the inverse maps of each other, therefore
 $\iw{G}_S\cong \iw{G}_{S_0}$.

On the other hand, the canonical homomorphisms 
\begin{equation*}
\iw{G}_{S_0} \longrightarrow \iw{\Gamma}_{S_0}[G^f]
\end{equation*}
obtained by the universality of \O re localizations and 
\begin{equation*}
\iw{\Gamma}_{S_0}[G^f]=\iw{\Gamma}_{S_0}\otimes_{\Zp} \Zp[G^f]  \longrightarrow \iw{G}_{S_0}
\end{equation*}
obtained by the universality of tensor products from the map 
\begin{equation*}
\iw{\Gamma}_{S_0}\times \Zp[G^f] \rightarrow \iw{G}_{S_0} \, ; \, (\phi_0, h) \mapsto \phi_0h
\end{equation*}
are inverse maps of each other. This implies $\iw{G}_S \cong \iw{\Gamma}_{S_0}[G^f]$.

The result for each $\iw{U_i/V_i}_S$ is obtained in the same way.

\end{proof}

By this lemma, we may replace the coefficient ring $\Zp$ by $\iw{\Gamma}_{S_0}$ and 
apply the same discussion in \S 4 and \S 5.

Since we need to use $p$-adic logarithms to translate the additive theta map, we take the $p$-adic completion $\comp{\iw{\Gamma}_{S_0}}$ of $\iw{\Gamma}_{S_0}$. Set $R=\comp{\iw{\Gamma}_{S_0}}$. Then the additive theta map 
\begin{equation*}
\comp{\theta}^+ \colon R[\conj{G^f}] \xrightarrow{\simeq} \comp{\Omega}
\end{equation*}
is obtained by the completely same argument as in \S 4, where $\comp{\Omega}$ is the $R$-submodule of $\displaystyle \prod_i R[U_i^f/V_i^f]$ defined by the 
same conditions as in Definition \ref{def:omega}. More precisely, $\comp{\Omega}$ is an $R$-submodule consisting of all elements $(\comp{y_i})_i$ satisfying the 
trace relations and $\comp{y_i} \in \comp{I_i}$ for each $i$ where $\comp{I_i} = I_i^f \otimes_{\Zp} R$.

We use the generalized integral $p$-adic logarithm (See \S 1.4)
\begin{equation*}
\Gamma_{R[G^f],J} \colon K_1(R[G^f], J) \longrightarrow R[\conj{G^f}],
\end{equation*}
where $J$ is the Jacobson radical of $R[G^f]$. In this case, $J$ is the
kernel of the natural surjection $R[G^f] \rightarrow R/pR$. By the same
argument as in \S 5, we have 
\begin{equation*}
\comp{\theta}^+ \circ \Gamma_{R[G^f],J} (\hat{\eta}') \in \comp{\Omega},
\end{equation*}
for an arbitrary $\hat{\eta}' \in K_1(R[G^f],J)$, and we may conclude that $\comp{\theta}(\hat{\eta}')$ is contained in $\comp{\Psi}$, where $\comp{\Psi}$ is the subgroup 
of $\displaystyle \prod_i R[U_i^f/V_i^f]$ consisting of all elements
$(\hat{\eta}'_i)_i$ which
satisfy the norm relations and the following congruences:
\begin{equation} \label{eq:conjcomp}
\begin{aligned}
\hat{\eta}'_1 & \equiv \frob{\hat{\eta}'_0} && \mod{\comp{I_1}}, & 
\hat{\tilder{\eta}}'_2 & \equiv \frob{\hat{\eta}'_0} && \mod{\comp{\tilder{I_2}}}, \\
\hat{\eta}'_2 & \equiv \frob{\hat{\eta}'_0}^p && \mod{\comp{I_2}}, &
\hat{\eta}'_3 & \equiv \frob{\hat{\eta}'_0}^{p^2} && \mod{\comp{I_3}}.
\end{aligned}
\end{equation}

Let $\hat{\eta}$ be an arbitrary element of $K_1(R[G^f])$. We use the
same notation $\hat{\eta}$ for its lift to $R[G^f]^{\times}$. Let
$F\subseteq R[G^f]$ be a subset of representatives of $\pi^{\times}
\colon R[G^f]^{\times} \rightarrow (R[G^f]/J)^{\times} \cong
\Fp((T))^{\times}$, that is, $F$ be a subset of $R[G^f]$ for which $\pi^{\times}$ induces a bijection of sets $F \xrightarrow{\simeq} (R[G^f]/J)^{\times}$. Note that we can take $F$ as a subset of $R^{\times}$.

Then there exists $\phi\in F$ and $\hat{\eta}' \in 1+J$ such that
$\tilder{\eta}=\phi \hat{\eta}'$. $\hat{\eta}'$ defines an element of
$K_1(R[G^f],J)$, and we have already known that
$\comp{\theta}(\hat{\eta}')$ is contained in $\comp{\Psi}$. Therefore, if we prove that
$\comp{\theta}(\phi)\in \comp{\Psi}$ for every $\phi \in F$, we may
conclude that $\comp{\theta}(K_1(R[G^f])) \subseteq \comp{\Psi}$. It suffices to show that $\comp{\theta}(r)\in \comp{\Psi}$ for every $r \in R^{\times}$.

For an arbitrary $r \in R^{\times}$, we have 
\begin{align*}
\comp{\theta_0}(r) &=r, & \comp{\theta_1}(r) &=\comp{\tilder{\theta_2}}(r)=r^p, \\
\comp{\theta_2}(r) &=r^{p^2}, &  \comp{\theta_3}(r) &=r^{p^3},
\end{align*}
by direct calculation.\footnote{Since $R$ is in the center of $R[G^f]$,
the images of norm maps of $r\in R$ are easily calculated by using the method in \S 1.3.} This implies that $(\comp{\theta_i}(r))_i$ satisfies the norm relations. 

The congruences for $(\theta_i (r))_i$ are derived from the following relation 
\begin{equation*}
r^p \equiv \frob{r} \qquad \mod{pR}
\end{equation*}
which follows from the definition of the Frobenius automorphism (See \S 1.4). 
Now $(\comp{\theta_i}(r))_i\in \comp{\Psi}$ holds.

\medskip
Set $\displaystyle \Psi_S = 
\comp{\Psi} \cap \prod_i \iw{\Gamma}_{S_0}[U_i^f/V_i^f]$.

\begin{prop} \label{prop:psis}
$\Psi_S$ is characterized as the subgroup of $\displaystyle \prod_i \iw{U_i/V_i}_S^{\times}$ consisting of all elements $(\eta_0, \eta_1, \tilder{\eta_2}, \eta_2, \eta_3)$ satisfying the following two conditions.
\begin{enumerate}[$(1)$]
 \item $($norm relations$)$
\begin{enumerate}[$(${\upshape rel-}$1)$]
  \item $\Nr_{\iw{U_0/V_0}_S/\iw{U_1/V_0}_S}\eta_0 \equiv \eta_1$,
  \item $\Nr_{\iw{U_0/V_0}_S/\iw{\tilder{U_2}/V_0}_S}\eta_0 \equiv \tilder{\eta_2}$,
  \item $\Nr_{\iw{\tilder{U_2}/\tilder{V_2}}_S/\iw{U_2/\tilder{V_2}}_S}\tilder{\eta_2} \equiv \eta_2$,

  \item $\Nr_{\iw{U_1/\tilder{V_2}}_S/\iw{U_1\cap \tilder{U_2}/\tilder{V_2}}_S}\eta_1 \equiv \Nr_{\iw{\tilder{U_2}/\tilder{V_2}}_S/\iw{U_1\cap \tilder{U_2}/\tilder{V_2}}_S}\tilder{\eta_2}$,
  \item $\Nr_{\iw{U_1/V_1}_S/\iw{U_3/V_1}_S}\eta_1 \equiv \eta_3$,
  \item $\Nr_{\iw{U_2/V_2}_S/\iw{U_3/V_2}_S}\eta_2 \equiv \eta_3$.
\end{enumerate}
$($These are the same as those of $\Psi$. See Figure $\ref{fg:relations})$

\item $($congruences$)$
\begin{align*}
\eta_1 & \equiv \frob{\eta_0} & \mod{I_{S,1}} \, ,  && \tilder{\eta_2} & \equiv \frob{\eta_0} & \mod{\tilder{I_{S,2}}} \, ,\\
\eta_2 & \equiv \frob{\eta_0}^p & \mod{I_{S,2}} \, , && \eta_3 & \equiv \frob{\eta_0}^{p^2} & \mod{I_{S,3}} \, ,
\end{align*}
where $I_{S,i}= I_i^f \otimes_{\Zp} \iw{\Gamma}_{S_0}$.
\end{enumerate}
\end{prop}

\begin{proof}
An arbitrary element $(\eta_i)_i$ of $\Psi_S$ satisfies the norm
 relations by the definition of $\Psi_S$. Since $(\eta_i)_i$ is
 contained in $\comp{\Psi}$, it satisfies (\ref{eq:conjcomp}). Then the
 desired congruences hold (note that $\comp{I_i}\cap \iw{\Gamma}_{S_0}[U_i^f/V_i^f]=I_i^f \otimes_{\Zp} \iw{\Gamma}_{S_0} = I_{S,i}$).
\end{proof}

Now let $\theta_S=(\theta_{S,i})_i$ be the family of the homomorphisms $\theta_{S,i}$ where each $\theta_{S,i}$ is defined as the composite of the norm map 
\begin{equation*}
\Nr_{\iw{G}_S/\iw{U_i}_S} \colon K_1(\iw{G}_S) \rightarrow K_1(\iw{U_i}_S)
\end{equation*}
and the canonical homomorphism
\begin{equation*}
K_1(\iw{U_i}_S) \rightarrow K_1(\iw{U_i/V_i}_S)=\iw{U_i/V_i}_S^{\times}
\end{equation*}
induced by $\iw{U_i}_S \rightarrow \iw{U_i/V_i}_S$.

\begin{prop} \label{prop:imageaks}
$\Psi_S$ contains the image of $\theta_S$.
\end{prop}

\begin{proof}
By the commutative diagram 
\begin{equation*}
\begin{CD}
K_1(\iw{G}_S) @>\theta_S>> \prod_i \iw{U_i/V_i}^{\times}_S \\
@VVV @VVV \\
K_1(R[G^f]) @>>\comp{\theta}> \prod_i R[U_i^f/V_i^f]^{\times},
\end{CD}\end{equation*}
we have $\mathrm{Image}(\theta_S) \subseteq
 \mathrm{Image}(\comp{\theta}) \subseteq \comp{\Psi}$ (note that the
 right vertical map is injective since the canonical map
 $\iw{\Gamma}_{S_0} \rightarrow R$ is injective, using the fact that
 $\iw{\Gamma}_{S_0}$ is separated with respect to $p$-adic topology). Therefore,
\begin{equation*}
\mathrm{Image}(\theta_S) \subseteq \comp{\Psi} \cap \left( \prod_i
						  \iw{U_i/V_i}_S^{\times}
						 \right) 
= \Psi_S.
\end{equation*}
\end{proof}

\begin{prop} \label{prop:cap}
$\displaystyle \Psi_S \cap \prod_i \iw{U_i/V_i}^{\times} = \Psi$.
\end{prop}

\begin{proof}
This follows easily from the fact $I_{S,i} \cap \iw{U_i/V_i}=I_i$.
\end{proof}

We can summarize the two propositions above as follows:

\begin{thm}
$\theta_S \colon K_1(\iw{G}_S) \longrightarrow \Psi_S$ is the localized theta map $($defined in Definition $\ref{def:theta})$ for $G$.
\end{thm}

%
%
\section{Congruences among abelian $p$-adic zeta pseudomeasures} \label{sect:cong}
%
%

%
\subsection{Norm relations}
%

Let $F_{U_i}$ (resp.\ $F_{V_i}$) be the maximal subfield of $F^{\infty}$ fixed by $U_i$ (resp.\ $V_i$). 

In the previous sections we have constructed the theta map $\theta$ 
and its localized version $\theta_S$. If the Iwasawa main conjecture is true 
in our case, the assumption of Theorem \ref{thm:ak-pri} should be true,
that is, the family of the abelian $p$-adic zeta pseudomeasures $(\xi_0,
\xi_1, \tilder{\xi_2}, \xi_2, \xi_3)$ (each $\xi_i \in
\mathrm{Frac}(\iw{U_i/V_i})$ is the $p$-adic zeta pseudomeasure for $F_{V_i}/F_{U_i}$) should be contained in $\Psi_S$. For $(\xi_i)_i$ to be an element of $\Psi_S$, it is necessary that $\xi_i$'s satisfy the norm 
relations and the congruences in Proposition \ref{prop:psis}. 

\begin{prop}
$\xi_i$'s satisfy the norm relations in Proposition $\ref{prop:psis}$.
\end{prop}

\begin{proof}
We may show this proposition by formal calculation, using only the
 interpolation properties (See Definition \ref{def:zeta} (\ref{eq:interp})). Here we only show 
\begin{align*}
\Nr_{\iw{U_0/V_0}_S/\iw{U_1/V_0}_S}\xi_0 & \equiv \xi_1 & \text{in }\iw{U_1/V_0}_S.
\end{align*}

Let $\rho$ be an arbitrary Artin representation of $U_1/V_0$. Then by Lemma \ref{lem:akind}, we have 
\begin{align*}
\left( \Nr \,  \xi_0 \right) (\rho \otimes \kappa^r) 
&= \xi_0 (\ind{U_0}{U_1}{\rho \otimes \kappa^r})  \\
&= \xi_0 (\ind{U_0}{U_1}{\rho} \otimes \kappa^r) \\
&= L_{\Sigma}(1-r; F_{V_0}/F_{U_0} , \ind{U_0}{U_1}{\rho}) \\
&= L_{\Sigma}(1-r; F_{V_0}/F_{U_1} , \rho) \\
&= \xi_1 (\rho \otimes \kappa^r) 
\end{align*}
for every positive integer $r$ divisible by $p-1$.

We may also show the other norm relations by the same argument.

\end{proof}

%
\subsection{Kato's observation}
%

Now let us study the congruences among $\xi_i$'s.

Let $I_2''$ be the image of the composition 
\[
 \Zp[[\conj{\tilder{U_2}}]] \xrightarrow{\Tr} \Zp[[\conj{U_2}]]
 \rightarrow \Zp[[U_2/V_2]].
\]

We have the following explicit description of $I_2''$:
\[
 I_2''=\varprojlim_n \left( [\beta^b \gamma^c h_{\zeta}(c\neq 0),
 p\beta^b \zeta^f]_{\Zp[\Gamma/\Gamma^{p^n}]}\right) \subseteq \iw{U_2/V_2}.
\]

\begin{prop}[Congruences among abelian $p$-adic zeta functions] \label{prop:congzeta}
The $p$-adic zeta functions $\xi_i$ for $F_{V_i}/F_{U_i}$ satisfy the following congruences$\colon$
\begin{align*}
(1) & &\xi_1 &\equiv \frob{\xi_0} & \mod{I_{S,1}}, \\
\tilder{(2)} & &\tilder{\xi_2} &\equiv \frob{\xi_0} & \mod{\tilder{I_{S,2}}}, \\
(2) &&\xi_2 &\equiv c_2 & \mod{I_{S,2}''}, \\
(3) &&\xi_3 & \equiv c_3 & \mod{I_{S,3}}, 
\end{align*}
where $c_2$ and $c_3$ are certain elements of $\iw{\Gamma}_{S_0}$.
\end{prop}

\begin{rem}
These congruences are \textbf{not} sufficient to show that $(\xi_i)_i \in \Psi_S$. Hence we should modify Burns' technique (Proposition \ref{thm:ak-pri}) 
to prove the Iwasawa main conjecture in our case. This modification is discussed in the next section.
\end{rem}

Kato observed in \cite{Kato1} that these kinds of congruences among abelian $p$-adic zeta
pseudomeasures were derived from Deligne-Ribet's method of $p$-adic
Hilbert modular forms. The fundamental philosophy of $p$-adic ($\Lambda$-adic)
Hilbert modular forms
is as follows (quite rough):

\begin{quotation}
$(\sharp)$ Congruence between constant terms of two $p$-adic
 ($\Lambda$-adic) Hilbert modular forms is derived from congruences
 between coefficients of all positive degree terms of them.
\end{quotation}

Before the precise proof of Proposition \ref{prop:congzeta}, let us review how to derive the congruence between
abelian $p$-adic zeta pseudomeasures from Hilbert modular forms,
following \cite{Kato1} \S 4. We take
the congruence $\xi_1 \equiv \frob{\xi_0} \mod I_{S,1}$ as an
example. 

The $\iw{U_i/V_i}$-adic $F_{U_i}$-Hilbert Eisenstein series
\begin{equation*}
f_i = \frac{\xi_i}{2^{[F_{U_i}: \Q]}} + \sum_{(\ideal{a}, x)\in P_i} \left( 
\frac{F_{V_i}/F_{U_i}}{\ideal{a}} \right) q^x \qquad i=0,1
\end{equation*}
was essentially constructed by Deligne and Ribet, where 
\begin{align*}
P_i= \{ (\ideal{a}, x)  \mid \ideal{a} \subseteq \integer{F_{U_i}} \colon
 \text{non-zero } & \text{ideal prime to } \Sigma, \\
& x\in \ideal{a} \colon \text{totally positive} \}
\end{align*}
and $\left( \dfrac{L/K}{-}\right)$ denotes the Artin symbol in
$\gal{L/K}$ for an abelian extension $L/K$.

By restricting $f_1$ to the Hilbert modular variety of $F$ embedded as the diagonal in the Hilbert modular variety of $F_{U_1}$, we obtain a $\iw{U_1/V_1}$-adic $F$-Hilbert modular form 
\begin{equation} \label{eq:eisen}
g_1=\frac{\xi_1}{2^{rp}} +\sum_{(\ideal{a}, x)\in P_1}  \left( \frac{F_{V_1}/F_{U_1}}{\ideal{a}} \right) q^{\Tr_{F_{U_1}/F}(x)}
\end{equation}
where $r$ is $[F:\Q]$. Note that $[F_{U_1} : \Q ]$ is equal to $rp$. The Galois group 
\begin{equation*}
\gal{F_{U_1}/F}=\langle \beta \rangle
\end{equation*}
acts on $P_1$ from the left as
\begin{equation*}
\beta^j * (\ideal{a}, x) = (\beta^j \ideal{a}, \beta^j x)
\end{equation*}
for each $0\leq j\leq p-1$.

Then we may divide positive degree terms of $g_1$ by 
isotropic subgroups $\gal{F_{U_1}/F}_{(\ideal{a},x)}$ as follows:
\begin{equation*}
\begin{split}
g_1=\frac{\xi_1}{2^{rp}} +\sum_{(\ideal{a}, x)\in P_1'} & \left( \frac{F_{V_1}/F_{U_1}}{\ideal{a}} \right)  q^{\Tr_{F_{U_1}/F}(x)} \\
 &+\sum_{(\ideal{a}, x)\in P_1''} \left( \frac{F_{V_1}/F_{U_1}}{\ideal{a}} \right) q^{\Tr_{F_{U_1}/F}(x)}
\end{split}
\end{equation*}
where
\begin{align*}
P_1' &= \{ (\ideal{a}, x) \in P_1 \mid \gal{F_{U_1}/F}_{(\ideal{a},x)}=\gal{F_{U_1}/F}  \}, \\
P_1'' &= \{ (\ideal{a}, x) \in P_1 \mid \gal{F_{U_1}/F}_{(\ideal{a},x)}=\{ \id_{F_{U_1}} \}  \}.
\end{align*}

One observes that there exits a bijection $P_0 \xrightarrow{\sim} P_1'; (\ideal{b},y) \mapsto (\ideal{a}, x)$ where $\ideal{a}=\ideal{b}\integer{F_{U_1}}$ and $x=y$. 
By the functoriality of the Artin symbol, we have 
\begin{equation*}
\begin{split}
g_1=\frac{\xi_1}{2^{rp}} +\sum_{(\ideal{b}, y)\in P_0} & \Ver \left( \left( \frac{F_{V_0}/F_{U_0}}{\ideal{b}} \right) \right) q^{py} \\
 &+\sum_{(\ideal{a}, x)\in P_1''} \left( \frac{F_{V_1}/F_{U_1}}{\ideal{a}} \right) q^{\Tr_{F_{U_1}/F}(x)},
\end{split}
\end{equation*}
where $\Ver \colon \gal{F_{V_0}/F_{U_0}} \rightarrow \gal{F_{V_1}/F_{U_1}}$ is the Verlagerung (transfer) homomorphism.

\begin{lem} \label{lem:verfrob}
Let $H_1^f$ and $H_2^f$ be arbitrary subgroups of $G^f$. Set $H_i=H_i^f
 \times \Gamma$. 

If $H_1$ contains $H_2$, the Verlagerung homomorphism 
\[
 \Ver^{H_1}_{H_2} \colon H_1^{\mathrm{ab}} \rightarrow H_2^{\mathrm{ab}}
\]
coincides with the composition of the Frobenius homomorphism $\varphi
 \colon H_1^{\mathrm{ab}} \rightarrow \Gamma$ and the canonical
 injection $\Gamma \rightarrow H_2^{\mathrm{ab}}$.
\end{lem}

\begin{proof}
By the transitivity of the Verlagerung homomorphism, it suffices to
 prove the proposition when $(H_1 : H_2)=p$. In this case it is well known that
 $H_2^f$ is normal in $H_1^f$. Let $\sigma$ be an element of $H_1^f$
 which generates $H_1^f/H_2^f$. Then $\{ \sigma^j \}_{j=0}^{p-1}$ is a
 set of representatives of $H_1/H_2$. Suppose that $h\in H_1^f$
 satisfies $h\equiv \sigma^{j(h)} \, \mod H_2^f$. By the definition of
 the Verlagerung homomorphisms, 
\[
 \Ver^{H_1}_{H_2} (h)=\prod_{j=0}^{p-1}
 \sigma^{-(j+j(h))}h\sigma^j=h^p=1.
\]
\end{proof}
 
Hence in our special case (the case when the exponent of $G^f$ is $p$), we may
identify every Verlagerung homomorphism with the Frobenius homomorphism
$\varphi$. 

If we set $\frob{q^y}=q^{py}$, we have 
\begin{equation*}
\begin{split}
g_1=\frac{\xi_1}{2^{rp}} +\sum_{(\ideal{b}, y)\in P_0} & \varphi \left( \left( \frac{F_{V_0}/F}{\ideal{b}} \right) q^{y} \right) \\
&+\sum_{(\ideal{a}, x)\in P_1''} \left( \frac{F_{V_1}/F_{U_1}}{\ideal{a}} \right) q^{\Tr_{F_{U_1}/F}(x)}
\end{split}
\end{equation*}
and the first summation is the same as positive degree terms of $\frob{f_0}$.

Next, let $\tilde{P_1}''$ be a system of representatives of the orbital decomposition 
$\gal{F_{U_1}/F_{U_0}} \backslash P_1''$. By the property of the Artin symbol, we have
\begin{equation*}
\sum_{j=0}^{p-1} \left( \frac{F_{V_1}/F_{U_1}}{\beta^j \ideal{c}} \right) = \sum_{j=0}^{p-1} \beta^j \left( \frac{F_{V_1}/F_{U_1}}{\ideal{c}} \right) \beta^{-j},
\end{equation*}
therefore the right hand side of the equation above is contained in
$I_1$ since $I_1$ is the image of $\theta_1^{+}$ (trace map).

Hence, we obtain
\begin{equation*}
\begin{split}
g_1-\frob{f_0}&=\left( \frac{\xi_1}{2^{rp}}-\frac{\frob{\xi_0}}{2^r} \right) \\  &+\sum_{(\ideal{c},z)\in \tilde{P_1}''} \left\{ \sum_{j=0}^{p-1} \beta^j \left( \frac{F_{V_1}/F_{U_1}}{\ideal{c}} \right) \beta^{-j} \right\}  q^{\Tr_{F_{U_1}/F}(z)} 
\end{split}
\end{equation*}
and all positive degree terms of $g_1-\varphi(f_0)$ are contained in $I_1$, which implies
\begin{equation*}
\frac{\xi_1}{2^{rp}} \equiv \frac{\frob{\xi_0}}{2^r} \qquad \mod{I_{S,1}}
\end{equation*}
by the philosophy $(\sharp)$.

Note that $2^p \equiv 2 \quad \mod{I_{S,1}}$ since $p\in I_{S,1}$. Hence we have
the desired congruence.

\bigskip
In the rest of this section, we make this argument precise and prove the
congruences in Proposition \ref{prop:congzeta}. We adopt the method of 
Ritter and Weiss (\cite{R-W3}).

%
\subsection{Approximation of abelian $p$-adic zeta pseudomeasures}
%

First, following \cite{R-W3}, we approximate the abelian $p$-adic
 zeta pseudomeasures and reduce the congruences among pseudomeasures to
 those among special values of partial zeta functions.

Let $W_i$ be $U_i/V_i$, and let $W_i^f$ be $U_i^f/V_i$. For an arbitrary open set $\cal{U}\subseteq W_i$, we define a non-negative
integer $m(\cal{U})$ by $\kappa^{p-1}(\cal{U})=1+p^{m(\cal{U})}\Zp$ where $\kappa$ is the
$p$-adic cyclotomic character. Let
$\omega=m(\Gamma)$. Then one can easily see that $m(\cal{U}^f \times
\Gamma^{p^j})=\omega+j$ for an arbitrary $\cal{U}^f \subseteq
W_i^f$ and $j \in \Z_{\geq 0}$.

\begin{lem} \label{lem:lim}
The canonical surjection 
\[
 \Zp[[W_i]] \rightarrow \Zp[W_i/\cal{U}]/p^{m(\cal{U})}\Zp[W_i/\cal{U}]
\]
induces the isomorphism 
\[
 \Zp[[W_i]] \xrightarrow{\sim} \varprojlim_{\cal{U} \subseteq W_i \colon \text{open}}\Zp[W_i/\cal{U}]/p^{m(\cal{U})}\Zp[W_i/\cal{U}].
\]
\end{lem}

\begin{proof}
The injectivity is obvious since $m(\cal{U})$ is not bounded. The
 surjectivity is also not difficult, for we can construct the lift of
 $(x_{\cal{U}})_{\cal{U}}$ to $\Zp[[W_i]]$. See
 \cite{Kakde} \S 4.2.1.
\end{proof}

It is clear that $\{ \Gamma^{p_j} \}_{j \in \Z_{\geq 0}}$ is the cofinal
system of the inverse limit above.

\medskip
Let $\varepsilon=\varepsilon_i$ is a $\C$-valued locally constant
function on $W_i$. Then there exists an open subset $\cal{U} \subseteq
W_i$ such that $\varepsilon$ is constant on each coset of
$W_i/\cal{U}$. Therefore we can write 
\[
 \varepsilon=\sum_{x\in W_i/\cal{U}} \varepsilon(x) \delta^{(x)}
\]
where $\delta^{(x)}$ is the characteristic function with respect to a
coset $x$, that is, 
\[
 \delta^{(x)}(w) = \begin{cases} 1 & \text{if $w\in x$}, \\
		    0 & \text{otherwise}.
\end{cases}
\]

\begin{defn}[partial zeta function]
Let $x$ be an arbitrary coset in $W_i/\cal{U}$. Then we call
\[
 \zeta_{F_{V_i}/F_{U_i}}(s, \delta^{(x)}) =\sum_{0\neq \mathfrak{a}
 \subseteq \cal{O}_{F_{U_i}}} \dfrac{\delta^{(x)} \left(
 \left( \frac{F_{V_i}/F_{U_i}}{\mathfrak{a}}\right)\right)}{(N \mathfrak{a})^s}
\]
{\em the partial zeta function for $\delta^{(x)}$}. Here $\left(
 \dfrac{F_{V_i}/F_{U_i}}{-} \right)$ is the Artin symbol and $N
 \ideal{a}$ is the absolute norm of an ideal $\ideal{a}$. This function has analytic
 continuation to the whole complex plane except for a simple pole at
 $s=1$, and for every $k\in \mathbb{N}$, $\zeta_{F_{V_i}/F_{U_i}}(1-k, \delta^{(x)})$ is a rational number.

For an arbitrary local constant function $\varepsilon$ on $W_i$, we define the partial zeta
 function for $\varepsilon$ as 
\[
 \zeta_{F_{V_i}/F_{U_i}}(s,\varepsilon)=\sum_{x\in W_i/\cal{U}} \varepsilon(x) \zeta_{F_{V_i}/F_{U_i}}(s,\delta^{(x)})
\]
where $\varepsilon=\sum_{x\in W_i/\cal{U}}\varepsilon(x)\delta^{(x)}$ is the
 decomposition as above.
\end{defn}

For an arbitrary element $w\in W_i$, $\Qp$-valued locally constant function
$\varepsilon$ on $W_i$, and positive integer $k$ divisible by $p-1$, we define 
\[
 \Delta_i^w(1-k, \varepsilon)=\zeta_{F_{V_i}/F_{U_i}}(1-k, \varepsilon)-
 \kappa(w)^k\zeta_{F_{V_i}/F_{U_i}}(1-k, \varepsilon_w) \in \Qp
\]
where $\varepsilon_w(w')=\varepsilon(ww')$ for every $w \in
W_i$. Deligne and Ribet showed the integrality of $\Delta_i^w(1-k,
\delta^{(x)})$, that is, for an arbitrary $w \in W_i$ and 
an arbitrary coset $x\in W_i/\cal{U}$,
\[
 \Delta_i^w(1-k, \delta^{(x)}) \in \Zp
\]
(\cite{De-Ri} Theoreme (0.4), see also Hypothesis $(H_{n-1})$ in \cite{Coates}). 

\begin{prop}[Approximation lemma, Ritter-Weiss] \label{prop:approx}
Let $\cal{U} \subseteq W_i$ be an arbitrary open set. Then for every
 positive integer $k$ divisible by $p-1$ and $w\in W_i$, $(1-w)\xi_i$
 maps to 
\begin{equation} \label{eq:delta}
 \sum_{x\in W_i/\cal{U}} \Delta_i^w(1-k, \delta^{(x)}) \kappa(x)^{-k} x \qquad \mod{p^{m(\cal{U})}}
\end{equation}
under the canonical surjection $\Zp[[W_i]] \rightarrow
 \Zp[W_i/\cal{U}]/p^{m(\cal{U})}\Zp[W_i/\cal{U}]$. $($Note that $\kappa(x)^{-k}$ is well-defined by the definition of $m(\cal{U}).)$
\end{prop}

\begin{proof}[Sketch of the proof]
Let $\xi_i^w$ be the limit element of (\ref{eq:delta}). This limit element is independent of $k$ due to
 Deligne-Ribet's congruence (\cite{De-Ri} Theoreme (0.4) and
 \cite{Coates} Hypothesis $(C_0)$). One sees that $\xi_i^w$ is contained
 in $\Zp[[W_i]]$ (that is, $\xi_i^w$
 is a {\em $p$-adic measure on $W_i$}) by the integrality of
 $\Delta_i^w(1-k, \delta^{(x)})$ and Lemma \ref{lem:lim}.

Then one sees that 
\[
 \int_{W_i} \varepsilon \kappa^k d\xi_i^w=\Delta_i^w (1-k, \varepsilon)
\]
for each locally constant function $\varepsilon$ on $W_i$ and positive integer
 $k$ divisible by $p-1$, but this property characterizes the $p$-adic
 measure $(1-w)\xi_i$ where $\xi_i$ is Serre's $p$-adic zeta pseudomeasure
 for $F_{V_i}/F_{U_i}$ (\cite{Serre2}). Thus $\xi_i^w=(1-w)\xi_i$.

See \cite{R-W3} Proposition 2 for the more precise proof.
\end{proof}

Let $NU_i$ be the normalizer of $U_i$ in $G$. Then we have
$NU_1=N\tilder{U_2}=NU_3=G$ and $NU_2=\tilder{U_2}$. The quotient group
$NU_i/U_i$ acts on the set of locally constant functions on $W_i$ by 
\[
 \varepsilon^{\sigma}(w)=\varepsilon(\sigma^{-1} w\sigma)
\]
where $\varepsilon$ is a locally constant functions on $W_i$ and
$\sigma\in NU_i/U_i$.

\begin{prop}[Sufficient conditions] \label{prop:suff}
Sufficient conditions for the congruences $(1)$, $\tilder{(2)}$, $(2)$,
 $(3)$ in Proposition $\ref{prop:congzeta}$ to hold are the following $(1)'$, $\tilder{(2)}'$, $(2)'$, $(3)'$ respectively$:$

\begin{enumerate}[$(1)'$]
\item $\Delta_1^{\frob{g}} (1-k, \varepsilon) \equiv \Delta_0^g (1-pk,
      \varepsilon \circ \varphi) \quad \mod p\Zp$

\noindent for all $g\in W_0$ and locally constant $\Z_{(p)}$-valued function $\varepsilon$
      on $W_1$ fixed by $NU_1/U_1(=G/U_1)$.
\end{enumerate}
\begin{enumerate}[$\tilder{(2)}'$]
\item $\tilder{\Delta}_2^{\frob{g}} (1-k, \varepsilon) \equiv \Delta_0^g (1-pk,
      \varepsilon \circ \varphi) \quad \mod p\Zp$

\noindent for all $g\in W_0$ and locally constant $\Z_{(p)}$-valued function $\varepsilon$
      on $\tilder{W_2}$ fixed by $N\tilder{U_2}/\tilder{U_2}(=G/\tilder{U_2})$.
\end{enumerate}
\begin{enumerate}[$(1)'$]
\stepcounter{enumi}
\item $\Delta_2^{w} (1-k, \delta^{(y)}) \equiv 0 \qquad \mod p\Zp$

\noindent for all $w\in \Gamma$ and for every coset $y\in
      W_2/\Gamma^{p^j}\, (j\in \N)$  
which is not contained in $\Gamma$ and fixed by
      $NU_2\maketitle/U_2(=\tilder{U_2}/U_2)$.

\item $\Delta_3^{w} (1-k, \delta^{(y)}) \equiv 0 \qquad \mod p^{m_y}\Zp$

\noindent for all $w\in \Gamma$ and for every coset $y\in
      W_3/\Gamma^{p^j}\, (j\in \N)$ 
which is not contained in $\Gamma$. Here $p^{m_y}$ is the order of
      $(NU_3/U_3)_y=(G/U_3)_y$, the isotropic subgroup of $G/U_3$ at $y$.
\end{enumerate}
\end{prop}

\begin{proof}
The proof of the sufficiency of $\tilder{(2)}'$ (resp.\ $(2)'$) for
 $\tilder{(2)}$ (resp.\ $(2)$) is completely the same as that of $(1)'$
 (resp. $(3)'$) for $(1)$ (resp.\ $(3)$). Therefore we only prove the
 latter two cases.

\begin{itemize}
\item The sufficiency of $(1)'$ for $(1)$.

We use similar method to \cite{R-W3}, Proposition 4.

The congruence $(1)$ is equivalent to 
\[
 (1-\frob{g}) \xi_1 \equiv \frob{(1-g)\xi_0} \qquad \mod I_1
\]
for an arbitrary $g\in W_0$. Applying the approximation lemma (Proposition
      \ref{prop:approx}) to them, we obtain
\begin{align}
(1-\frob{g})\xi_1 &\equiv \sum_{y\in W_1/\Gamma^{p^{j+1}}} \Delta_1^{\frob{g}}
 (1-k, \delta^{(y)}) \kappa(y)^{-k}y \label{eq:cong1}\\
& \qquad \qquad \qquad \quad \text{in
 }\Zp[W_1/\Gamma^{p^{j+1}}]/p^{\omega+j+1} \nonumber \\
\frob{(1-g)\xi_0} &\equiv \sum_{x\in W_0/(W_0^f \times\Gamma^{p^j})} \Delta_0^{g}
 (1-pk, \delta^{(x)}) \kappa(x)^{-pk}\frob{x} \label{eq:cong2}\\
& \qquad \qquad \qquad  \qquad \text{in
 }\Zp[W_0/(W_0^f \times \Gamma^{p^j})]/p^{\omega+j} \nonumber 
 \end{align}

Note that $\frob{(1-g)\xi_0}$ is fixed by the conjugate action of
      $G/U_1$ 
since it is contained in $\iw{\Gamma}$. $\frob{g}$ is also fixed by
      $G/U_1$, which implies that $(1-\frob{g})\xi_1$ is fixed by $G/U_1$ (see
      \cite{R-W3} Lemma 3.2). Therefore we see that
      $(1-\frob{g})\xi_1-\frob{(1-g)\xi_0}$ is contained in
      $(\Zp[W_1/\Gamma^{p^{j+1}}]/p^{\omega+j+1})^{G/U_1}$ (the
      $G/U_1$-fixed part).

\begin{enumerate}[({Case-}1)]
\item $y$ is fixed by $G/U_1$. Then $\delta^{(y)}$ is also fixed by
      $G/U_1$, therefore we may apply $(1)'$ for $\delta^{(y)}$:
\[
 \Delta_1^{\frob{g}} (1-k,\delta^{(y)}) \equiv \Delta_0^{g} (1-pk,
      \delta^{(y)}\circ \varphi) \qquad \mod \,p\Zp.
\]

If $y=\frob{x}(=\Ver(x))$, $x$ is determined uniquely and
      $\kappa(y)^{-k}=\kappa(\Ver(x))^{-k}=\kappa(x)^{-pk}$. Moreover,
      $\delta^{(x)}$ coincides with $\delta^{(y)}\circ \varphi$. This implies
      that 
\begin{align*}
 &\Delta_1^{\frob{g}} (1-k, \delta^{(y)}) \kappa(y)^{-k}y-\Delta_0^g
      (1-pk, \delta^{(x)}) \kappa(x)^{-pk} \frob{x} \\ 
& \qquad \in 
      p(\Z_p[W_1/\Gamma^{p^{j+1}}]/p^{\omega+j+1})^{G/U_1} \subseteq I_1^{(j+1)}/p^{\omega+j+1}.
\end{align*}

If $y\notin \mathrm{Image}(\varphi)$, we have $\delta^{(y)}\circ \varphi=0$, hence the
      $y$-summand vanishes modulo $I_1^{(j+1)}$.

\item $y$ is not fixed by $G/U_1$. Since     
\[
 \Delta_1^{\frob{g}}(1-k,\delta^{(y)})=\Delta_1^{\frob{g}}(1-k,
      \delta^{(y^{\sigma})})
\]
for every $\sigma \in G/U_1$ (see \cite{R-W3} Lemma 3.2), we have 
\begin{align*}
 & \sum_{\sigma\in G/U_1} (\Delta_1^{\frob{g}} (1-k,\delta^{(y)})
      \kappa(y)^{-k}y)^{\sigma}  \\
 &= \Delta_1^{\frob{g}}
 (1-k,\delta^{(y)}) \kappa(y)^{-k} \sum_{\sigma \in G/U_1} y^{\sigma} \in I_1^{(j+1)}/p^{\omega+j+1}.
\end{align*}
\end{enumerate}
Therefore (\ref{eq:cong1}) and (\ref{eq:cong2}) are congruent modulo
      $I_1^{(j+1)}/p^{\omega+j+1}$. Taking the projective limit, we
      obtain the congruence $(1)$ in Proposition \ref{prop:congzeta} (Lemma \ref{lem:lim}).

\item The sufficiency of $(3)'$ for $(3)$.

Apply the approximation lemma (Proposition \ref{prop:approx}) to
     $(1-w)\xi_3$. Then we have 

\begin{align}
(1-w)\xi_3 &\equiv \sum_{y\in W_3/\Gamma^{p^j}} \Delta_3^{w}
 (1-k, \delta^{(y)}) \kappa(y)^{-k}y \label{eq:cong3}\\
& \qquad \qquad \qquad \quad \text{in
 }\Zp[W_3/\Gamma^{p^j}]/p^{\omega+j}. \nonumber 
 \end{align}

Then by $(3)'$, 
\begin{align*}
 & \qquad \sum_{\sigma \in (G/U_3)/(G/U_3)_y} \Delta_3^w
 (1-k,\delta^{(y^{\sigma})}) \kappa(y^{\sigma})^{-k} y^{\sigma}  \\
 &= \Delta_3^w (1-k,\delta^{(y)}) \kappa(y)^{-k} \sum_{\sigma \in
 (G/U_3)/(G/U_3)_y} y^{\sigma} \\
 &\equiv p^{m_y} \sum_{\sigma \in (G/U_3)/(G/U_3)_y} \sigma^{-1} y
 \sigma \quad (\text{by } (3)') \\
 & \in I_3^{(j)}/p^{\omega+j}
\end{align*}
for every $y\in W_3/\Gamma^{p^j}$ which is not contained in
      $\Gamma$. Therefore it is obvious that the right hand side of the
      congruence 
      (\ref{eq:cong3}) is contained in $(\Zp[\Gamma/\Gamma^{p^j}]+I_3^{(j)})/p^{\omega+j}$.

By taking the projective limit (Lemma \ref{lem:lim}), we have
\[				       
  (1-w)\xi_3 \equiv c_w \qquad \mod I_3
\]
for a certain element $c_w\in \iw{\Gamma}$. Since $1-w\in S_0$, we can
      obtain the congruence $(3)$ in Proposition \ref{prop:congzeta}.

\end{itemize}
\end{proof}

%
\subsection{Hilbert modular forms and Hilbert-Eisenstein series}
%

In this subsection, we review Deligne-Ribet's theory of Hilbert modular
forms. See \cite{De-Ri} and \cite{R-W3} Section 3 for more information.

Let $K$ be a totally real number field of degree $r$, $K^{\infty}/K$ a
totally real $p$-adic Lie extension, and $\Sigma$ a
fixed finite set of primes of $K$ containing all primes which ramify in $K^{\infty}$. Let $\ideal{f}$ be an
integral ideal of $\mathcal{O}_K$ with all prime factors in $\Sigma$.

We denote $\ideal{h}=\{ \tau\in K \otimes \C \mid \mathrm{Im}(\tau) \gg0
\}$ the Hilbert upper-half plane.

For an even positive integer $k$, we define the action of 
\[
 GL(2,K)^+=\{ g\in GL(2,K) \mid \det(g)\gg 0 \}
\]
on functions $F\colon \ideal{h} \rightarrow \C$ by 
\[
 (F|_k\left(\begin{array}{cc} a & b \\ c & d\end{array}\right) ) (\tau)=
  \mathcal{N}(ad-bc)^{k/2}\mathcal{N}(c\tau+d)^{-k}
   F(\frac{a\tau+b}{c\tau +d})
\] 
where $\mathcal{N} \colon K\otimes \C \rightarrow \C$ is the norm map.

\begin{defn}[Hilbert modular forms]

Let
\[
 \Gamma_{00}(\ideal{f})=\{ \left( \begin{array}{cc} a & b \\ c &
			    d\end{array}
\right) \in SL(2,K)\mid a,d\in 1+\ideal{f},
 b\in \ideal{D}^{-1}, c\in \ideal{fD} \}
\]
where $\ideal{D}$ is the differential of $K$. {\em A Hilbert modular
form $F$ of weight $k$ on $\Gamma_{00}(\ideal{f})$} is a holomorphic
function\footnote{If $K=\Q$, we also
assume that $F$ is holomorphic at $\infty$.} $F\colon \ideal{h}\rightarrow \C$ satisfying $F|_kM=F$ for
all $M\in \Gamma_{00}(\ideal{f})$.  

We denote the space of Hilbert modular forms of weight $k$ on
$\Gamma_{00}(\ideal{f})$ by $M_k(\Gamma_{00}(\ideal{f}),\C)$.
\end{defn}

A Hilbert modular form $F$ has a Fourier series expansion (standard $q$-expansion)
\[
 c(0)+\sum_{\mu \in \mathcal{O}_K, \\ \mu \gg 0} c(\mu)q_K^{\mu}
\]
where $q_K^{\mu}=\exp(2\pi \sqrt{-1} \Tr_{K/\Q}(\mu\tau))$.

Deligne and Ribet constructed the Hilbert-Eisenstein series attached to
every locally constant function $\varepsilon$ on $G=\gal{K^{\infty}/K}$
(\cite{De-Ri} Theorem (6.1)).

\begin{thmdef}[Hilbert-Eisenstein series]  \label{thmdef:eisen}
Let $\varepsilon$ be a locally constant function on
 $G=\gal{K^{\infty}/K}$ and $k$ a positive even integer. Then there
 exists an integral ideal $\ideal{f}\subseteq \mathcal{O}_K$ with its
 prime factors in $\Sigma$ and a Hilbert modular form $G_{k,\varepsilon}
 \in M_k(\Gamma_{00}(\ideal{f}),\C)$ whose standard $q$-expansion is 
\[
 2^{-r} \zeta_{K^{\infty}/K}(1-k,\varepsilon) +\sum_{\mu \in
 \mathcal{O}_K, \mu \gg 0} \left( \sum_{\mu \in \ideal{a} \subseteq
 \mathcal{O}_K, \text{prime to }\Sigma} \varepsilon(\ideal{a}) \kappa(\ideal{a})^{k-1} \right)
\]

Here we denote $\varepsilon(\ideal{a})=\varepsilon \left( \left(
 \frac{K^{\infty}/K}{\ideal{a}}\right)\right)$ and $\kappa(\ideal{a})=
 \kappa \left( \left( \frac{K^{\infty}/K}{\ideal{a}}\right)\right)$ where $\left( \frac{K^{\infty}/K}{-}\right)$ is the Artin symbol.

We call $G_{k,\varepsilon}$ {\em the Hilbert-Eisenstein series of weight
 $k$ attached
 to $\varepsilon$}.
\end{thmdef}

\begin{proof} 
See \cite{De-Ri} Theorem (6.1).
\end{proof}

Next let us discuss the $q$-expansion of Hilbert modular forms at cusps.
Let $\mathbb{A}_K^f$ be the finite ad\`ele ring of $K$. By the strong
approximation theorem
\[
 SL(2,\mathbb{A}_K^f)=\hat{\Gamma}_{00}(\ideal{f}) \cdot SL(2,K),
\]
every $M\in SL(2,\mathbb{A}_K^f)$ is decomposed as $M=M_1M_2$ where
$M_1 \in \hat{\Gamma}_{00}(\ideal{f})$ (the closure of
$\Gamma_{00}(\ideal{f})$ in $SL(2, \mathbb{A}_K^f)$) and $M_2\in
SL(2,K)$. We define $F|_kM$ to be $F|_kM_2$.

For every finite id\`ele $\alpha\in (\mathbb{A}_K^f)^{\times}$, set 
\[
 F_{\alpha}=F|_k\left( \begin{array}{cc} \alpha & 0 \\ 0 & \alpha^{-1} \end{array}\right).
\] 

$F_{\alpha}$ also has a Fourier series expansion
\[
 F_{\alpha}=c(0,\alpha)+\sum_{\mu\in \mathcal{O}_K,\mu \gg 0} c(\mu, \alpha)q_K^{\mu},
\]
and we call it {\em the $q$-expansion of $F$ at the cusp determined by
$\alpha$}.

\begin{prop} \label{prop:eisencusp}
Let $k$ be a positive even integer and $\varepsilon$ a locally constant
 function on $G$. Then the $q$-expansion of $G_{k,\varepsilon}$ at the
 cusp determined by $\alpha\in (\mathbb{A}_K^f)^{\times}$ is given by
\begin{align*}
 \mathcal{N}((\alpha))^k &\Biggl\{ 2^{-r} \zeta_{K^{\infty}/K}(1-k,
 \varepsilon_{a}) \Biggr. \\ 
& \Biggl. +\sum_{\mu\in \mathcal{O}_K, \mu \gg 0} \left(
 \sum_{\mu \in \ideal{a}\subseteq \mathcal{O}_K, \text{prime to }\Sigma}
 \varepsilon_a (\ideal{a})
 \kappa(\ideal{a})^{k-1} \right) \Biggr\}
\end{align*}
where $(\alpha)$ is the ideal generated by $\alpha$ and $a=\left(
 \frac{K^{\infty}/K}{(\alpha)\alpha^{-1}}\right)$. 
\end{prop}

\begin{proof}
See \cite{De-Ri} Theorem (6.1).
\end{proof}

Now we introduce the Deligne-Ribet's deep principle.

\begin{thm}[Deligne-Ribet] \label{thm:DRpri}
Let $F_k \quad (k\geq 0)$ be rational Hilbert modular forms of weight $k$ on
 $\Gamma_{00}(\ideal{f})$ $($that is, all coefficients of the
 $q$-expansion of $F_k$ at every cusp are rational numbers$)$, and $F_k=0$ for all but finitely many $k$.

Let $\alpha\in (\mathbb{A}_K^f)^{\times}$. We denote by $\alpha_p$ the
 $p$-th component of $\alpha$. Set 
\[
 S(\alpha)=\sum_{k\geq 0} \mathcal{N}\alpha_p^{-k} F_{k,\alpha}.
\]

If $S(\alpha)$ has all coefficients in $p^j\Z_p$ $(j\in \Z)$ for one
 $\alpha \in (\mathbb{A}^f_K)^{\times}$, 
 $S(\alpha)$ has all coefficients in $p^j\Z_p$ for {\em every} $\alpha
 \in (\mathbb{A}_K^f)^{\times}$.
\end{thm}

\begin{proof}
See \cite{De-Ri} Theorem (0.2).
\end{proof}

\begin{cor}[{\cite{De-Ri}} Corollary (0.3)] \label{cor:DRpri}
Let $S(\alpha)$ as in Theorem $\ref{thm:DRpri}$.
Suppose that there exists $\alpha\in (\mathbb{A}_K^f)^{\times}$ such that 
all non-constant coefficients of $S(\alpha)$ are contained in $p^j \Z_{(p)}$
 $(j\in \Z)$. Then for arbitrary two distinct elements $\beta, \beta'
 \in (\mathbb{A}_K^f)^{\times}$, the difference between 
the constant terms of the $q$-expansions of $S(\beta)$ and $S(\beta')$
 is an element of $p^j\Z_{(p)}$.
\end{cor}

\begin{proof}
Let $c(0,\alpha)$ be the constant term of $S(\alpha)$. We may interpret
 $c(0,\alpha)$ as an element of $M_0(\Gamma_{00}(\ideal{f}),\Q)$.

Set $S(\beta)_{c(0,\alpha)}=S(\beta)-c(0,\alpha)$. Then
 $S(\alpha)_{c(0,\alpha)}$ has all coefficients in
 $p^j\Z_{(p)}$, so does $S(\beta)_{c(0,\alpha)}$ by Theorem \ref{thm:DRpri}. Especially 
the constant term of $S(\beta)_{c(0,\alpha)}$ is also an element of
 $p^j\Z_{(p)}$, but it is no other than $c(0,\beta)-c(0,\alpha)$.
\end{proof}

%
\subsection{Proof of the sufficient conditions}
%

Now we prove the sufficient conditions in Proposition
\ref{prop:suff}. We only prove the conditions $(1)'$ and $(3)'$. 
One can prove the conditions $\tilder{(2)}'$ (resp. $(2)'$) in the same
manner as $(1)'$ (resp. $(3)'$).

\medskip

\noindent \textbf{Condition} $(1)'$. Let $k$ be a positive even integer and
      $\varepsilon$ a locally constant $\Z_{(p)}$-valued function on
      $W_1$ fixed by the action of $G/U_1$. Let
      $G_{k,\varepsilon}$ (resp.\ $G_{pk, \varepsilon \circ \varphi}$) be
      the Hilbert-Eisenstein series of weight $k$ (resp.\ weight $pk$)
      attached to $\varepsilon$ (resp.\ $\varepsilon \circ \varphi$) (Theorem-Definition \ref{thmdef:eisen}).

      The natural inclusion $F\rightarrow F_{U_1}$ induces the
      restriction of the Hilbert-Eisenstein series $\mathrm{res
      }\, G_{k,\varepsilon}$ on $\ideal{h}_F$. It is easy to see that 
      $\mathrm{res}\, G_{k,\varepsilon}$ is a Hilbert modular form of weight $pk$ and 
      its standard $q$-expansion is given by (\cite{R-W3} Lemma 7)
      \begin{align*}
       \mathrm{res}\, G_{k,\varepsilon} &=2^{-[F_{U_1}:\Q]}
        \zeta_{F_{V_1}/F_{U_1}}(1-k, \varepsilon) \\ &+ \sum_{\mu \in
      \mathcal{O}_F, \mu \gg 0} 
      \left( \sum_{(\ideal{b}, \nu) \in P_1^{\mu}}\varepsilon
       (\ideal{b})
       \kappa(\ideal{b})^{k-1}  \right)  q_F^{\mu}
      \end{align*}
      where $q_F^{\mu}=\exp(2\pi \sqrt{-1}\Tr_{F/\Q}(\mu \tau))$ and 
\begin{align*}
 P_1^{\mu}=\{ (\ideal{b}, \nu)  \mid \nu \in \ideal{b} \subseteq 
      \mathcal{O}_{F_{U_1}}, &\ideal{b} \text{ is prime to $\Sigma$},  \\
 &  \nu \gg 0, \Tr_{F_{U_1}/F}(\nu)=\mu \}.
\end{align*}
      
      For $\lambda \in \mathcal{O}_F$, we may construct the Hecke
      operator $U_{\lambda}$ associated to $\lambda$ (See \cite{R-W3}
      Lemma 6), which ``shifts the coefficients of $q$-expansion by
      $\lambda$.'' Therefore we have 
       \begin{align*}
       (\mathrm{res}\, G_{k,\varepsilon})|_{pk}U_p &=2^{-[F_{U_1}:\Q]}
        \zeta_{F_{V_1}/F_{U_1}}(1-k, \varepsilon) \\ &+ \sum_{\mu \in
      \mathcal{O}_F, \mu \gg 0} 
      \left( \sum_{(\ideal{b}, \nu) \in P_1^{p\mu}}\varepsilon 
	(\ideal{b})
      \kappa(\ideal{b})^{k-1}  \right)  q_F^{\mu}
      \end{align*}

      On the other hand, the standard $q$-expansion of $G_{pk,
      \varepsilon \circ \varphi}$ is 
      \begin{align*}
       & G_{pk, \varepsilon \circ \varphi} = 2^{-[F:\Q]}
      \zeta_{F_{V_0}/F_{U_0}} (1-pk, \varepsilon \circ \varphi) \\
       &+
      \sum_{\mu \in \mathcal{O}_F, \mu \gg 0} \left( \sum_{\mu \in
      \ideal{a} \subseteq \mathcal{O}_F, \text{prime to }\Sigma}
      \varepsilon \circ \varphi \left( \left(
      \frac{F_{V_0}/F_{U_0}}{\ideal{a}}\right)\right)
      \kappa(\ideal{a})^{pk-1} \right) q_F^{\mu}.
      \end{align*}
      and note that 
      \[
	\varphi \left( \left( \frac{F_{V_1}/F_{U_1}}{\ideal{a}}
      \right)\right) =\Ver \left( \left(
      \frac{F_{V_1}/F_{U_1}}{\ideal{a}} \right)\right)	=\left( \frac{F_{V_0}/F_{U_0}}{\ideal{a}\mathcal{O}_{F_{U_1}}}\right) 
		\]
by the commutativity of the Artin symbol and the Verlagerung
homomorphism (see Lemma \ref{lem:verfrob}).

      Set $S=(\mathrm{res}\,
      G_{k,\varepsilon})|_{pk}U_p-G_{pk,\varepsilon \circ \varphi}$,
      then the $\mu$-th coefficient of $S$ is 
\[
 \sum_{(\ideal{b},\nu)\in P_1^{p\mu}} \varepsilon
      (\ideal{b})\kappa(\ideal{b})^{k-1}  -\sum_{\mu \in \ideal{a}
      \subseteq \mathcal{O}_F} \varepsilon
      (\ideal{a}\mathcal{O}_{F_{U_1}}) \kappa(\ideal{a})^{pk-1}
\]

Let $(G/U_1)_{(\ideal{b}, \nu)}$ be the isotropic subgroup of $G/U_1$ at
      $(\ideal{b}, \nu) \in P_1^{p\mu}$.

\begin{enumerate}[({Case}-1)]
\item $(G/U_1)_{(\ideal{b}, \nu)}=\{ \id \}$. Note that
\begin{align*}
\varepsilon \left( \left(
 \frac{F_{V_1}/F_{U_1}}{\ideal{b}^{\sigma}}\right)\right) &= 
\varepsilon \left( \sigma^{-1} \left( \frac{F_{V_1}/F_{U_1}}{\ideal{b}}
 \right)\sigma \right) \\ 
 &= \varepsilon^{\sigma} \left( \left(
 \frac{F_{V_1}/F_{U_1}}{\ideal{b}}\right)\right) \\
 &= \varepsilon \left( \left(
 \frac{F_{V_1}/F_{U_1}}{\ideal{b}}\right)\right)
\end{align*}
for every $\sigma \in G/U_1$. Similarly we have
      $\kappa(\ideal{b}^{\sigma})^{k-1}=\kappa(\ideal{b})^{k-1}$. Therefore the sum of $(\ideal{b},
      \nu)$-orbit is given by 
\[
 \sum_{\sigma \in G/U_1} \varepsilon  (\ideal{b}^{\sigma})
      \kappa(\ideal{b}^{\sigma})^{k-1} =p\varepsilon (\ideal{b})
      \kappa(\ideal{b})^{k-1} \in p\Z_{(p)}.
\]

\item $(G/U_1)_{(\ideal{b},\nu)}=G/U_1$. In this case $G/U_1$ fixes
      $(\ideal{b}, \nu)$, therefore $\nu \in F$ and
      $\ideal{b}=\ideal{a}\mathcal{O}_{F_{U_1}}$ for unique integral
      ideal $\ideal{a}$ of $F$ prime to $\Sigma$. Since
      $\Tr_{F_{U_1}/F}(\nu)=p\mu$, we have $\nu=\mu$. Therefore
\begin{align*}
 \varepsilon (\ideal{b}) \kappa(\ideal{b})^{k-1} &=\varepsilon
      (\ideal{a}\mathcal{O}_{F_{U_1}})
      \kappa(\ideal{a}\mathcal{O}_{F_{U_1}})^{k-1} =\varepsilon
      (\ideal{a} \mathcal{O}_{F_{U_1}}) \kappa(\ideal{a})^{p(k-1)} \\
      & \equiv 
      \varepsilon
      (\ideal{a} \mathcal{O}_{F_{U_1}}) \kappa(\ideal{a})^{pk-1} \quad
      \mod p
\end{align*}
and so the $(\ideal{a}\mathcal{O}_F, \mu)$-term vanishes modulo $p$.
\end{enumerate}

\medskip
Therefore $S$ has all the non-constant coefficients in $p\Z_{(p)}$, and we
      can apply Deligne-Ribet's principle (Corollary \ref{cor:DRpri}) to
      $S=S(1)$. There exists a finite id\`ele $\gamma$ such that $\left(
      \frac{F_{V_0}/F_{U_0}}{((\gamma))\gamma^{-1}}\right)=g$ (see \cite{De-Ri} (2.23)).
      Then by Corollary \ref{cor:DRpri} $S(1)-S(\gamma)$ has its
      constant term in $p\Z_{(p)}$. By easy calculation, the constant
      term of $E-E(\gamma)$ is
\begin{align*}
2^{-pr} &\zeta_{F_{V_1}/F_{U_1}}(1-k,\varepsilon)-2^{-r}\zeta_{F_{V_0}/F_{U_0}}
 (1-pk, \varepsilon \circ \varphi)-\mathcal{N}(\gamma_p)^{-pk}
 \mathcal{N}((\gamma))^{pk}   \\
& \qquad \times \{ 2^{-pr}
 \zeta_{F_{V_1}/F_{U_1}} (1-k,\varepsilon_{\frob{g}})-2^{-r}
 \zeta_{F_{V_0}/F_{U_0}} (1-pk,(\varepsilon\circ \varphi)_g) \} \\
= &2^{-pr} \{ \zeta_{F_{V_1}/F_{U_1}} (1-k,\varepsilon)-\kappa(\varphi(g))^{k}
 \zeta_{F_{V_1}/F_{U_1}} (1-k, \varepsilon_{\frob{g}})\} \\
&- 2^{-r} \{ \zeta_{F_{V_0}/F_{U_0}}(1-pk, \varepsilon\circ \varphi)
 -\kappa(g)^{pk} \zeta_{F_{V_0}/F_{U_0}}(1-pk, (\varepsilon \circ
 \varphi)_g)\} \\
= & 2^{-pr} \Delta_1^{\frob{g}} (1-k, \varepsilon)-2^{-r} \Delta_0^g
 (1-pk, \varepsilon \circ \varphi) \\
\equiv & 2^{-r} \{ \Delta_1^{\frob{g}} (1-k, \varepsilon) -\Delta_0^g
 (1-pk, \varepsilon \circ \varphi) \} \quad \mod p.
\end{align*}
Here we set $r=[F:\Q]$ and use
      $\mathcal{N}((\gamma))=\mathcal{N}(\gamma_p) \kappa(g)$ (\cite{R-W3}
      Lemma 9), $\kappa(\varphi(g))=\kappa(g)^p$. For precise calculation, see \cite{R-W3}, the
      proof of the Theorem.

\bigskip
\noindent \textbf{Condition} $(3)'$. Let $j$ be a sufficiently large integer and
      $y\in W_3/\Gamma^{p^j}$ is a coset which is not contained in
      $\Gamma$. 

      Let $S=G_{k,\delta^{(y)}}$ be the Hilbert-Eisenstein series of
      weight $k$ attached to $\delta^{(y)}$, then the standard
      $q$-expansion of $S$ is given by 
\begin{align*}
S &=2^{-[F_{U_3}:\Q]}  \zeta_{F_{V_3}/F_{U_3}} (1-k, \delta^{(y)}) \\
 & + \sum_{\nu \in \mathcal{O}_{F_{U_3}}, \mu \gg 0} \left( \sum_{\nu
 \in \ideal{b} \subseteq \mathcal{O}_{F_{U_3}}, \text{prime to }\Sigma}
 \delta^{(y)}( \ideal{b}) \kappa(\ideal{b})^{k-1} \right) q_{F_{U_3}}^{\nu}
\end{align*}
where $q_{F_{U_3}}^{\nu}=\exp(2\pi {\sqrt{-1}} \Tr_{F_{U_3}/\Q}(\nu \tau))$.

\begin{enumerate}[({Case}-1)]
\item $(G/U_3)_{(\ideal{b}, \nu)}=\{ \id \}$. In this case, we can
      easily calculate the sum of $(\ideal{b}, \nu)$-orbit :
\[
 \sum_{\sigma\in G/U_3} \delta^{(y)}(\ideal{b}^{\sigma})
      \kappa(\ideal{b}^{\sigma})^{k-1} =p^{m_y} \sum_{\sigma\in
      (G/U_3)/(G/U_3)_y}\delta^{(y^{\sigma^{-1}})}(\ideal{b})
      \kappa(\ideal{b})^{k-1}.
\] 

\item $(G/U_3)_{(\ideal{b}, \nu)}\neq \{ \id\}$. Let $F_{(\ideal{b},
      \nu)}$ be the fixed field of $(G/U_3)_{(\ideal{b}, \nu)}$. Then by
      the same argument as the proof of the condition $(1)'$ (Case-2), there exists $\mu \in F_{(\ideal{b}, \nu)}$ and
      $\ideal{a}\subseteq \mathcal{O}_{F_{(\ideal{b},\nu)}}$ uniquely
      such that $(\ideal{b}, \nu)=(\ideal{a}\mathcal{O}_{F_{U_3}},
      \mu)$. For such $(\ideal{a}, \mu)$, 
\begin{align*}
 \delta^{(y)} (\ideal{b}) &=\delta^{(y)} \left( \left(
 \frac{F_{U_3}/F_{V_3}}{\ideal{a}\mathcal{O}_{F_{U_3}}}\right) \right) \\
&= \delta^{(y)} \circ \Ver \left(\left( \frac{F_{V_3}/F_{(\ideal{b},
 \nu)}}{\ideal{a}} \right)\right) =0
\end{align*}
because $\mathrm{Im}(\Ver)=\mathrm{Im}(\varphi) \subseteq \Gamma$.
\end{enumerate}

\medskip
Therefore $S$ has all the non-constant coefficients in
      $p^{m_y}\Z_{(p)}$. Take a finite id\`ele $\gamma'$ such that $\left(
      \frac{F_{V_3}/F_{U_3}}{(\gamma'){\gamma'}^{-1}}\right)=w$. Then by
      Corollary \ref{cor:DRpri}, the constant term of $S-S(\gamma')$ is
      also in $p^{m_y}\Z_{(p)}$, and it is equal to 
\[
  2^{-p^3r} \Delta_3^{w} (1-k, \delta^{(y)}).
\] 	

Thus we have finished the proof of Proposition \ref{prop:suff}.

%
%
\section{Proof of the main theorem}
%
%

Unfortunately, we cannot conclude that $(\xi_i)_i$ is contained in
$\Psi_S$ by the congruences obtained in the previous section, so we may
not apply Burns' technique (Theorem \ref{thm:ak-pri}) directly to
$(\xi_i)_i$. 

In this section, we modify the proof of Theorem \ref{thm:ak-pri} and prove our main theorem (Theorem \ref{thm:maintheorem}) using
an inductive technique.

%
\subsection{Kato's $p$-adic zeta function for $F_N/F$}
%

Let 
\begin{equation*}
N=\begin{pmatrix} 1 & 0 & 0 & \Fp \\ 0 & 1 & 0 & \Fp \\ 0 & 0 & 1 & \Fp \\
 0 & 0 & 0 & 1 \end{pmatrix} \times \{ 1 \}
\end{equation*}
be a closed normal subgroup of $G$ and set $\line{G}=G/N, \line{U}_i=U_i/N$ and $\line{V}_i=V_i\cdot N/N$.

Then we have the splitting exact sequence
\begin{equation*}
\begin{CD}
1 @>>> N @>>> G @>\pi>> \line{G} @>>> 1.
\end{CD}
\end{equation*}
Let 
\begin{equation*}
s \colon \line{G} \rightarrow G ; \left( 
\begin{pmatrix} 1 & a & d\\  0 & 1 & b \\ 0 & 0 & 1 \end{pmatrix}, t^z
				      \right) \mapsto
\left( \begin{pmatrix} 1 & a & d & 0 \\ 0 & 1 & b & 0 \\ 
0 & 0 & 1 & 0 \\ 0 & 0 & 0 & 1 \end{pmatrix}, t^z \right)
\end{equation*}
be a section of $\pi$.

The $p$-adic Lie group
\begin{equation*}
\line{G}=\begin{pmatrix} 1 & \Fp & \Fp \\ 0 & 1 & \Fp \\ 0 & 0 & 1 \end{pmatrix} \times \Gamma
\end{equation*}
is a group ``of Heisenberg type,'' for which Kazuya Kato has already
proven the existence of the $p$-adic zeta function in \cite{Kato1}.

\begin{thm}[Kato] \label{thm:kato}
The $p$-adic zeta function $\line{\xi}$ for $F_{N}/F$ exists uniquely and it
 satisfies the Iwasawa main conjecture $($Conjecture $\ref{conj:iwasawa}$ $(2))$.
\end{thm}

\begin{proof}[Sketch of the proof]
This theorem is the special case of \cite{Kato1}, Theorem 4.1.

First, he constructed the theta map (and its localized version)
\begin{align*}
\line{\theta} &\colon K_1(\iw{\line{G}}) \rightarrow \line{\Psi} \subseteq
 \iw{\line{U_0}/\line{V_0}}^{\times} \times
 \iw{\line{U_1}/\line{V_1}}^{\times} \\
\line{\theta}_S &\colon K_1(\iw{\line{G}}_S) \rightarrow \line{\Psi}_S \subseteq
 \iw{\line{U_0}/\line{V_0}}_S^{\times} \times
 \iw{\line{U_1}/\line{V_1}}_S^{\times} 
\end{align*}
where $\line{\Psi}$ (resp.\ $\line{\Psi}_S$) was defined to be the subgroup consisting of
 all elements
 $(\line{\eta}_0, \line{\eta}_1)$ which
 satisfied the norm relation
\begin{equation*}
\Nr_{\iw{\line{U}_0/\line{V}_0}/\iw{\line{U}_1/\line{V}_1}}
 \line{\eta}_0 \equiv \line{\eta}_1 \quad (\text{resp.\ } \Nr_{\iw{\line{U}_0/\line{V}_0}_S/\iw{\line{U}_1/\line{V}_1}_S}
 \line{\eta}_0 \equiv \line{\eta}_1)
\end{equation*}
and the congruence 
\begin{equation*}
\line{\eta}_1 \equiv \frob{\line{\eta}_0} \quad \mod{\line{I_1}} \qquad
 (\text{resp.\ } \mod{\line{I_{S,1}}}).
\end{equation*}  

In this case, we could show the congruence 
\begin{equation*}
\line{\xi}_1 \equiv \frob{\line{\xi}_0} \quad \mod{\line{I_{S,1}}} 
\end{equation*}
directly in the same manner as in \S 7. Hence, by using
 Burns' technique (Theorem \ref{thm:ak-pri}), we might show the
 existence (and uniqueness) of the $p$-adic zeta function
 $\line{\xi}$ for $F_N/F$. 
\end{proof}

Let $C=C_{F^{\infty}/F}$ be the complex defined in Definition
\ref{def:selmer}. Since we always assume the condition $(\ast)$ in
Proposition \ref{prop:selmer}, $C$ is contained in $K_0(\iw{G},
\iw{G}_S)$. Let $[\line{C}]$ be its image in $K_0(\iw{\line{G}},
\iw{\line{G}}_S)$. Then Kato's $p$-adic zeta function $\line{\xi}$
satisfies the main conjecture $\del (\line{\xi}) =-[\line{C}]$.

\begin{prop} \label{prop:red}
There exists a characteristic element $f\in K_1(\iw{G}_S)$ for
 $F^{\infty}/F$ whose image in 
$K_1(\iw{\line{G}}_S)$ coincides with $\line{\xi}$ .
\end{prop}

\begin{proof}
In the following, we denote by $\line{x}$
 the image of $x$ in $K_1(\iw{\line{G}}_S)$ for an element $x\in K_1(\iw{G}_S)$. 

Let $f'\in K_1(\iw{G}_S)$ be a characteristic element of $[C]$. Then by using the functoriality of the connecting homomorphism $\del$, 
we have $\del(\line{\xi} \cdot \line{f'}^{-1})=-[\line{C}]+[\line{C}]=0$. By the localization exact sequence (Theorem \ref{thm:locseq}), 
$\line{u}=\line{\xi} \cdot \line{f'}^{-1}$ is the image of an element of $K_1(\iw{\line{G}})$, which we also denote $\line{u}$ by abuse of notation. 
Then the element $f=f's(\line{u})$ satisfies the assertion of the 
proposition where $s$ denotes the homomorphism $K_1(\iw{\line{G}})\rightarrow K_1(\iw{G})$ induced by $s \colon \line{G} \rightarrow G$.
\end{proof}

%
\subsection{Completion of the proof}
%

Let $f\in K_1(\iw{G}_S)$ be a characteristic element for $F^{\infty}/F$ satisfying
Proposition \ref{prop:red}, that is, $\pi(f)=\line{\xi}$. 

Let $f_i=\theta_{S,i}(f) \in \iw{U_i/V_i}_S^{\times}$ and let $u_i=\xi_if_i^{-1}$. Then $\del(u_i)=0$, so we have $u_i \in \iw{U_i/V_i}^{\times}$ by the localization sequence (Theorem \ref{thm:locseq}).

Then it is sufficient to show that $(u_i)_i$ is contained  in $\Psi$: if
$(u_i)_i$ is contained in $\Psi$, there
exists $u\in K_1(\iw{G})$ such that $\theta_i(u)=u_i$ by the
surjectivity of the theta map (Proposition \ref{prop:theta}). One can
easily check that 
$\xi=uf$ satisfies the interpolating properties of the $p$-adic zeta function
for $F^{\infty}/F$ (Definition \ref{def:zeta}) and also satisfies
$\del(\xi)=-[C]$. Namely, $\xi$ is the desired $p$-adic zeta function.

Note that it was easy to show that $(u_i)_i \in \Psi$ in the proof of
Theorem \ref{thm:ak-pri} because of the assumption $(\xi_i)_i \in \Psi_S$. 

\begin{prop} \label{prop:ucong}
$u_i$'s satisfy the norm relations in Definition $\ref{def:psi}$, and
 satisfy following congruences$\colon$
\begin{align*}
u_1 & \equiv \frob{u_0} & \mod{I_1}, \\
\tilder{u_2} & \equiv \frob{u_0} & \mod{\tilder{I_2}}, \\
u_2 & \equiv d_2 & \mod{I_2''}, \\
u_3 & \equiv d_3 & \mod{I_3},
\end{align*}
where $d_2$ and $d_3$ are certain elements of $\iw{\Gamma}$.
\end{prop}

\begin{proof}
$f_i$'s and $\xi_i$'s satisfy the norm relations, so do $u_i$'s.

Since $(f_i)_i$ is contained in $\Psi_S$, $f_i$'s satisfy the
 congruences in Proposition \ref{prop:psis}. On the other hand, 
$\xi_i$'s satisfy the congruences in Proposition \ref{prop:congzeta}. Hence,
 we can easily show that $u_i$'s satisfy the desired congruences.\footnote{Note that
 we may replace $c_2$ and $c_3$ in Proposition \ref{prop:congzeta} by invertible elements.}
\end{proof}

\begin{lem} \label{lem:ker}
$\displaystyle (u_i)_i \in \Ker \left( \pi^{\times}\colon \prod_i \iw{U_i/V_i}^{\times} 
\longrightarrow \prod_i \iw{\line{U_i}/\line{V_i}}^{\times} \right)$.
\end{lem}

\begin{proof}
This lemma follows from the construction of $f$ and the commutativity of $\pi$ and
 norm maps of $K$-groups (Proposition \ref{prop:commutativity}).
\end{proof}

\begin{lem} \label{lem:sp}
$s\circ \pi (I_i) \subseteq I_i$.
\end{lem}

\begin{proof}
Just simple calculation.
\end{proof}

\begin{lem} \label{lem:fker}
For each $i$, $\frob{u_i}=1$.
\end{lem}

\begin{proof}
The Frobenius homomorphism $\varphi\colon \iw{U_i/V_i} \rightarrow \iw{\Gamma}$ factors
as 
\begin{equation*}
\begin{CD}
\iw{U_i/V_i} @>\varphi>> \iw{\Gamma} \\
@V{\pi}VV                @A{\|}AsA \\
\iw{\line{U_i}/\line{V_i}} @>>\line{\varphi}> \iw{\Gamma}
\end{CD}
\end{equation*}
where $\line{\varphi}\colon \iw{\line{U_i}/\line{V_i}}\longrightarrow \iw{\Gamma}$ is
 the Frobenius homomorphism ($\mod{N}$). 

Then this lemma holds since $(u_i)_i \in \Ker(\pi^{\times})$ by Lemma \ref{lem:ker}.
\end{proof}

\begin{proof}[Proof of Theorem $\ref{thm:maintheorem}$]
By operating $s \circ \pi$ to the congruences of Proposition \ref{prop:ucong}, we have 
\begin{align*}
s\circ \pi(u_2) & \equiv d_2 & \mod s\circ \pi(I_2''), \\
s\circ \pi(u_3) & \equiv d_3 & \mod s\circ \pi(I_3).
\end{align*} 
By Lemma \ref{lem:ker}, both of $s\circ \pi(u_2)$ and $s \circ \pi(u_3)$ are $1$. Therefore we obtain
\begin{align}
u_2 & \equiv d_2 \equiv 1 = \frob{u_2} & \mod{I_2''}  \label{eq:congu2}, \\
u_3 & \equiv d_3 \equiv 1 = \frob{u_3} & \mod{I_3} \label{eq:congu3},
\end{align}
by Lemma \ref{lem:sp} and \ref{lem:fker}. Hence if we show that 
\begin{align} \label{eq:congI2}
u_2 &\equiv 1 (= \frob{u_2}) & \mod{I_2},
\end{align}
we have $(u_i)_i \in \Psi$, which is the desired result.

Now we show the congruence (\ref{eq:congI2}). We know that $u_2 \in 1+I_2''$
 and $u_3 \in 1+I_3$ by (\ref{eq:congu2}) and (\ref{eq:congu3}). Note
 that the $p$-adic logarithmic homomorphism induces an isomorphism of
 abelian groups
 $1+I_2'' \xrightarrow{\simeq} I_2''$ by the same argument as the proof
 of Lemma \ref{lem:logi}. By the logarithmic isomorphisms (Lemma
 \ref{lem:logi}), we have $\log u_2 \in I_2''$ and $\log u_3\in I_3$.

We may describe $\log u_2$ and $\log u_3$ as $\iw{\Gamma}$-linear
 combinations of generators of $I_2''$ and $I_3$:
\begin{align*}
\log u_2 &= \sum_{b,c\neq 0}  \tilder{\nu}_{bc}^{(3)} \beta^b \gamma^c h_{\zeta}+ \sum_{b,f} p \tilder{\nu}^{(4)}_{bf} \beta^b \zeta^f,  \\
\log u_3 &= \sum_f p^3 \sigma_f^{(1)} \zeta^f + \sum_{e\neq 0} p^2 \sigma_e^{(2)} \varepsilon^e h_{\zeta}+\sum_{c\neq 0} p \sigma_c^{(3)} \gamma^c h_{\varepsilon} h_{\zeta}.
\end{align*}

Then we have 
\begin{equation*}
\Tr_{\iw{U_2/V_2}/\iw{U_3/V_2}} \log u_2 = \sum_{c\neq 0} p \tilder{\nu}_{0c}^{(3)} \gamma^c h_{\zeta} +\sum_{f} p^2 \tilder{\nu}^{(4)}_{0f}h_{\zeta}
\end{equation*}
and 
\begin{equation*}
\log u_3 \equiv \sum_f p^2 \left( p\sigma_f^{(1)}+\sum_{e\neq 0} \sigma^{(2)}_e \right) \zeta^f + \sum_{c\neq 0} p^2 \sigma_c^{(3)} \gamma^c h_{\zeta} 
\qquad (\mod {V_2}).
\end{equation*}

By comparing the coefficients, we have
\begin{align*}
\tilder{\nu}^{(4)}_{0f} &=p\sigma^{(1)}_f+ \sum_{e\neq 0} \sigma^{(2)}_e, \\
\tilder{\nu}^{(3)}_{0c} &= p\sigma^{(3)}_c \qquad (c\neq 0).
\end{align*}

Therefore if we set 
\begin{align*}
\nu^{(1)}_f & = \sigma^{(1)}_f , & \nu^{(2)}_c & =
\begin{cases}  \sum_{e\neq 0} \sigma^{(2)}_e  & \text{if $c=0$}, \\
\sigma^{(3)}_c& \text{if $c\neq 0$}, \end{cases}  \\
\nu^{(3)}_{bc} & = \tilder{\nu}^{(3)}_{bc} \quad (b\neq 0, c\neq 0) , &  \nu^{(4)}_{bf} & = \tilder{\nu}^{(4)}_{bf} \quad (b\neq 0), 
\end{align*}
we have 
\begin{equation*}
\begin{split}
\log u_2 =&  \sum_f p^2 \nu^{(1)}_f \zeta^f + \sum_c p\nu^{(2)}_c \gamma^c h_{\zeta} \\
      &+\sum_{b\neq 0, c\neq 0} \nu^{(3)}_{bc} \beta^b \gamma^c h_{\zeta} +\sum_{b\neq 0,f} p\nu^{(4)}_{bf} \beta^b \zeta^f 
\end{split}
\end{equation*}
which implies $\log u_2 \in I_2$. Hence by the logarithmic isomorphism
 (Proposition \ref{lem:logi}), $u_2$ is contained in $1+I_2$, which implies the congruence (\ref{eq:congI2}).

\end{proof}

\begin{rem}
By the construction of $\xi=uf$, we have $\theta_{S}(\xi)=(\xi_i)_i$. Since $\Psi_S$ contains 
the image of $\theta_S$, we especially obtain the following non-trivial
 congruences among abelian $p$-adic pseudomeasures:
\begin{align*}
\xi_2 & \equiv \frob{\xi_0}^p & \mod{I_2}, \\
\xi_3 & \equiv \frob{\xi_0}^{p^2} & \mod{I_3}. 
\end{align*}
 
It seems to be impossible to show these congruences directly by using
 only the theory of Deligne-Ribet.
\end{rem}


\end{document}